\DocumentMetadata{uncompress}
\documentclass[11pt]{amsart}

 \usepackage{amsfonts,graphics,amsmath,amsthm,amsfonts,amscd,amssymb,amsmath,latexsym,multicol,
 mathrsfs}
\usepackage{epsfig,url}
\usepackage{flafter}
\usepackage{fancyhdr}
\usepackage[dvipsnames]{xcolor}
\usepackage{hyperref}
\hypersetup{
	colorlinks=true,
	linkcolor=purple,
	citecolor=ForestGreen,
	filecolor=cyan,
	urlcolor=teal
}

\usepackage{graphicx}
\usepackage{float}
\usepackage{bm}
\usepackage{enumitem}
\usepackage{multirow}
\usepackage{scalerel,stackengine}
\usepackage{stmaryrd}
\usepackage{datetime}
\usepackage[title]{appendix}
\usepackage{url}
\usepackage{tabulary}
\usepackage{booktabs}
\usepackage{tikz}
\usetikzlibrary{positioning}
\usepackage{verbatim}

\usepackage[utf8]{inputenc}
\usepackage[english]{babel}
\usepackage{csquotes}

\usepackage{comment}

\usepackage[
backend=biber,
style=alphabetic,
sorting=anyt,
doi=false,isbn=false,url=false,
maxalphanames=99,maxbibnames=99,backref=true,
]{biblatex}
\DefineBibliographyStrings{english}{
  backrefpage  = {\hspace{-0.35em}},
  backrefpages = {\hspace{-0.35em}},
}

\DeclareFieldFormat{bracketswithperiod}{\mkbibbrackets{#1}}

\renewbibmacro*{pageref}{%
  \iflistundef{pageref}
    {}
    {\printtext[bracketswithperiod]{%
       \ifnumgreater{\value{pageref}}{1}
         {\bibstring{backrefpages}\ppspace}
         {\bibstring{backrefpage}\ppspace}%
       \printlist[pageref][-\value{listtotal}]{pageref}}%
     }}

 \theoremstyle{plain}
 \newtheorem{theorem}[subsection]{Theorem}
 \newtheorem{lemma}[subsection]{Lemma}
 \newtheorem{corollary}[subsection]{Corollary}

 \theoremstyle{definition}
 \newtheorem{definition}[subsection]{Definition}
 \newtheorem{definitionnotation}[subsection]{Notation}

 \newtheorem{remark}[subsection]{Remark}

 \theoremstyle{remark}

 \theoremstyle{plain} 
\newcommand{\thistheoremname}{}
\newtheorem{genericthm}[subsection]{\thistheoremname}

  \newtheorem*{genericthm*}{\thistheoremname}
\newenvironment{namedthm*}[1]
  {\renewcommand{\thistheoremname}{#1}%
   \begin{genericthm*}}
  {\end{genericthm*}}
 
\newcommand{\spec}{\operatorname{Spec}}

\newcommand{\divisor}{\operatorname{Div}}
\newcommand{\chara}{\operatorname{char}}
\newcommand{\pr}{\operatorname{pr}}
\newcommand{\homo}{\operatorname{Hom}}
\newcommand{\supp}{\operatorname{Supp}}
\newcommand{\red}{\operatorname{red}}

\newcommand{\gp}{\operatorname{gp}}
\newcommand{\et}{\operatorname{\Acute{e}t}}
\newcommand{\Et}{\operatorname{Et}}
\newcommand{\zar}{\operatorname{Zar}}

\newcommand{\rank}{\operatorname{rank}}
\newcommand{\fs}{\operatorname{fs}}
\newcommand{\inte}{\operatorname{int}}

\newcommand{\cone}{\operatorname{Cone}}
\newcommand{\CDiv}{\operatorname{CDiv}}

\newcommand{\logzar}{\operatorname{\mathbf{Log}_{\zar}(\mathbb K)}}
\newcommand{\loget}{\operatorname{\mathbf{Log}_{\et}(\mathbb K)}}
\newcommand{\logsch}{\operatorname{\mathbf{LogSch}}}
\newcommand{\Modd}{\operatorname{\mathbf{Mod.}}}
\newcommand{\modd}{\operatorname{\mathbf{Mod}}}
\newcommand{\Log}{\operatorname{\mathbf{Log}}}
\newcommand{\K}{\operatorname{K}}

\newcommand{\mni}{\medskip\noindent}

\newcommand{\mbb}{\mathbb}

\newcommand{\NN}{\mbb{N}}
\newcommand{\ZZ}{\mbb{Z}}

\newcommand{\RR}{\mbb{R}}
\newcommand{\AAA}{\mbb{A}}

\newcommand{\mc}{\mathcal}

\newcommand{\mcF}{\mc{F}}
\newcommand{\mcG}{\mc{G}}

\newcommand{\mcM}{\mc{M}}
\newcommand{\mcN}{\mc{N}}
\newcommand{\mcO}{\mc{O}}

\newcommand{\mcZ}{\mc{Z}}

\newcommand{\mf}{\mathfrak}
\newcommand{\mb}{\mathbf}
\newcommand{\wht}{\widehat}
\newcommand{\whtOO}{\wht{\mathcal{O}}}
\newcommand{\wh}{\widehat}
\newcommand{\wt}{\widetilde}
\newcommand{\ol}{\overline}
\newcommand{\ul}{\underline}

\usepackage{graphicx}
\usepackage{xcolor}
\usepackage{float}
\usepackage{amssymb}
\usepackage{amsmath}
\usepackage{mathrsfs}
\usepackage{bm}
\usepackage[all,cmtip]{xy}
\usepackage{enumitem}
\usepackage[utf8]{inputenc}
\usepackage[english]{babel}
\usepackage{multirow}
\usepackage{mathtools}
\usepackage{scalerel,stackengine}
\usepackage{stmaryrd}
\usepackage{datetime}
\usepackage[title]{appendix}
\usepackage{url}
\usepackage{tabulary}
\usepackage{booktabs}
\usepackage{tikz}
\usetikzlibrary{positioning}
\usepackage{verbatim}
\usepackage{braket}
\usepackage[T1]{fontenc}
\usepackage[many]{tcolorbox}

\DeclareFieldFormat{postnote}{#1}
\DeclareFieldFormat{multipostnote}{#1}
\DeclareMathAlphabet{\mathpzc}{OT1}{pzc}{m}{it}
\DeclareSymbolFont{yhlargesymbols}{OMX}{yhex}{m}{n} 
\DeclareMathAccent{\yhwidehat}{\mathord}{yhlargesymbols}{"62}

\makeatletter
\newcommand*{\da@rightarrow}{\mathchar"0\hexnumber@\symAMSa 4B }
\newcommand*{\da@leftarrow}{\mathchar"0\hexnumber@\symAMSa 4C }
\newcommand*{\xdashrightarrow}[2][]{%
  \mathrel{%
    \mathpalette{\da@xarrow{#1}{#2}{}\da@rightarrow{\,}{}}{}%
  }%
}
\newcommand{\xdashleftarrow}[2][]{%
  \mathrel{%
    \mathpalette{\da@xarrow{#1}{#2}\da@leftarrow{}{}{\,}}{}%
  }%
}
\newcommand*{\da@xarrow}[7]{%
  \sbox0{$\ifx#7\scriptstyle\scriptscriptstyle\else\scriptstyle\fi#5#1#6\m@th$}%
  \sbox2{$\ifx#7\scriptstyle\scriptscriptstyle\else\scriptstyle\fi#5#2#6\m@th$}%
  \sbox4{$#7\dabar@\m@th$}%
  \dimen@=\wd0 %
  \ifdim\wd2 >\dimen@
    \dimen@=\wd2 %
  \fi
  \count@=2 %
  \def\da@bars{\dabar@\dabar@}%
  \@whiledim\count@\wd4<\dimen@\do{%
    \advance\count@\@ne
    \expandafter\def\expandafter\da@bars\expandafter{%
      \da@bars
      \dabar@ 
    }%
  }%
  \mathrel{#3}%
  \mathrel{%
    \mathop{\da@bars}\limits
    \ifx\\#1\\%
    \else
      _{\copy0}%
    \fi
    \ifx\\#2\\%
    \else
      ^{\copy2}%
    \fi
  }%
  \mathrel{#4}%
}
\makeatother

\title{\large S\MakeLowercase{aturated base change of toroidal morphisms}}
\author{\large S\MakeLowercase{antai} Q\MakeLowercase{u}}
\date{\today}

\begin{document}

\begin{abstract}
    We show that saturated base change of a dominant toroidal morphism is also toroidal.
    For completeness, we give full details on equivalence between definitions
    regarding toroidal embeddings and toroidal morphisms in literature.
    Moreover, we show in detail the equivalence between toroidal embeddings and
    logarithmically regular log-varieties, respectively,
    toroidal morphisms and logarithmically smooth morphisms.  
\end{abstract}

\maketitle

\tableofcontents

\addtocontents{toc}{\protect\setcounter{tocdepth}{1}}


\section{Introduction}

We work over an algebraically closed field $\mbb K$ of characteristic zero unless stated otherwise.

\medskip

Toroidal varieties are first introduced in \cite{KKMB} to establish the semi-stable reduction theorem.
In the recent development of birational geometry,
techniques from toroidal geometry play an important role in many profound results.
For example, Caucher Birkar applies toroidal techniques 
in the proof of Borisov-Alexeev-Borisov conjecture in \cite{B-BAB};
more recently, \cite{birkar2023singularities} uses the toroidalisation method to complete 
the proof of Shokurov-M$^{\rm c}$Kernan conjecture on singularities on Fano type fibrations.

On the other hand, it is well-known to experts in logarithmic geometry 
that the study of toroidal varieties is essentially equivalent to the study of logarithmically 
smooth log-schemes; cf. \cite{Kato-toric, Ogus-log-geo}.
However, as techniques from logarithmic geometry are not widely familiar to birational geometors,
even many classical results in logarithmic geometry have not been applied to the toroidal setting,
and more recent advances in logarithmic geometry,
such as \cite{ATW-principalization, ATW20, ATW25}, remain largely unexplored in the toroidal context.

The connection between logarithmic geometry and toroidal geometry has been outlined 
in many literature, for example, \cite{Kato89, Kato-toric, denef2013remarkstoroidalmorphisms}.
Based on the solid literature in logarithmic geometry, the goals of this note are to
\begin{itemize}
     \item demonstrate that \emph{saturated base change} of a dominant toroidal morphism is also toroidal (by applying results in logarithmic geometry), i.e., Theorem~\ref{sat-base-change-intro},
     \item show that different definitions of toroidal embeddings and toroidal morphisms in literature are all equivalent, and
    \item provide complete proofs on equivalence between toroidal embeddings and logarithmically regular log-varieties, as well as, toroidal morphisms and logarithmically smooth morphisms.
\end{itemize}
Although these results are all well-known to experts in logarithmic geometry,
we hope this note can provide sufficient details to help the readers unfamiliar with the subject 
translate techniques from logarithmic geometry into applications of toroidal geometry,
particularly in contexts related to birational geometry.

\mni
{\textbf{\sffamily{Saturated base change of toroidal morphisms.}}}
We adopt the definitions of toroidal embeddings and toroidal morphisms
from \cite[\S II.1]{KKMB} and \cite[Definition 1.3]{weak:s-stable:reduction},
which characterise these notions as objects formal-locally isomorphic to toric objects.
By a \emph{strict toroidal embedding}, we mean a \emph{toroidal embedding without self-intersection}
as \cite[page 57]{KKMB}.
We refer readers to \S\ref{toroidal-geo-section} for the precise definitions.
The primary objective of this paper is to show the following result 
on saturated base change of dominant toroidal morphisms.

\begin{theorem}\label{sat-base-change-intro}
	Let $(U_X\subset X)$, $(U_Y\subset Y)$ and $(U_Z\subset Z)$ be strict toroidal embeddings.  Let 
    \[ f\colon (U_Y\subset Y)\to (U_Z\subset Z)\,\text{ and }\,g\colon (U_X\subset X)\to (U_Z\subset Z) \]
	be morphisms of embeddings, that is, $f,g$ are morphisms of varieties such that 
    $f(U_Y)\subseteq U_Z$ and $g(U_X)\subseteq U_Z$.  Assume that 
	$g$ is a dominant toroidal morphism of toroidal embeddings.
    
	Denote by $W$ the normalisation of an irreducible
	component of $Y\times_Z X$ that dominates $Y$, and by $p\colon W\to Y$ and $q\colon W\to X$
	the induced projection morphisms respectively.
    \[\xymatrix{
	(U_W\subset W)\ar[r]^-q\ar[d]^p & (U_X\subset X)\ar[d]^g \\
	(U_Y\subset Y)\ar[r]^-f & (U_Z\subset Z)
	}\]
	Let 
    \[ U_W \coloneqq p^{-1}(U_Y)\cap q^{-1}(U_X). \]
    Then $U_W$ is a nonempty open subset of $W$, and $(U_W\subset W)$ is also a strict toroidal embedding.
    Moreover, the induced morphism of embeddings
    \[ p\colon (U_W\subset W)\to (U_Y\subset Y) \]
    is a dominant toroidal morphism of toroidal embeddings.
\end{theorem}

Keep the notation in Theorem~\ref{sat-base-change-intro}.
If $Y\times_Z X$ has a unique irreducible component dominating $Y$, we call this irreducible component
the \emph{main component} of $Y\times_Z X$.
In this case, let $W$ be the normalisation of the main component of $Y\times_Z X$.
Then we call the morphism of embeddings $p\colon (U_W\subset W)\to (U_Y\subset Y)$ the
\emph{saturated base change} of $g\colon (U_X\subset X)\to (U_Z\subset Z)$ with respect to 
$f\colon (U_Y\subset Y)\to (U_Z\subset Z)$.  
Thus, Theorem~\ref{sat-base-change-intro} says that the saturated base change 
of a dominant toroidal morphism between strict toroidal embeddings is also toroidal.
The counterpart to Theorem~\ref{sat-base-change-intro} for general toroidal embeddings,
which are not necessarily strict, is given in Theorem~\ref{sat-base-change}.


\mni
{\textbf{\sffamily{Equivalence of definitions of toroidal embeddings and morphisms.}}}
The definitions of toroidal embeddings and toroidal morphisms are not quite consistent across the literature.

\mni
\textbf{\sffamily{(I) The formal-local definitions.}} 
Toroidal embeddings are originally defined as 
varieties that are formal-locally isomorphic to toric varieties on \cite[page 54]{KKMB}.
More specifically, let $X$ be a variety (over $\mbb K$) and $U\subseteq X$ be an open subset of $X$.
We say that the embedding $(U\subset X)$ is \emph{toroidal} if for every closed point $x\in X$,
there exist a normal affine toric variety $W$ and a closed point $w\in W$ such that 
\begin{itemize}
    \item there is an isomorphism of $\mbb K$-algebras of completion of local rings
          \begin{equation}
              \phi\colon \wh{\mcO}_{X,x}\to \wh{\mcO}_{W,w}, \label{toroidal-iso-completion}
          \end{equation}
    \item $I\wh{\mcO}_{X,x}$ is mapped isomorphically to $J\wh{\mcO}_{W,w}$ by $\phi$, where 
                \begin{itemize}
                    \item [$\bullet$] $I$ is the ideal sheaf defining the reduced divisor $X\setminus U$,
                    \item [$\bullet$] $\mbb T_W$ is the torus of the toric variety $W$, and
                    \item [$\bullet$] $J$ is the ideal sheaf defining the reduced divisor $W\setminus \mbb T_W$.
                \end{itemize}
\end{itemize}
Similarly, a toroidal morphism is defined as a morphism of varieties that
is formal-locally isomorphic to a toric morphism of toric varieties; 
see \cite[Definition 1.3]{weak:s-stable:reduction}.
For the precise terminologies about toroidal embeddings and toroidal morphisms, 
we refer readers to \S\ref{toroidal-geo-section}.
In order to distinguish this type of definitions from the others, we call the definitions given above
the \emph{formal-local definitions} for toroidal embeddings and toroidal morphisms.

\mni
\textbf{\sffamily{(II) The \'etale-local definitions.}} 
Toroidal embeddings in \cite{weak-factor, ACP-tropical-curves} are defined
via \'etale neighbourhoods, that is, an embedding $(U\subset X)$ is called toroidal
if for every closed point $x\in X$, there is a diagram
\begin{equation}
    \begin{aligned}\label{toroidal-etale-comm-ngbh}
        \xymatrix{
          & V\ar[rd]^q\ar[ld]_p & \\
        X & & W
        }\end{aligned}
\end{equation}
where 
\begin{itemize}
    \item $p$ is an \'etale neighbourhood of $x$,
    \item $W$ is a normal toric variety with torus $\mbb T_W$,
    \item $q$ is an \'etale morphism, and
    \item $p^{-1}(U) = q^{-1}(\mbb T_W)$.
\end{itemize}
Similarly, a toroidal morphism is defined by compatible \'etale neighbourhoods 
of the toroidal embeddings in \cite[\S 5.3]{ACP-tropical-curves}.
We call these the \emph{\'etale-local definitions} for toroidal embeddings and toroidal morphisms.

\mni
\textbf{\sffamily{(III) Equivalence of definitions.}}
It is evident that the \'etale-local definitions for toroidal embeddings 
and toroidal morphisms imply their formal-local counterparts.
However, the converse direction is not immediate.  
Keep the notation in {\sffamily{(I)}} and {\sffamily{(II)}}.  Then \cite{Artin} shows that \eqref{toroidal-iso-completion}
induces a diagram of \'etale neighbourhoods $p\colon V\to U, q\colon V\to W$
as in \eqref{toroidal-etale-comm-ngbh}, whereas there is no guarantee that $p^{-1}(U)=q^{-1}(\mbb T_W)$.
Thus, the equivalence of these two types of definitions does not follow 
from Artin's algebraic approximation theory in a trivial manner.

\mni
\textbf{\sffamily{\quad(III.a) Equivalence for toroidal embeddings.}}
The implication from the formal-local definition
to the \'etale-local side for toroidal embeddings is sketchily proved in a footnote on \cite[page 195]{KKMB}. 
We provide a complete and detailed proof of this equivalence in \S\ref{toroidal-geo-section}, 
which has not been fully elaborated in literature.
The proof of this equivalence is elementary and does not dependent on logarithmic geometry.

\mni
\textbf{\sffamily{\quad(III.a') Log-regularity and toroidal embeddings.}}
For notions about logarithmic geometry, see \S\ref{log-geo-section}.  
Having established the equivalence between formal-local and \'etale-local definitions
for toroidal embeddings, we see that toroidal embeddings are \'etale-locally toric varieties.
It is very well-known that toric varieties are logarithmically regular (or \emph{log-regular} for short) 
in the category of log-varieties, hence toroidal embeddings are log-regular log-varieties.
Conversely, it is a classical result that log-regular log-varieties 
are formal-locally isomorphic to toric varieties; for example, see \cite[(3.2)]{Kato-toric}.
Thus, we get the following equivalence:
\begin{tcolorbox}[colback=white, ams align]
\label{equi-toroidal-log-reg-intro}
    \text{toroidal embeddings}\Longleftrightarrow \text{log-regular log-varieties}. \tag{\dag}
\end{tcolorbox}
\noindent
We present the proofs of this equivalence in \S\ref{log-reg-toroidal-equi-section};
see Theorem~\ref{toroidal-to-log-reg} and Theorem~\ref{toroidal-regular}.

\mni
\textbf{\sffamily{\quad(III.b) Equivalence for toroidal morphisms.}}
The verification of equivalence between formal-local and \'etale-local definitions
for toroidal morphisms is not  quite straight forward, and
it seems that arguments involving logarithmic geometry are unavoidable.
We show this equivalence along the approach mainly consisting of two steps:
\begin{itemize}
    \item [(A)] dominant formal-local toroidal morphisms are logarithmically smooth, and
    \item [(B)] logarithmically smooth morphisms admit \'etale-toric charts to toric morphisms.
\end{itemize}
Note that for toroidal morphisms, 
the existence of \'etale-toric charts to toric morphisms is exactly the \'etale-local definition; 
see \cite[\S 5.3]{ACP-tropical-curves}.
Then (A) and (B) together show the direction from the formal-local definition
for \emph{dominant} toroidal morphisms to the \'etale-local definition.
\begin{equation}\label{road-map}
    \text{Formal-local defn.}\overset{\text{(A)}}{\Longrightarrow} \text{Log-smoothness}
\overset{\text{(B)}}{\Longrightarrow} \text{\'Etale-toric charts} \Longleftrightarrow \text{\'Etale-local defn.} \tag{$\star$}
\end{equation}
We complete the proofs of (A) and (B) in \S\ref{etale-charts-toroidal-morphisms-section}.
More specifically, (A) is proved in Theorem~\ref{toroidal-to-sm-zar} for Zariski topology,
and in Theorem~\ref{toroidal-to-sm-et} for \'etale topology;
(B) is verified in Theorem~\ref{etale-chart-log-sm-thm}.

We note that the approach \eqref{road-map} has already been outlined 
in \cite[\S 3]{denef2013remarkstoroidalmorphisms}.  
Specifically, (A) is mentioned in \cite[Remark 3.2]{denef2013remarkstoroidalmorphisms},
and \cite[Proposition 3.3]{denef2013remarkstoroidalmorphisms} gives a self-contained proof of (B)
for varieties over arbitrary fields of characteristic zero without using logarithmic geometry.
Indeed, (B) is a standard result in logarithmic geometry; see \cite[Theorem (3.5)]{Kato89},
\cite[\S IV.3.3]{Ogus-log-geo} and \cite[12.3.37]{gabber2018foundationsringtheory}, etc.
The result (A) is taken for granted in \cite[Remark 3.2]{denef2013remarkstoroidalmorphisms},
which should be well-known to experts in logarithmic geometry.  In this paper, we fulfill the details
proving (A) in Theorem~\ref{toroidal-to-sm-zar};
I learn this proof from Martin Olsson.

\mni
\textbf{\sffamily{\quad(III.b') Log-smoothness and toroidal morphisms.}}
Recall that \'etale-local toroidal morphisms are automatically formal-locally toroidal, hence 
the road map \eqref{road-map} also shows the equivalence:
\begin{tcolorbox}[colback=white, ams align]
\label{equi-toroidal-log-sm-intro}
    \text{dominant toroidal morphisms}\Longleftrightarrow \text{logarithmically smooth morphisms}. \tag{\ddag}
\end{tcolorbox}

\mni
\textbf{\sffamily{(IV) Proof of Theorem~\ref{sat-base-change-intro}.}}
After establishing the equivalences \eqref{equi-toroidal-log-reg-intro} 
and \eqref{equi-toroidal-log-sm-intro}, Theorem~\ref{sat-base-change-intro} follows immediately from 
two classical results in logarithmic geometry:
\begin{itemize}
    \item log-smoothness is stable under base change, and
    \item if $p\colon W\to Y$ is a log-smooth morphisms of fine and saturated log-schemes and if $Y$ is log-regular, then $W$ is also log-regular.
\end{itemize}
We complete the proof of Theorem~\ref{sat-base-change-intro} in \S\ref{sat-base-change-section}.


\mni
{\textbf{\sffamily{Acknowledgements.}}}
I am deeply grateful to my postdoc mentor Caucher Birkar for his encouragement to pursue this paper. 
My sincere thanks go to Dan Abramovich for patiently answering my numerous questions on logarithmic geometry. 
I am also greatly indebted to Martin Olsson for insightful discussions related to this work 
and for suggesting the proof of Theorem~\ref{toroidal-to-sm-zar}. 
I extend my appreciation to Arthur Ogus for clarifying aspects of his monograph \cite{Ogus-log-geo}. 
I would also like to thank Jaros{\l}aw W{\l}odarczyk and Dingxin Zhang for helpful conversations on logarithmic geometry. 
I am thankful to Junchao Shentu and Heer Zhao for their careful reading of an earlier draft. 
Finally, I wish to thank Jia Jia for valuable suggestions regarding the exposition of this paper.


\section{Toroidal geometry}\label{toroidal-geo-section}

In literature, a \emph{toroidal variety} refers to a \emph{toroidal embedding} or a \emph{toroidal couple}.
The original notion is the so-called toroidal embedding introduced in \cite{KKMB},
while the terminology toroidal couple is usually used in birational geometry
as pairs and boundary divisors are more emphasised there; see \cite{birkar2023singularities}.
The two notions are usually used interchangeably in practice, hence in this section we introduce
both of the two notions, and in the subsequent sections of this paper, we will switch between toroidal embeddings 
and toroidal couples without mentioning.  For toroidal couples,
we adopt the notation from \cite{birkar2023singularities}.

\subsection{Schemes and varieties}\label{schemes-varieties-defn}
A \emph{scheme} is defined in the sense of \cite[page 74]{Hart}.
All the schemes we consider in this paper are Noetherian.
An \emph{algebraic scheme} (or an \emph{algebraic scheme over $\mbb K$})
refers to a separated and finite type scheme over $\spec \mbb K$.
A \emph{variety} is an irreducible and reduced algebraic scheme over $\mbb K$.
For a scheme $X$, denote by $X_{\red}$ the maximal reduced closed subscheme of $X$, called 
the \emph{reduction} of $X$.

\begin{lemma}\label{local-irreducibility}
    Let $X$ be an excellent, normal scheme and $x\in X$ a scheme-theoretic point.
    Then $\wh{\mcO}_{X,x}$ is an integral, normal local ring; in other words,
    $\wh{X} \coloneqq \spec \wh{\mcO}_{X,x}$ is an irreducible, normal scheme.
\end{lemma}

\begin{proof}
    By \cite[Chap. 0, (23.2.1)]{EGA-IV-1}, $\mcO_{X,x}$ is geometrically unibranch.
    Then by \cite[Corollaire (7.6.3)]{EGA-IV-II}, $\wh{\mcO}_{X,x}$
    is an integral domain, hence $\wh{X}$ is irreducible.
    The normality of $\wh{\mcO}_{X,x}$ follows from \cite[Scholie (7.8.3)]{EGA-IV-II}.
\end{proof}


\subsection{Couples}\label{couples-defn}
A \emph{couple} $(X,D)$ consists of a variety $X$ and a reduced Weil divisor $D$ on $X$.
This is more general than the definition given in \cite[\S 2.19]{B-Fano} because we are not assuming $X$ to be normal 
nor quasi-projective.   We say a couple $(X, D)$ is \emph{projective} if the underlying variety $X$ 
is projective over $\spec \mbb K$.
Note that a couple is not necessarily a pair in the sense that we are not assuming 
$K_X+D$ to be $\RR$-Cartier; cf. \cite[\S 2.8]{B-Fano}.
For a couple $(X,D)$, if the underlying variety $X$ admits a morphism
$X\to S$ to an algebraic scheme $S$, we denote the couple as $(X/S, D)$ or $(X,D)\to S$.

\subsection{Morphisms of couples}\label{morphism-couple-defn}
A \emph{morphism} $(X, D)\to (Y, E)$ between couples is a morphism $f\colon X\to Y$
of varieties such that $f^{-1}(E)\subseteq D$.


\subsection{Toric varieties and toric couples}\label{toric-couples-defn}
We refer the readers to \cite{CLS:toric} for notation and general theory of toric varieties.
An \emph{affine toric variety $X$ of dimension $d$} is an affine variety containing 
an algebraic torus $\mathbb{T}_X\cong \mathbb{G}_m^d$ as
a Zariski open subset such that the action of $\mathbb{T}_X$ on itself
extends to an algebraic action of $\mathbb{T}_X$ on $X$.
By \cite[Theorem 1.3.5]{CLS:toric}, when the affine toric variety $X$ is normal, 
this is equivalent to giving a pair $(N_X, \sigma)$
where $N_X$ is a lattice of finite rank and $\sigma$ is a strongly convex rational polyhedral cone
in $N_X\otimes_{\ZZ}\RR$ such that $X =\spec \mbb K [\sigma^{\vee}\cap M_X]$,
where $M_X$ is the dual lattice of $N_X$ and $\sigma^{\vee}\subset M_X\otimes_{\ZZ} \RR$ is the dual cone of $\sigma$.
If $D$ is the \emph{toric boundary} of $X$, that is, $D$ is the reduction of complement 
of $\mathbb{T}_X$ in $X$, it is well-known that $(X, D)$ is log canonical and $K_X + D\sim 0$;
see \cite[Theorem 8.2.3]{CLS:toric}.
In this case, we say that $(X, D)$ is a \emph{toric couple}.

A \emph{toric morphism} of normal affine toric varieties $X$ and $Y$
is given by a linear map of lattices $\phi\colon N_X\to N_Y$
such that the image of the cone of $X$ under the extended map 
$\phi_{\RR}\colon N_X\otimes_{\ZZ}\RR \to N_Y\otimes_{\ZZ}\RR$
is contained in the cone of $Y$; see \cite[\S 3.3]{CLS:toric}.

\begin{lemma}\label{minimal-elements-lem}
    Let $s\in \NN$, and let $\Lambda$ be a set of sequences of natural numbers of length $s$, that is,
    every $\ell\in \Lambda$ is of the form $(n_1, \dots, n_s)$ with $n_j\in \NN$.
    Let $\preceq$ be the partial order on $\Lambda$ given by:
    for every $\ell, \ell'\in \Lambda$, we say $\ell\preceq \ell'$ if every component of $\ell'$ 
    is greater or equal than that of $\ell$.  
    Denote by $\Lambda_{\min}$ the set of all minimal elements of $(\Lambda, \preceq)$.
    Then $\Lambda_{\min}$ is a nonempty finite set, and
    for every $\ell\in \Lambda$, there exists an $\ell'\in \Lambda_{\min}$ such that $\ell'\preceq \ell$.
\end{lemma}

\begin{proof}
    For every $\ell\in \Lambda$, let $\ell^{\preceq}$ be the subset of $\Lambda$
    \[ \ell^{\preceq} \coloneqq \Set{ u\in \Lambda \mid u\preceq \ell }. \]
    Denote by $\ell^{\preceq}_{\min}$ the set of all minimal elements of $\ell^{\preceq}$.
    As $\ell^{\preceq}$ is a finite set, $\ell^{\preceq}_{\min}$ is nonempty.
    For every $\wt{\ell}\in \ell^{\preceq}_{\min}$, if $\wt{\ell}$ is not a minimal
    element of $\Lambda$, then there exists a $v\in \Lambda$, which is not equal to $\wt{\ell}$,
    such that $v\preceq \wt{\ell}$.  This implies that $v\preceq \wt{\ell}\preceq \ell$, hence $v\in \ell^{\preceq}$.
    However, this shows that $\wt{\ell}$ is not minimal in $\ell^{\preceq}$ neither, 
    contradicting that $\wt{\ell}\in \ell^{\preceq}_{\min}$.  Hence 
    \[ \ell^{\preceq}_{\min}\subseteq \Lambda_{\min} \,\text{ for every }\,\ell\in \Lambda. \]
    In particular, we see that for every $\ell\in \Lambda$, 
    there exists an $\ell'\in \Lambda_{\min}$ such that $\ell'\preceq \ell$.

    Assume that $\Lambda_{\min}$ contains an infinite subset $\Set{\ell_i}_{i\in \NN}$.
    For every $1\le j\le s$, let $\pr_j\colon \ZZ^s\to \ZZ$ be the projection to the $j$-th component.
    As every element in $\Lambda_{\min}$ is minimal, for every $i\in \NN$, 
    there exists a $j$ such that
    $\pr_j(\ell_{i+1}-\ell_i)\le -1$.  Then
    there exists a $j$ such that
    \[ \pr_j(\ell_{i+1} - \ell_i)\le -1 \,\text{ for every }\,i\in \NN. \]
    Then $\Set{\pr_j(\ell_i)}_{i\in \NN}$ is a strictly decreasing sequence of natural numbers,
    which is clearly impossible.  Thus, we can conclude that $\Lambda_{\min}$ is finite.
\end{proof}


\begin{lemma}\label{toric-boundary-ideal}
    Let $X$ be a normal affine toric variety given by $(N,\sigma)$,
    and let $D$ be the toric boundary of $X$, which is defined by an ideal $I_D\subset \K[X]$.  
    Let $M$ be the dual lattice of $N$ and $\sigma^{\vee}\subseteq M_{\RR} \coloneqq M\otimes_{\ZZ}\RR$
    the dual cone of $\sigma$.
    Then
    \[ I_D = \chi^{m_1}\cdot \K[X] + \cdots + \chi^{m_r}\cdot \K[X], \]
    where $\chi^{m_1}, \dots, \chi^{m_r}$ are characters of $\mbb T_X$
    with $m_1,\dots, m_r\in \sigma^{\vee}\cap M$ 
    such that every $\chi^{m_i}$ vanishes on the whole $D$ for $i = 1,\dots, r$.
\end{lemma}

\begin{proof}
    \emph{Step 1}.
    Note that $I_D$ is a $\mbb T_X$-invariant ideal of $\K[X]$.
    By \cite[Lemma 1.1.16]{CLS:toric}, we see that
    \[ I_D = \bigoplus_{\chi^m\in I_D} \mbb K\cdot \chi^m. \]
    Let $m\in M$ such that $\chi^m\in I_D$.
    By \cite[Proposition 4.1.2]{CLS:toric}, $m$ is non-negative on every ray generator of $\sigma$;
    see \cite[page 29]{CLS:toric} for the notion of \emph{ray generator}.
    As $\sigma$ is generated by its ray generators (see \cite[Lemma 1.2.15]{CLS:toric}),
    we can deduce that $m$ is non-negative on the whole $\sigma$;
    in other words, we have $m\in \sigma^{\vee}$.  Thus, we can conclude that
    \[ m\in \sigma^{\vee}\cap M \,\text{ for every }\, \chi^m\in I_D. \]
    Pick an arbitrary $\chi^m\in I_D$.  By \cite[Proposition 4.1.2]{CLS:toric},
    $\divisor(\chi^m)$ is a $\mbb T_X$-invariant divisor on $X$, hence $\supp (\divisor(\chi^m)) \subseteq \supp D$. 
    If $\supp (\divisor(\chi^m))$ is a proper subset of $\supp D$,
    then $V(I_D)\subsetneq \supp D$, contradicting the assumption that $I_D$ defines $D$.  Thus,
    \[ \supp (\divisor (\chi^m)) = \supp D \,\text{ for every }\,\chi^m\in I_D. \]

    \medskip
    
    \emph{Step 2}.
    By \cite[Proposition 1.2.17]{CLS:toric}, $\sigma^{\vee}\cap M$ is a finitely generated monoid.
    Denote by $\Set{\alpha_1, \dots, \alpha_s}$ a minimal set of generators of $\sigma^{\vee}\cap M$.
    Then every $m\in M$ with $\chi^m\in I_D$ has the form 
    \begin{equation}
        m = \sum_{j=1}^s n_j(m) \cdot\alpha_j \,\text{ with }\,n_j(m)\in \NN. \label{integer-coeffs}
    \end{equation}
    For every $\chi^m\in I_D$, we denote by
    \[ \ell(\chi^m) \coloneqq (n_1(m), \dots, n_s(m)) \]
    the sequence of natural numbers, which are coefficients in \eqref{integer-coeffs}, and let
    \[ \Lambda(I_D) \coloneqq \Set{ \ell(\chi^m) \mid \chi^m\in I_D}. \]
    For $m, m'\in I_D$, we say that $\ell(\chi^m)\preceq \ell (\chi^{m'})$ 
    if $n_j(m)\le n_j(m')$ for every $1\le j\le s$.
    It is evident that $(\Lambda(I_D), \preceq)$ is a partially ordered set.

    Denote by $\Lambda(I_D)_{\min}$ the set of all minimal elements of $(\Lambda(I_D), \preceq)$.
    Then by Lemma~\ref{minimal-elements-lem}, 
    \[ \Lambda(I_D)_{\min} = \Set{\ell(\chi^{m_1}), \dots, \ell(\chi^{m_r})} \] 
    is a nonempty finite set, 
    and for every $\chi^m\in I_D$, we have $\ell(\chi^{m_i}) \preceq \ell (\chi^m)$ 
    for some $\ell(\chi^{m_i})$.
    Then every $\chi^m\in I_D$ can be written as $\chi^{m_i}\cdot \chi^{\alpha}$
    for some $1\le i\le r$ and $\alpha\in\sigma^{\vee}\cap M$.
    Now for an arbitrary $f\in I_D$, recall that $f$ is a $\mbb K$-linear combination
    \[ f = \sum_{k=1}^t a_k\cdot \chi^{\beta_k}, \]
    where $a_k\in \mbb K$ and $\chi^{\beta_k}\in I_D$.  For every $\chi^{\beta_k}$, we have
    \[ \chi^{\beta_k} = \chi^{m_i}\cdot \chi^{\alpha_k} \]
    for some $\chi^{m_i}$ as above and some $\alpha_k\in \sigma^{\vee}\cap M$.
    Therefore, we can conclude that $I_D$ is the ideal generated 
    by $\chi^{m_1}, \dots, \chi^{m_r}$ as desired.
\end{proof}


\subsection{Toroidal couples}\label{rel-toroidal-couple}
Let $(X,D)$ be a couple.  
We say the couple is \emph{toroidal} at a closed point $x\in X$ 
if there exist (cf. \cite[II \S 1, Definition 1]{KKMB})
\begin{itemize}
    \item a \emph{normal} affine toric variety $W$, and
    \item a closed point $w\in W$ such that there is a $\mbb K$-algebra isomorphism
    \[ \phi\colon \wh{\mc{O}}_{X,x} \to \wh{\mc{O}}_{W,w} \]
    of completions of local rings so that the ideal of $D$ 
    is mapped to the ideal of the toric boundary divisor $C\subset W$, 
    that is, the reduction of complement of the torus $\mbb T_W$ in $W$;
\end{itemize}
see also \cite[Definition 1.2]{weak:s-stable:reduction} and \cite[\S 3.7]{birkar2023singularities}.
We call the data $\Set{(W,C),w}$ a \emph{local toric model of $\Set{(X,D),x}$}.
We say that the couple $(X,D)$ is \emph{toroidal} 
if it is toroidal at every closed point.
If $(X, D)$ is a toroidal couple, then $X$ is normal and Cohen-Macaulay;
see \cite[Lemma 3.8]{birkar2023singularities}.


\subsection{Embeddings and morphisms of embeddings}\label{embedding-morphism-defn}
Let $X$ be a scheme and $U$ an open subset of $X$ that is allowed to be equal to $X$.
We denote by $(U\subset X)$ the open immersion $U\hookrightarrow X$, 
and we call it an \emph{embedding}.  
A \emph{morphism of embeddings} $f\colon (U\subset X)\to (V\subset Y)$ is a morphism $f\colon X\to Y$
of schemes such that $f(U)\subseteq V$.


\subsection{Toroidal embeddings}
Let $(X, D)$ be a toroidal couple, and let $U_X\coloneqq X\setminus \supp D$, which is an open subset of $X$.
We call $(U_X\subset X)$ a \emph{toroidal embedding}; see \cite[page 54]{KKMB}. 
In particular, $U_X$ is smooth as $\mc{O}_{X,x}$ is regular if
and only if $\wh{\mc{O}}_{X ,x}$ is regular; cf. \cite[Chap. 0, Lemme (17.3.8.1)]{EGA-IV-1}.
We will use the notions toroidal couples and toroidal embeddings interchangeably
to be consistent with literature.
If the embedding $(U_X\subset X)$, or equivalently the couple $(X, D)$, 
is clear from the context, we just say that $X$ is a \emph{toroidal variety}.


\subsection{Strict toroidal embeddings and couples}
Let $(U\subset X)$ be a toroidal embedding, and let $(X, D)$ be the corresponding toroidal couple.  
We say that $(U\subset X)$ is a
\emph{strict toroidal embedding} (or a \emph{toroidal embedding without self-intersection}),
respectively, $(X, D)$ is a \emph{strict toroidal couple} (or a \emph{toroidal couple
without self-intersection})
if every irreducible component of $D$ is a normal variety;
see \cite[page 57]{KKMB}.  


\subsection{\'Etale-toric charts for strict toroidal embeddings}
The following existence result for \'etale-toric charts of strict toroidal embeddings is well-known,
which is outlined as a footnote on \cite[page 195]{KKMB}.  
We include the full details here for lack of direct references.
Note that to avoid confusions with the notion of \'etale-local charts for log-schemes 
in the setting of logarithmic geometry,
we call these charts \emph{strict \'etale-toric charts} in Definition~\ref{strict-chart-defn}
as they arise from local toric models.
This terminology is clearly not standard.

\begin{theorem}\label{strict-chart}
    Let $(U\subset X)$ be a strict toroidal embedding.
    Let $D$ be the reduced divisor that is the reduction of $X\setminus U$.
    Then for every closed point $x\in X$, and for every 
    local toric model $\Set{(W, C), w}$ of $\Set{(X, D), x}$, there exist 
    \begin{itemize}
        \item an open neighbourhood $U_x\subset X$ of $x$, and
        \item an \'etale morphism $\pi\colon U_x\to W$ such that 
        $U\cap U_x = \pi^{-1}(\mbb T_W)$, where $\mbb T_W$ is the torus of $W$.
    \end{itemize}
\end{theorem}

\begin{definition}[Strict \'etale-toric chart]\label{strict-chart-defn}
    Let $(U\subset X)$ be a strict toroidal embedding and $x\in X$ a closed point.
    The data consisting of $\pi\colon U_x\to W$ and $\Set{(W, C),w}$ as above is called a
    \emph{strict \'etale-toric chart} for $(U\subset X)$ near $x$.
\end{definition}

\begin{proof}[Proof of Theorem~\ref{strict-chart}]
    Let $\Set{(W, C),w}$ be a local toric model of $\Set{(X, D), x}$.
    As $W$ is a normal affine toric variety,
    it is given by the data $(N, \sigma)$ as in \S\ref{toric-couples-defn}. 
    Then $W =\spec \mbb K [\sigma^{\vee}\cap M]$,
    where $M$ is the dual lattice of $N$ and $\sigma^{\vee}\subset M_{\RR}$ is the dual cone of $\sigma$.
    Let $d \coloneqq \dim X$, then $M\cong \ZZ^d$ is isomorphic to the group of characters of $\mbb T_W$.  Denote by 
    \begin{equation}
        \phi\colon \wh{\mc{O}}_{W,w}\to \wh{\mc{O}}_{X,x} \label{iso-completion-starting-pf}
    \end{equation}
    the isomorphism of $\mbb K$-algebras that gives the local toric model. 

    \medskip

    In Steps 1 to 3, we construct strict \'etale-toric chart when $w\in W$ is a fixed point.

    \medskip

    \emph{Step 1}.
    Assume that $w\in W$ is a fixed point under the $\mbb T_W$-action.
    By \cite[Corollary 1.3.3]{CLS:toric}, such a fixed point is unique,
    $\dim \sigma = \dim N_{\RR}$, and $w$ corresponds to the cone $\sigma$
    in the orbit-cone correspondence \cite[Theorem 3.2.6]{CLS:toric}.
    Then by \cite[Lemma 1, page 60]{KKMB}, $\phi$ in \eqref{iso-completion-starting-pf} 
    induces an isomorphism of Abelian groups
    \[ \rho\colon  M\to \CDiv (X, D; x) \coloneqq \Set{\begin{array}{c}
    \text{group of Cartier divisors } \\ \text{at }x, \text{ supporting in }D
    \end{array} }. \]
    As $M\cong \ZZ^d$, take basis 
    $e_1,\dots,e_d$ of $M$, then $\Set{\rho(e_i)}_{i=1}^d$ 
    is a set of basis for $\CDiv (X, D; x)$.
    Fix a choice of local equation $f_i$ for every $\rho(e_i)$.
    Note that the image of $f_i$ in $\wh{\mcO}_{X,x}$ and $\phi(\chi^{e_i})$ differ by a unit in $\wh{\mc{O}}_{X,x}$,
    where $\chi^{e_i}$ is the character of $\mbb T_W$ corresponding to $e_i$.
    For an arbitrary $\alpha = \sum_{i=1}^d m_i e_i\in M$ with $m_i\in \ZZ$, set
    \[ f^{\alpha} \coloneqq \prod_{i=1}^d f_i^{m_i}. \]
    It is easy to see that the set of rational functions $\Set{f^{\alpha}} \subset \K(X)$ satisfies that
    \[ f^{\alpha+\beta} = f^{\alpha}\cdot f^{\beta} \text{ for any }\alpha,\beta\in M. \]
    By \cite[Proposition 1.2.17]{CLS:toric}, there is a finite set
    $\Set{\alpha_1,\dots, \alpha_s}$ generating the monoid $\sigma^{\vee}\cap M$.
    Then $\K[W]$ is generated  
    by characters $\chi^{\alpha_1}, \dots, \chi^{\alpha_s}$ as a $\mbb K$-algebra.
    Note that for every $j=1,\dots, s$, 
    the image $\rho(\alpha_j)$ is an effective Cartier divisor in $\CDiv(X, D;x)$ by \cite[Lemma 1, page 60]{KKMB}, 
    hence $f^{\alpha_1}, \dots, f^{\alpha_s}$ are regular functions 
    on an open neighbourhood $U_x$ of $x\in X$ (as they are defining equations
    of $\rho(\alpha_1), \dots, \rho(\alpha_s)$).
    Shrinking $U_x$ near $x$, we can assume that $U_x$ is affine.
    
    Recall that the function field $\K(W)$ of $W$ is equal to 
    \[ \mbb K (\chi^{e_1}, \chi^{e_2}, \dots, \chi^{e_d}). \]
    Thus, there is a well-defined homomorphism of $\mbb K$-algebras
    \[ \varphi^{\circ}\colon \K(W) \to \K(X), \,\,\chi^{\alpha}\mapsto f^{\alpha} \,\text{ for every }\,\alpha\in M, \]
    where $\K(X)$ is the function field of $X$.
    As $\K[W]$ is a subring of $\K(W)$, $\varphi^{\circ}$ restricts to an injective homomorphism of $\mbb K$-algebras
    \[ \varphi\colon \K[W]\to \K[U_x], \,\, \chi^{\alpha_j} \mapsto f^{\alpha_j} \,\text{ with }\,j=1,\dots, s, \]
    Let 
    \[ \psi\colon U_x\to W \] 
    be the morphism of varieties induced by $\varphi$.  Since $\varphi$ is injective, $\psi$ is dominant. 

    \medskip

    \emph{Step 2}.
    We show that $\psi$ is \'etale at $x$.
    Let $\mf{n}_w\subset \K[W]$ be the maximal ideal defining $w$.
    By \cite[Corollary 1.3.3]{CLS:toric}, $\mf{n}_w$ is the ideal generated by $\chi^{\alpha_1}, \dots, \chi^{\alpha_s}$.
    Denote by $\mf{m}_x\subset \K[U_x]$ the maximal ideal defining $x\in U_x$,
    then $\mf{m}_x$ consists of all the regular functions on $U_x$ vanishing at $x$.
    Again by \cite[Lemma 1, page 60]{KKMB}, $f^{\alpha_j}$ vanishes at $x$ for every $j=1\, \dots, s$.
    Thus, 
    \[ \varphi^{-1}(\mf{m}_x) = \mf{n}_w, \]
    which is equivalent to saying that $\psi(x) = w$.

    As $\chi^{\alpha_1}, \dots, \chi^{\alpha_s}$ generate $\mf{n}_w$,
    the images of $\chi^{\alpha_1}, \dots, \chi^{\alpha_s}$ in $\wh{\mcO}_{W,w}$ generate $\mf{n}_w\wh{\mcO}_{W,w}$.
    Since $\mcO_{W,w}\to \wh{\mcO}_{W,w}$ is injective, we also denote by $\chi^{\alpha_j}$ 
    the image of every $\chi^{\alpha_j}$ in $\wh{\mcO}_{W,w}$
    in the rest of the proof.  As $\phi$ in \eqref{iso-completion-starting-pf} is an isomorphism, 
    we see that $\phi(\chi^{\alpha_1}), \dots, \phi(\chi^{\alpha_s})$ generate $\mf{m}_x\wh{\mc{O}}_{X,x}$.
    By construction in Step 1, $f^{\alpha_j}$ and $\phi(\chi^{\alpha_j})$ differ by a unit in $\wh{\mc{O}}_{X,x}$
    for every $j=1, \dots, s$, say,
    \[ \phi(\chi^{\alpha_j}) = u_j\cdot f^{\alpha_j} \,\text{ with }\, u_j\in \wh{\mc{O}}_{X,x}^*. \]
    Then the set $\Set{f^{\alpha_1}, \dots, f^{\alpha_s}}$ also generates $\mf{m}_x\wh{\mc{O}}_{X,x}$.
    As $\mc{O}_{X,x}\to \wh{\mc{O}}_{X,x}$ is faithfully flat, $\mf{m}_x$
    is generated by $f^{\alpha_1}, \dots, f^{\alpha_s}$ by \cite[Theorem 7.5]{CA-Mat}.
    Thus, we can conclude that
    \[ \varphi (\mf{n}_w)\cdot \mcO_{X,x} = \mf{m}_x\mcO_{X,x}, \]
    which is equivalent to saying that $\psi\colon U_x\to W$ is unramified at $x$.
    In addition, recall that $\mc{O}_{W,w}\to \mc{O}_{X,x}$ is injective as $\psi\colon U_x\to W$ is dominant,
    then by \cite[Lemma 1.5, page 9]{etale-cohomology-weil}, we can deduce that $\psi$ is \'etale at $x$. 
    Up to shrinking $U_x$ near $x$, we can assume that $\psi\colon U_x\to W$ is an \'etale morphism
    of normal affine varieties. 

    \medskip

    \emph{Step 3}.
    We show that $\psi^{-1}(\mbb T_W) = U\cap U_x$ in this step.
    Note that the isomorphism $\wh{\varphi}\colon \wh{\mc{O}}_{W,w}\to \wh{\mc{O}}_{X,x}$ 
    of completion of local rings induced by $\varphi\colon \K [W]\to \K[U_x]$ 
    may not be the given isomorphism $\phi$ in \eqref{iso-completion-starting-pf}.  Let 
    \[ \phi_w\colon \wh{\mcO}_{W,w}\to \wh{\mcO}_{W,w} \] 
    be the isomorphism of completion of local rings, which is the composite
    of $\phi^{-1}$ and $\wh{\varphi}$.  Then
    \[ \phi_w(\chi^{\alpha_j}) = v_j\cdot \chi^{\alpha_j} \]
    for every $j=1,\dots, s$, where $v_j \coloneqq \phi^{-1}(u_j^{-1})\in \wh{\mcO}_{W,w}^*$.
    By Lemma~\ref{toric-boundary-ideal}, the ideal $I_C\subset \K[W]$ defining
    the toric boundary $C$ 
    is generated by $\chi^{m_1},\dots, \chi^{m_r}$, where $m_i\in \sigma^{\vee}\cap M$,
    and every $\chi^{m_i}$ vanishes on the whole $C$.  As the set $\Set{\alpha_1,\dots, \alpha_s}$ 
    generates $\sigma^{\vee}\cap M$, we see that
    \[ \phi_w(\chi^{m_i}) = \wt{v}_i\cdot \chi^{m_i} \]
    for every $i=1,\dots, r$, where $\wt{v}_i\in \wh{\mcO}_{W,w}^*$.
    Thus, $\phi_w$ maps the ideal $I_C\cdot \wh{\mcO}_{W,w}$ isomorphically to $I_C\cdot \wh{\mcO}_{W,w}$.
    Since $\phi$ in \eqref{iso-completion-starting-pf} maps 
    $I_C\cdot \wh{\mcO}_{W,w}$ isomorphically to $I_D\cdot \wh{\mcO}_{X,x}$,
    we can conclude that $\wh{\varphi}\colon \wh{\mcO}_{W,w}\to \wh{\mcO}_{X,x}$ 
    maps $I_C\cdot \wh{\mcO}_{W,w}$ isomorphically to the ideal $I_D\cdot \wh{\mcO}_{X,x}$.
    Again by \cite[Theorem 7.5]{CA-Mat}, the pullback $\psi^*I_C$ is equal to $I_D$ near $x$,
    hence up to shrinking $U_x$ near $x$, we get $\psi^{-1}(\mbb T_W) = U\cap U_x$.
    Thus, $\psi\colon U_x\to W$ is a strict \'etale-toric chart. 

    \medskip

    In Steps 4 and 5, we show that if $w\in W$ is not a fixed point, then $W$ has a torus factor
    locally around a general closed point in the orbit of $w$.

    \medskip

    \emph{Step 4}.
    Assume that $w\in W$ is not a fixed point of the $\mbb T_W$-action. 
    By the orbit-cone correspondence, there is a face $\tau$ of $\sigma$ such that $O(\tau) = \mbb T_W\cdot w$;
    see \cite[Theorem 3.2.6]{CLS:toric}.  For any other closed point $w'\in O(\tau)$,
    $\Set{(W, C),w'}$ is also a local toric model of $\Set{(X, D),x}$.  Thus,    
    we can assume that $w$ is a general closed point of $O(\tau)$.  Note that
    \[ \dim O(\tau) = d - \dim \tau. \]
    By \cite[Proposition 1.3, page 7]{oda-convex-bodies-toric}, 
    there is an $m\in \sigma^{\vee}\cap M$ such that $\tau = \sigma\cap H_m$,
    where $H_m \coloneqq \Set{u\in N_{\RR} \mid \langle m, u \rangle = 0} \subseteq N_{\RR}$.
    Also by \cite[Proposition 1.3]{oda-convex-bodies-toric}, $\tau$ is also a strongly convex
    rational polyhedral cone in $N_{\RR}$.
    Let $U_{\tau, N}$ be the normal affine toric variety associated to the data $(N, \tau)$,
    that is, $U_{\tau, N} = \spec \mbb K [\tau^{\vee}\cap M]$.  By \cite[Proposition 1.3.16]{CLS:toric}, 
    $\mbb K[\tau^{\vee}\cap M]$ is the localisation of $\mbb K[\sigma^{\vee}\cap M]$ at $\chi^m\in \K[W]$,
    hence $U_{\tau, N}$ is an affine open subset of $W$.
    
    It is evident that $m\in \sigma^{\vee}\cap M\cap \tau^{\perp}$.
    Let $V(\tau)$ be the Zariski closure of $O(\tau)$, which is a toric subvariety of $W$;
    see \cite[Proposition 3.2.7]{CLS:toric}.
    By \cite[page 53]{fulton-toric}, 
    \[ V(\tau) = \spec \mbb K [\sigma^{\vee}\cap M \cap \tau^{\perp}]. \]
    Thus, $\chi^m$ does not vanish identically on $V(\tau)$.
    Then $U_{\tau, N}$ contains an open subset of $O(\tau)$. 
    Since $w$ is a general closed point of $O(\tau)$, we see that $w\in U_{\tau, N}$.
    Furthermore, by the proof of \cite[Theorem 1.2.18]{CLS:toric},
    the torus $\mbb T_{U_{\tau, N}}$ of $U_{\tau, N}$ is also equal to $\spec \mbb K[M]$,
    that is, $\mbb T_{U_{\tau, N}} = \mbb T_W$. 
    Then $C_{\tau, N} \coloneqq C\cap U_{\tau, N}$ is the toric boundary of $U_{\tau, N}$.  
    Thus, $\Set{(U_{\tau, N}, C_{\tau, N}),w}$ is also a local toric model of $\Set{(X, D),x}$.
    Therefore, we can replace $W$ by $U_{\tau, N}$ and assume that $W$ is given by the data $(N, \tau)$. 

    \medskip

    \emph{Step 5}.
    Now we adapt the argument from \cite[page 41]{CLS:toric}.
    Set $t \coloneqq \dim \tau$.  Let $N_1\subseteq N$ be the smallest saturated sublattice containing the 
    generators of $\tau$, then $N/N_1$ is torsion free.  By \cite[Exercise 1.3.5]{CLS:toric},
    there is a sublattice $N_2\subseteq N$ such that $N = N_1\oplus N_2$.
    Note that $\rank N_1 = t$ and $\rank N_2 = d-t$.
    Denote by $M_1, M_2$ the dual lattices of $N_1, N_2$ respectively.
    Then $M = M_1\oplus M_2$.  It is evident that
    \[ \tau^{\vee}\cap M = (\tau^{\vee}\cap M_1)\oplus M_2, \]
    which can be written in terms of $\mbb K$-algebras as
    \[ \mbb K[\tau^{\vee}\cap M] = \mbb K[\tau^{\vee}\cap M_1]\otimes_{\mbb K} \mbb K[M_2]. \]
    Thus, we can conclude that
    \[ W = U_{\tau, N_1} \times_{\spec \mbb K} \mbb T_{N_2}, \]
    where $\mbb T_{N_2}$ is the torus associated to the lattice $N_2$, and $U_{\tau, N_1}$
    is the normal affine toric variety defined by the data $(N_1, \tau)$.
    Note that $\dim U_{\tau, N_1} = t$ and $\dim \mbb T_{N_2} = d-t$.
    By \cite[Corollary 1.3.3]{CLS:toric}, as $\tau$ has maximal dimension in $N_1$, 
    $U_{\tau, N_1}$ has a unique fixed point by which we denote $p_{\tau}$.
    Let $C_{\tau, N_1}$ be the toric boundary of $U_{\tau, N_1}$,
    that is, $C_{\tau, N_1}$ is the reduction of $U_{\tau, N_1}\setminus \mbb T_{N_1}$,
    where $\mbb T_{N_1}$ is the torus associated to the lattice $N_1$.  Then
    it is evident that
    \begin{equation}
        \mbb T_W = \mbb T_{N_1}\times_{\spec \mbb K} \mbb T_{N_2} \,\text{ and }\,
    C = C_{\tau, N_1}\times_{\spec \mbb K} \mbb T_{N_2}. \label{torus-and-boundary}
    \end{equation}
    Moreover, for any closed point $p'\in U_{\tau, N_1}$ that is not the fixed point $p_{\tau}$,
    and for any closed point $q\in \mbb T_{N_2}$, the orbit of $\Set{p'}\times_{\spec \mbb K} \Set{q}$
    under the $\mbb T_W$-action has dimension $>d-t$, hence for dimension reason, we have
    \[ O(\tau) = \Set{p_{\tau}}\times_{\spec \mbb K}\mbb T_{N_2}. \]
    We assume that $w\in O(\tau)$ is the point $\Set{p_{\tau}}\times_{\spec \mbb K} (1, 1,\dots, 1)$
    in the rest of the proof.

    \medskip

    In Steps 6 to 9, we construct strict \'etale-toric chart when $w\in W$ is not a fixed point.
    We keep the assumptions in Steps 4 and 5.

    \medskip
 
    \emph{Step 6}.
    In this step, we write out the maximal ideal $\mf{n}_w$ defining $w\in W$ explicitly.
    By \cite[Lemma 1, page 60]{KKMB}, there is a surjective homomorphism of Abelian groups
    \[ \rho\colon  M\to \CDiv (X,D; x). \]
    As $M = M_1\oplus M_2$, by \cite[Proposition 4.1.2]{CLS:toric}, 
    it is easy to see that $M_2$ is the kernel of $\rho$.
    Let $\Set{e_1, \dots, e_t}$ be a set of generators of $M_1$, and let $\Set{e_{t+1}, \dots, e_d}$
    be a set of generators of $M_2$. 
    Let $\chi^{e_1}, \dots, \chi^{e_t}$ be characters of $\mbb T_W$ corresponding to $e_1,\dots, e_t$.
    Fix a choice of local equation $f_i$ for every $\rho (e_i)$, $i=1,\dots, t$.
    Then $f_i$ and $\phi(\chi^{e_i})$ differ by a unit in $\wh{\mc{O}}_{X,x}$ for every $i=1,\dots, t$.
    As in Step 1, for every $\alpha = \sum_{i=1}^t m_i e_i \in M_1$ with $m_i\in \ZZ$, let
    \[ f^{\alpha} \coloneqq \prod_{i=1}^t f_i^{m_i}. \]
    Then the set $\Set{f^{\alpha}}_{\alpha\in M_1}$ also satisfies the relation that 
    $f^{\alpha + \beta} = f^{\alpha}\cdot f^{\beta}$
    for any $\alpha,\beta\in M_1$.
    Moreover, for every $\alpha\in M_1$, $f^{\alpha}$ is a defining equation 
    of the Cartier divisor $\rho(\alpha)$ at $x$.

    Let $\alpha_1, \dots, \alpha_s\in M_1$ be a set of generators of the monoid $\tau^{\vee}\cap M_1$.  Then
    by \cite[Corollary 1.3.3]{CLS:toric}, $p_{\tau}\in U_{\tau, N_1}$ is defined by the ideal 
    \begin{align*}
        I_{\tau, N_1} &\coloneqq \Set{ \chi^m \mid m\in (\tau^{\vee}\cap M_1)\setminus \Set{0} }
    \cdot \K[U_{\tau, N_1}] \\
     &= \chi^{\alpha_1}\cdot \K[U_{\tau, N_1}] + \cdots + \chi^{\alpha_s}\cdot \K[U_{\tau, N_1}].
    \end{align*}
    Then $O(\tau)$ is the closed subset of $W$ defined by the ideal 
    \[ I_{\tau, N} \coloneqq I_{\tau, N_1} \cdot \K[W]. \] 
    On the other hand, the point $(1,1, \dots, 1)\in \mbb T_{N_2}$ is defined by 
    the equations $\chi^{e_{t+1}} -1, \dots, \chi^{e_d} - 1$.  Then
    \[ \mf{n}_w \coloneqq \langle I_{\tau, N}, \chi^{e_{t+1}} -1, \dots, \chi^{e_d} - 1 \rangle \] 
    is the maximal ideal of $\K [W]$ defining $w\in W$.

    \medskip

    \emph{Step 7}.
    In this step, we construct a morphism of varieties $\varphi\colon U_x\to W$
    from an affine open subset $U_x$ of $x\in X$.
    We can assume that $X$ is affine, hence quasi-projective over $\spec \mbb K$.
    Let $H_{t+1}, \dots, H_d$ be sufficiently general hyperplane sections of $X$ passing through $x$
    so that every $H_j$ intersects $D$ transversally.
    If $X$ is embedded into a projective space $\mbb P$, we can take every $H_j$
    as the pullback of a general hyperplane of $\mbb P$.
    Then every $H_j$ is a Cartier divisor, hence we can take its defining equation $g_j$ with $j=t+1, \dots, d$.
    Let $I_D$ be the ideal defining $D$.  Then the closed point $x$ is defined by the maximal ideal
    \[ \mf{m}_x = \langle I_D, g_{t+1}, \dots, g_d \rangle\subset \K[X]. \]
    Recall that the function field $\K(W)$ of $W$ is equal to 
    \[ \mbb K (\chi^{e_1}, \dots, \chi^{e_t}, \chi^{e_{t+1}}, \dots, \chi^{e_d}). \]
    Define the homomorphism of $\mbb K$-algebras
    \[ \varphi^{\circ} \colon \K(W) \to \K(X) \text{ via } \chi^{\alpha}\mapsto f^{\alpha} 
    \text{ and } \chi^{e_j}\mapsto g_j + 1 \]
    for every $\alpha\in M_1$ and every $e_j$ with $j=t+1, \dots, d$.  As $\Set{\chi^{e_1}, \dots, \chi^{e_d}}$
    is a transcendence basis of the field $\K(W)$, we see that $\varphi^{\circ}$ is well-defined.
    Recall that $\Set{\chi^{\alpha_1}, \dots, \chi^{\alpha_s}}$ 
    is a set of generators of the ideal $I_{\tau, N}$, then
    every $\chi^{\alpha_k}$ with $k=1, \dots, s$ is a regular function on $W$.
    Thus, by \cite[Lemma 1, page 60]{KKMB},
    every $\rho(\alpha_k)$ with $k=1, \dots, s$ is an effective Cartier divisor in $\CDiv(X, D; x)$.
    As $f^{\alpha_k}$ is a defining equation of $\rho(\alpha_k)$,
    there is an affine open neighbourhood $U_x$ 
    of $x$ so that every $f^{\alpha_k}$ with $k=1,\dots, s$
    is a regular function on $U_x$.  By restricting $\varphi^{\circ}$ to the subring $\K[W]$, 
    we get an injective homomorphism of $\mbb K$-algebras
    \[ \varphi\colon \K[W]\to \K[U_x]. \]
    Let $\psi\colon U_x\to W$
    be the morphism of normal affine varieties induced by $\varphi$.

    \medskip
    
    \emph{Step 8}. 
    In this step, we show that $\psi\colon U_x\to W$ is \'etale at $x$.
    By the same argument in Step 2,
    we see that
    \[ \varphi^{-1}(\mf m_x) = \mf n_w, \]
    hence $\psi(x) = w$.  Recall that $O(\tau)$ is defined by the ideal $I_{\tau, N}$ 
    that is generated by $\chi^{\alpha_1}, \dots, \chi^{\alpha_s}$, 
    and that $\mf n_w$ is the maximal ideal
    \[ \langle \chi^{\alpha_1}, \dots, \chi^{\alpha_s}, \chi^{e_{t+1}} - 1, \dots, \chi^{e_d} -1 \rangle. \]
    As $f_i$ and $\phi(\chi^{e_i})$ differ by a unit in $\wh{\mcO}_{X,x}$ 
    for every $i=1,\dots, t$ (see Step 6), $f^{\alpha_j}$ and $\phi(\chi^{\alpha_j})$
    also differ by a unit in $\wh{\mcO}_{X,x}$ for every $j=1,\dots, s$.
    Since $\phi$ in \eqref{iso-completion-starting-pf} maps the maximal ideal 
    $\mf n_w\wh{\mcO}_{W,w}$ of $\wh{\mcO}_{W,w}$ to the maximal ideal of $\wh{\mcO}_{X,x}$, we see that 
    \[ \Set{\phi(\chi^{\alpha_1}), \dots, \phi(\chi^{\alpha_s}), g_{t+1}, \dots, g_d} \]
    generates the maximal ideal of $\mf m_x\wh{\mcO}_{X,x}$.  Thus, the set
    \[ \Set{f^{\alpha_1}, \dots, f^{\alpha_s}, g_{t+1}, \dots, g_d} \]
    also generates $\mf m_x\wh{\mcO}_{X,x}$.  Again by \cite[Theorem 7.5]{CA-Mat}, we can conclude that
    \[ \varphi (\mf{n}_w) \cdot \mcO_{X,x} = \mf{m}_x\mcO_{X,x}. \]
    In other words, $\psi$ is unramified at $x$.
    Again by \cite[Lemma 1.5, page 9]{etale-cohomology-weil}, we see that $\psi$ is \'etale at $x$.
    Then up to shrinking $U_x$ near $x$, we assume $\psi\colon U_x\to W$ is \'etale. 

    \medskip

    \emph{Step 9}.
    In this step, we show $\psi^{-1}(\mbb T_W) = U\cap U_x$.
    Denote by 
    \[ \wh{\varphi}\colon \wh{\mc{O}}_{W,w}\to \wh{\mc{O}}_{X,x} \]
    the induced homomorphism of completion of local rings.
    Notice that as remarked in Step 3,
    $\wh{\varphi}$ may not be the given isomorphism 
    $\phi\colon \wh{\mc{O}}_{W,w}\to \wh{\mc{O}}_{X,x}$ in \eqref{iso-completion-starting-pf}.
    Keep the notation in Step 5.  
    By Lemma~\ref{toric-boundary-ideal}, the toric boundary $C_{\tau, N_1}$ of $U_{\tau, N_1}$
    is defined by a finite set of characters
    \[ \chi^{m_1}, \dots, \chi^{m_r} \,\text{ with }\, 
    m_1,\dots, m_r\in \tau^{\vee}\cap M_1\subseteq \tau^{\vee}\cap M, \]
    where every $\chi^{m_i}$ with $i=1,\dots, r$ vanishes on the whole $C_{\tau, N_1}$.
    Then by \eqref{torus-and-boundary}, the toric boundary $C$ of $W$ is also defined by 
    the characters $\chi^{m_1}, \dots, \chi^{m_r}$.  As in Step 3, let
    \[ \phi_w\colon \wh{\mcO}_{W,w}\to \wh{\mcO}_{W,w} \]
    be the composite of $\phi^{-1}$ and $\wh{\varphi}$.  Then the same argument in Step 3
    shows that $\phi_w$ maps the ideal $I_C\cdot \wh{\mcO}_{W,w}$ isomorphically to $I_C\cdot \wh{\mcO}_{W,w}$,
    hence the pullback $\psi^*I_C$ is equal to $I_D$ near $x$.
    Therefore, up to shrinking $U_x$ near $x$,
    we can conclude that $\psi^{-1}(\mbb T_W) = U\cap U_x$.
\end{proof}


\subsection{\'Etale-toric charts for toroidal embeddings}

\begin{lemma}[cf. \protect{\cite[Remark 2.4]{denef2013remarkstoroidalmorphisms}}]\label{pass-to-strict}
    Let $(U\subset X)$ be a toroidal embedding and $x\in X$ a closed point.
    Then there is an \'etale morphism $\pi\colon X'\to X$ onto an open neighbourhood of $x$ 
    such that $(\pi^{-1}(U)\subset X')$ is a strict toroidal embedding.
\end{lemma}

\begin{proof}
    Denote by $D$ the reduced divisor that is the reduction of $X\setminus U$.
    Let $\Set{(W,C),w}$ be a local toric model of $\Set{(X,D),x}$.
    Let $I_D\subset \mc{O}_{X,x}$ (respectively, $I_C\subset \mcO_{W,w}$) 
    be the ideal corresponding to the divisor $D$ (respectively, $C$).
    Denote by $\mc{O}_{X,x}^h$ the henselisation of $\mc{O}_{X,x}$.
    By \cite[Th\'eor\`eme (18.6.6)]{EGA-IV-4},
    $\mcO_{X,x}^h$ is also a Noetherian local ring,
    whose completion $\widehat{\mcO_{X,x}^h}$ is canonically isomorphic to $\wh{\mcO}_{X,x}$.
    Let $\mf{p}\subset \mc{O}_{X,x}^h$ be a prime ideal of height one
    containing $I_D\mc{O}_{X,x}^h$, that is, $\mf p$ corresponds to
    an irreducible component of the inverse image of $D$ in $\spec \mcO_{X,x}^h$.  
    We claim that $\mf{p}\cdot \widehat{\mcO_{X,x}^h}$ is a prime ideal of $\widehat{\mcO_{X,x}^h}$.
    By \cite[Corollaire (18.7.6)]{EGA-IV-4}, $\mcO_{X,x}^h$ is excellent as $\mcO_{X,x}$ is so.
    Then by \cite[Proposition (7.8.6)]{EGA-IV-II}, 
    $\mc{O}_{X,x}^h/\mf p$ is also excellent;
    in particular, $\mcO_{X,x}^h/\mf p$ has geometrically regular formal fibres.
    Moreover, by \cite[Proposition (18.6.8)]{EGA-IV-4}, 
    $\mcO_{X,x}^h/\mf p$ is also henselian.  Thus, by \cite[Corollaire (18.9.2)]{EGA-IV-4},
    we can conclude that $ \wh{\mcO_{X,x}^h}/\mf p\cdot \wh{\mcO_{X,x}^h}$ is integral,
    so $\mf{p}\cdot \widehat{\mcO_{X,x}^h}$ is a prime ideal, i.e., the claim holds.
    Now, since every irreducible component of
    the inverse image of $C$ in $\spec \wh{\mc{O}}_{W,w}$ is normal, and since
    \[ \yhwidehat{\mcO_{X,x}^h/\mf p} = (\mcO_{X,x}^h/\mf p)\otimes_{\mcO_{X,x}^h}\wh{\mcO_{X,x}^h}
    = \wh{\mcO_{X,x}^h}/\mf p\cdot \wh{\mcO_{X,x}^h} \cong \wh{\mcO}_{W,w}/\mf q \]
    for a prime ideal of height one of $\wh{\mcO}_{W,w}$ containing $I_C\wh{\mcO}_{W,w}$, we see that
    $\yhwidehat{\mc{O}_{X,x}^h/\mf{p}}$ is normal.
    Then by \cite[Scholie (7.8.3)]{EGA-IV-II}, 
    $\mc{O}_{X,x}^h/\mf{p}$ is also normal;  
    in other words, every irreducible component of the inverse image of $D$
    to $\spec \mc{O}_{X,x}^h$ is normal.
    By construction of henselisation (see \cite[\S 18.6]{EGA-IV-4}), there is an \'etale neighbourhood 
    $\pi\colon X'\to X$ of $x$ such that $(\pi^{-1}(U)\subset X')$ is a strict toroidal embedding.  
\end{proof}


\begin{theorem}\label{general-chart}
    Let $(U\subset X)$ be a toroidal embedding, and
    let $D$ be the reduced divisor that is the reduction of $X\setminus U$.
    Then for every closed point $x\in X$, and for every local toric model 
    $\Set{(W, C), w}$ of $\Set{(X, D), x}$, there exist 
    \'etale morphisms $p\colon V\to X$ and $q\colon V\to W$
    \[\xymatrix{
      & V\ar[rd]^q\ar[ld]_p \\
    X &  & W 
    }\]
    such that $p^{-1}(U) = q^{-1}(\mbb T_W)$ on $V$, and $(p^{-1}(U)\subset V)$ is a strict toroidal embedding.
\end{theorem}

\begin{definition}[\'Etale-toric chart]
    Let $(U\subset X)$ be a toroidal embedding
    and $x\in X$ a closed point.
    The data consisting of $p\colon V\to X$, $q\colon V\to W$ and $\Set{(W, C), w}$ as above
    is called an \emph{\'etale-toric chart} for $(U\subset X)$ near $x$.
\end{definition}

\begin{proof}[Proof of Theorem~\ref{general-chart}]
    By Lemma~\ref{pass-to-strict}, there is an \'etale morphism $p\colon V\to X$
    onto an open neighbourhood of $x$ such that $(p^{-1}(U)\subset V)$ is a strict toroidal embedding.
    Let $v\in V$ be a closed point 
    mapping to $x$, and let $E$ be the reduced divisor that is the reduction of $V\setminus p^{-1}(U)$.
    Then $\Set{(W, C), w}$ is also a local toric model of 
    $\Set{(V, E), v}$ as $p$ is \'etale.  By Theorem~\ref{strict-chart},
    up to shrinking $V$ around $v$, there exists an \'etale morphism
    $q\colon V\to W$, which is a strict \'etale-toric chart for 
    $(p^{-1}(U)\subset V)$ near $v$.
\end{proof}


\subsection{Toroidal morphisms}\label{toroidal-morphisms}

Let $(X,D)$ and $(Y,E)$ be couples, 
and let $f\colon X\to Y$ be a morphism of varieties. 
Let $x\in X$ be a closed point and $y=f(x)$. We say $(X,D)\to (Y,E)$ 
is \emph{toroidal at $x$} if there exist local 
toric models $\Set{(W,C),w}$ and $\Set{(V,B),v}$ of 
$\Set{(X,D), x}$ and $\Set{(Y,E), y}$ respectively, 
and a toric morphism $g\colon W\to V$ of normal affine toric varieties 
so that we have a commutative diagram 
\begin{equation}
    \begin{aligned}\label{formal-defn-diagram}
    \xymatrix{
    \widehat{\mathcal{O}}_{{X},{x}}\ar[r] & \widehat{\mathcal{O}}_{{W},w}\\
    \widehat{\mathcal{O}}_{{Y},{y}} \ar[u] \ar[r] & \widehat{\mathcal{O}}_{{V},v} \ar[u]
    }\end{aligned}
\end{equation}
where the vertical maps are induced by the given morphisms $f$ and $g$, and the horizontal maps are 
isomorphisms induced by the local toric models; 
see \cite[Definition 1.3]{weak:s-stable:reduction}.
We say that the morphism
$f\colon (X, D) \to (Y, E)$ is \emph{toroidal} 
if it is toroidal at every closed point of $X$. 

\begin{lemma}\label{toroidal-morphism-couples}
    Let $(X,D)$ and $(Y,E)$ be couples, and let $f\colon X\to Y$ be a morphism of varieties. 
    If $f\colon (X, D)\to (Y, E)$ is a toroidal morphism, then $f\colon (X, D)\to (Y, E)$
    is a morphism of couples, that is, $f^{-1}(E)\subseteq D$; see \S\ref{morphism-couple-defn}.
\end{lemma}

\begin{proof}
    The result is clear if $(X, D)$ and $(Y, E)$ are toric couples and if $X\to Y$
    is a toric morphism of toric varieties.  Then the result for general toroidal morphisms
    follows immediately from the diagram~\eqref{formal-defn-diagram}.
\end{proof}

Let $f\colon (X, D)\to (Y, E)$ be a toroidal morphism of couples, 
and let $U_X \coloneqq X\setminus \supp D$ and $U_Y \coloneqq Y\setminus\supp E$.
By Lemma~\ref{toroidal-morphism-couples}, $(U_X\subset X)\to (U_Y\subset Y)$
is a well-defined morphism of embeddings (see \S\ref{embedding-morphism-defn}), 
which we also denote it by $f$. 
Hence we call $f\colon (U_X\subset X) \to (U_Y\subset Y)$
a \emph{toroidal morphism of toroidal embedding}.

Similar to toroidal embeddings (see Theorem~\ref{general-chart}), toroidal morphisms
of toroidal embeddings also admit \'etale-toric charts, 
which will be treated in \S\ref{etale-charts-toroidal-morphisms-section}.


\section{Logarithmic geometry}\label{log-geo-section}

In this section, we fix our conventions about logarithmic geometry.
We follow standard literature for logarithmic geometry, such as
\cite{Kato89, Kato-toric, Ogus-log-geo, gabber2018foundationsringtheory}.
When studying logarithmic schemes in \'etale topology, we mainly refer to \cite{Ogus-log-geo}.
While \cite{gabber2018foundationsringtheory} also treats logarithmic schemes in
the usual Zariski topology, hence in this case, we refer to \cite{gabber2018foundationsringtheory}.

\subsection{Small \'etale sites and localisations}\label{et-Zar}
Let $X$ be a scheme.
We denote by $\Et (X)$ (respectively, by $\zar(X)$) the \emph{small \'etale site} of $X$
(respectively, the \emph{Zariski site} of $X$),
and by $X_{\et}$ (respectively, by $X_{\zar}$) 
the topos of sheaves on $\Et (X)$ (respectively, on $\zar(X)$); cf. \cite[Chap. 2]{Olsson-book}.
A \emph{geometric point of $X$} (or a \emph{geometric point of $\Et(X)$}) refers to 
a morphism $\spec k\to X$ of schemes from the spectrum of an algebraically closed field $k$.
If $x$ is a scheme-theoretic point of $X$, we
denote by $\ol{x}$ the geometric point $\spec \ol{k(x)}\to X$,
where $k(x)$ is the residue field of $X$ at $x$.

For a sheaf $\mc{F}$ on $\Et (X)$, denote by $\mc{F}^{\zar}$
the image of $\mc{F}$ via the projection 
\[ \varepsilon\colon X_{\et}\to X_{\zar} \]
from the \'etale to the Zariski topos.
Let $\ol{x}$ be a geometric point of $X$. 
We denote by $\mc{F}_{\ol{x}}$ the localisation of $\mc{F}$ at $\ol{x}$ in the \'etale topology,
and by $\mc{F}_x$ (more precisely, $\mc{F}^{\zar}_x$)
the localisation of $\mc{F}^{\zar}$ at $x$ in Zariski topology.
Hence $\mc{O}_{X,x}$ is the usual localisation of $X$ at $x$, while
$\mc{O}_{X, \ol{x}}$ (more precisely, $\mc{O}_{\Et (X), \ol{x}}$) is the strict henselisation of $\mc{O}_{X,x}$;
see \cite[\S I.4]{etale-cohomology}.


\subsection{Sheaf of Grothendieck groups}\label{sheaf-gp}

For the general theory of \emph{monoids}, we refer readers to \cite[Chap. I]{Ogus-log-geo}.
Let $X$ be a scheme and $\mcM$ a sheaf of monoids on $\Et(X)$.  
Then the \emph{sheaf of Grothendieck groups of $\mcM$} is the sheaf $\mcM^{\gp}$ associated to the presheaf
\[ U\mapsto \mc{M}(U)^{\gp}, \]
where $U\to X$ is an \'etale neighbourhood, and $\mcM(U)^{\gp}$ is the Grothendieck group of $\mcM(U)$; 
see \cite[page 10]{Ogus-log-geo} and \cite[\S 4.8.35]{gabber2018foundationsringtheory}.  
The sheaf of Grothendieck groups of a sheaf of monoids on 
the Zariski site $\zar(X)$ is defined in the same way.

Let $\mcM$ be a sheaf of monoids on $\Et(X)$ (respectively, on $\zar(X)$), then there is a canonical morphism
of sheaves of monoids
\[ \lambda_{\mcM}\colon \mcM\to \mcM^{\gp}. \]
For $U\to X$ in $\Et(X)$ (respectively, in $\zar(X)$), 
the canonical homomorphism $\mcM(U)^{\gp}\to \mcM^{\gp}(U)$ is usually not an isomorphism;
see the proof of \cite[II.1.1.3]{Ogus-log-geo}.


\subsection{Log-structures}

Let $X$ be a scheme.  A \emph{prelog-structure} on $\Et(X)$ (respectively, on $\zar(X)$) 
is a homomorphism of sheaves of monoids
\[ \alpha\colon \mcM\to \mcO_{\Et(X)} \,\text{ (respectively, } \alpha\colon \mcM \to \mc{O}_X), \]
where the monoid structure on $\mcO_{\Et(X)}$ (respectively, on $\mcO_X$)
is given by the multiplications of local sections.
A \emph{log-structure} on $\Et(X)$ (respectively, on $\zar(X)$) is a prelog-structure $\alpha$ as above such that
\[ \alpha^{-1}(\mc{O}_{\Et(X)}^*) \to \mc{O}_{\Et (X)}^* \,\,(\text{respectively, } \alpha^{-1}(\mcO_X^*)\to \mcO_X^*)\]
is an isomorphism; see \cite[\S III.1.1]{Ogus-log-geo}.
The log-structure is usually denoted by $\mc{M}$ without mentioning the homomorphism $\alpha$,
which should not lead to confusions in the context.  
If the topology of $X$ (that is, $\Et(X)$ or $\zar(X)$) 
is clear from the context, we just say that $\mcM$ is a log-structure on $X$.
Moreover, for a scheme $X$, the inclusion 
\[ \mc{O}_{\Et (X)}^* \hookrightarrow \mc{O}_{\Et (X)} \,\text{ (respectively, }
\mcO_X^*\hookrightarrow \mcO_X) \]
is called the \emph{trivial log-structure} on $\Et(X)$ (respectively, on $\zar(X)$).  
We always give $\spec \mbb K$ the trivial log-structure.


\subsection{Log-schemes and log-morphisms}\label{log-scheme-defn}

A \emph{log-scheme} $(X, \mc{M})$ is a scheme $X$ endowed with a log-structure
$\mc{M}$ on $\Et(X)$ or $\zar(X)$; see \cite[\S III.1.2]{Ogus-log-geo} and \cite[\S 12.2]{gabber2018foundationsringtheory}.
If the underlying scheme $X$ of a log-scheme $(X, \mc{M})$
is variety, we call $(X, \mc{M})$ a \emph{log-variety}.
For the definition of \emph{integral} or \emph{saturated} log-structures, see \cite[page 185]{Ogus-log-geo}
(in particular, \cite[II.1.1.3]{Ogus-log-geo}).

For definition of \emph{morphisms of log-schemes}, see \cite[\S III.1.2]{Ogus-log-geo}.  
When emphasising that the morphism belongs to the category of log-schemes, 
we also call a morphism of log-schemes a \emph{log-morphism}.
If 
\[ f\colon (X, \mc{M}_X)\to (Y, \mc{M}_Y) \] 
is a log-morphism of log-schemes,
we usually denote by 
\[ \ul{f}\colon X\to Y \] 
the morphism of the underlying schemes.

\begin{definitionnotation}
    Denote by 
    \[ \logsch_{\et} \,\,(\text{respectively, } \logsch_{\zar}) \] 
    the category of log-schemes whose log-structures are defined on small \'etale sites 
    (respectively, on Zariski sites). The categories $\logsch_{\et}$ and $\logsch_{\zar}$ admit fibre products; 
    see \cite[III.2.1.2]{Ogus-log-geo}.
\end{definitionnotation}

\begin{definitionnotation}
    Recall that algebraic schemes are defined over the algebraically closed field $\mbb K$ of characteristic zero.
    We always equip $\spec \mbb K$ with the trivial log-structure.
    Denote by
    \[ \loget \,\,(\text{respectively, }  \logzar) \]
    the category of log-schemes $(X, \mcM)$, where $X$ is an algebraic scheme over $\mbb K$, and
    $\mcM$ is a log-structure on $\Et(X)$ (respectively, on $\zar(X)$).
\end{definitionnotation}


\subsection{Zariski log-structures}\label{zariski-log-structures-defn-sec}

For \emph{direct image} and \emph{inverse image} of log-structures, see \cite[page 272]{Ogus-log-geo}
and \cite[\S 12.1.6]{gabber2018foundationsringtheory}.
Let $X$ be a scheme, and let $\varepsilon\colon X_{\et}\to X_{\zar}$ be the canonical projection of topoi;
see \S\ref{et-Zar}.  Let $\mcZ$ be a log-structure on $\zar(X)$. 
For every \'etale morphism $\pi\colon X'\to X$, let
\[ \mcM_{X'} \coloneqq \pi^*_{\log}(\mcZ) \]
be the inverse image of the log-structure $\mcZ$ on $\zar(X')$.  Then the assignment 
\[ X'\mapsto \mcM_{X'}(X') \]
defines a sheaf $\mcM$ on $\Et(X)$, and $\mcM\to \mcO_{\Et(X)}$ is a log-structure; 
see \cite[III.1.4.1 (1)]{Ogus-log-geo}.  We denote this log-structure on $\Et(X)$ by 
\[ \varepsilon^*_{\log}(\mcZ). \]
Note that $\varepsilon^*_{\log}(\mcZ)$ is denote by $\varepsilon^*(\mcZ)$ 
on \cite[page 3]{niziol06} and called the 
\emph{pullback of $\mcZ$}.

\begin{definition}[cf. \protect{\cite[page 3]{niziol06}}]\label{zar-log-structure}
    Let $(X, \mcM)$ be a log-scheme, where $\mcM$ is a log-structure on $\Et(X)$.
    We say that $\mcM$ is \emph{Zariski} if $\mcM = \varepsilon^*_{\log}(\mcM^{\zar})$,
    where $\mcM^{\zar}$ is the image of $\mcM$ via $\varepsilon\colon X_{\et}\to X_{\zar}$,
    that is, the restriction of $\mcM$ to $\zar(X)$.
\end{definition}


\subsection{Log-scheme associated to monoids}\label{log-monoids}

Let $P$ be a monoid and $e\colon P\to \mbb{K}[P]$ the monoid algebra; see \cite[\S I.3.1]{Ogus-log-geo}.
Denote by $\mathsf{A}_{\mbb K, P}$ (or $\mathsf{A}_P$ for simplicity) the log-scheme 
whose underlying scheme is $\AAA_P \coloneqq \spec (\mbb{K}[P])$
and the log-structure $\mc{M}_P$ on $\Et(\AAA_P)$ is induced by the homomorphism $P\hookrightarrow \mbb{K}[P]$;
in other words,
\[ \mathsf{A}_P = (\AAA_P, \mcM_P), \]
which is also denote by 
\[ \mathsf{A}_P \coloneqq \spec (P\to \mbb K[P]) \]
on \cite[page 275]{Ogus-log-geo}.
For more details of this construction, we refer readers to \cite[(1.5)]{Kato89},
\cite[\S III.1.2]{Ogus-log-geo} and \cite[\S 12.1.15]{gabber2018foundationsringtheory}.
We denote by $\mathsf{A}_P^{\zar}$ the log-scheme $(\AAA_P, \mcZ_P)$, 
where $\mcZ_P$ is the log-structure on $\zar(\AAA_P)$ induced by the homomorphism $P\hookrightarrow \mbb K[P]$.
Then $\mcZ_P$ is the restriction of $\mcM_P$ to $\zar(\AAA_P)$.
Moreover, by \cite[\S 12.2.2]{gabber2018foundationsringtheory},
\[ \mcM_P = \varepsilon^*_{\log}(\mcZ_P), \]
where $\varepsilon\colon \AAA_{P,\et}\to \AAA_{P,\zar}$ is the canonical projection of topoi.

Let $P$ be a fine monoid, and let $\mathsf{A}_P^*$ be the maximal open subset of $\AAA_P$
on which the log-structure of $\mathsf{A}_P$ is trivial; cf. \cite[III.1.2.8]{Ogus-log-geo}.
Then by \cite[page 280]{Ogus-log-geo}, $\mathsf{A}_P^*$ is the underlying scheme 
of $\mathsf{A}_{P^{\gp}}$, that is, the scheme $\AAA_{P^{\gp}} = \spec \mbb K[P^{\gp}]$.


\subsection{Charts for log-schemes and log-morphisms}
For \emph{charts of log-schemes}, see \cite[\S II.2.1, \S III.1.2]{Ogus-log-geo}.
For \emph{charts of log-morphisms}, see \cite[III.1.2.6]{Ogus-log-geo}.
Note that \cite[\S 12.1]{gabber2018foundationsringtheory} 
treats uniformly charts of log-schemes and log-morphisms
for small \'etale sites and Zariski sites.

For the definition of \emph{coherent} log-structures, see \cite[II.2.1.5]{Ogus-log-geo} and \cite[(2.1)]{Kato89}
for small \'etale sites; see \cite[12.1.17]{gabber2018foundationsringtheory} for both
small \'etale sites and Zariski sites.
Note that coherent log-structures are defined by using charts.
A log-structure is called \emph{fine} if it is coherent and integral; see \S\ref{log-scheme-defn}.

If a log-structure $\mcM$ on $\Et(X)$ or $\zar(X)$ for a scheme $X$ is coherent 
(respectively, integral, respectively, fine), we say that
$(X, \mcM)$ is a \emph{coherent log-scheme} (respectively, an \emph{integral log-scheme}, 
respectively, a \emph{fine log-scheme}).
If a log-structure $\mcM$ on $\Et(X)$ or $\zar(X)$ for a scheme $X$ is fine and saturated, we say that
$\mcM$ is \emph{fs} or $(X, \mcM)$ is an \emph{fs log-scheme}.

For fine log-structures on small \'etale sites, charts exist \'etale-locally, and they can be constructed 
by applying \cite[Lemma (2.10)]{Kato89}.
For fs log-structures on Zariski sites, charts exist Zariski-locally, and such charts can be
constructed by \cite[(1.6)]{Kato-toric}.
In general, \cite[12.1.35]{gabber2018foundationsringtheory} treats the existence of local charts
for fine log-structures on small \'etale sites and Zariski sites in a uniform way.

\begin{lemma}\label{trivial-fibre-product}
    Let $(X, \mcM_X), (Y, \mcM_Y), (Z, \mcM_Z)$ be log-schemes admitting log-morphisms
    \[ (X, \mcM_X)\to (Z, \mcM_Z)\,\text{ and }\,(Y, \mcM_Y)\to (Z, \mcM_Z), \]
    where the log-structures are defined on either small \'etale sites or Zariski sites.  Let
    \[ (W, \mcM_W) \coloneqq (X, \mcM_X)\times_{(Z, \mcM_Z)}(Y, \mcM_Y) \]
    be the fibre product in the category of log-schemes.  If the log-structures $\mcM_X, \mcM_Y, \mcM_Z$ are trivial,
    then $\mcM_W$ is also the trivial log-structure.
\end{lemma}

\begin{proof}
    Let $\mb{e}$ be the trivial monoid, that is, the monoid consisting of a single element.
    Then $\mb e\to \mcO_{\Et(X)}$ is a chart for $\mcM_X$;
    similarly, $\mb e\to \mcO_{\Et(Y)}, \mb e \to \mcO_{\Et(Z)}$ are charts for $\mcM_Y, \mcM_Z$.
    It is evident that the induced $\mb e\to \mcO_{\Et(W)}$ is a chart for $\mcM_W$, hence $\mcM_W$ is also trivial.
    The proof for Zariski sites is the same.
\end{proof}


\begin{lemma}\label{formal-log-regular-schemes}
    Let $(X, \mcM)$ be an fs log-scheme in $\logzar$.
    Let $x\in X$ and $\wh{X}\coloneqq \spec \whtOO_{X,x}$.
    Denote by $\rho\colon \wh{X}\to X$ the canonical morphism, and by $\wh{\mcM} \coloneqq \rho^*_{\log}\mcM$.
    Then $(\wh{X}, \wh{\mcM})$ is an fs log-scheme in $\logsch_{\zar}$.
\end{lemma}

\begin{proof}
    By \cite[(1.6)]{Kato-toric}, there exist a Zariski open neighbourhood $U$ of $x$, a sharp fs monoid $P$ and
    a homomorphism of monoids $\varphi\colon P\to \mcM|_U$,
    which is a chart for $\mcM$ on $U$.
    Note that $\rho$ factors through the open immersion $U\hookrightarrow X$.
    As the log-morphisms $(\wh{X}, \wh{\mcM})\to (U, \mcM|_U)$ 
    and $(U, \mcM|_U)\to \mathsf{A}_P$ are strict,
    we see that the induced homomorphism $P\to \wh{\mcO}_{X,x}$ is a chart for $\wh{\mcM}$.
    By \cite[II.2.3.6]{Ogus-log-geo}, $\wh{\mcM}$ is an fs log-structure on $\zar(\wh{X})$.
\end{proof}


\subsection{Comparison of charts in \'etale and Zariski sites}
The following result shows that an integral log-structure admitting a global chart must be Zariski 
(see Definition~\ref{zar-log-structure}),
hence we can always assume such log-structures are Zariski after passing to some \'etale neighbourhoods.

\begin{lemma}\label{et-to-zar-charts}
    Let $(X, \mcM)$ be an integral log-scheme, where $\mcM$ is a log-structure on $\Et(X)$.
    Let $a\colon (X, \mcM)\to \mathsf{A}_P$ be a chart for $(X, \mcM)$ subordinate to an integral monoid $P$.
    Then the log-structure $\mcM$ is Zariski.  Moreover, the restriction of $a$ to $\zar(X)$,
    \[ a^{\zar}\colon (X, \mcM^{\zar})\to \mathsf{A}_P^{\zar}, \]
    is a chart for $(X, \mcM^{\zar})$ subordinate to $P$.
\end{lemma}

\begin{proof}
    This is \cite[III.1.4.1 (2)]{Ogus-log-geo}.  
    This is also proved for fs log-schemes by \cite[Lemma 2.4]{niziol06} (and its proof).
\end{proof}

Conversely, a chart for an integral log-structure on Zariski site is also
a chart for the induced log-structure on small \'etale site.

\begin{lemma}\label{zar-charts-to-et-charts}
    Let $(X, \mcZ)$ be an integral log-scheme, where $\mcZ$ is an integral log-structure on $\zar(X)$.
    Assume that $p\colon P\to \mcZ(X)$ is a chart for $(X, \mcZ)$ subordinate to an integral monoid $P$.
    Then the homomorphism $p\colon P\to \mcZ(X)$ is a also chart for 
    $\varepsilon^*_{\log}(\mcZ)$ (see \S\ref{zariski-log-structures-defn-sec}),
    where $\varepsilon\colon X_{\et}\to X_{\zar}$ is the canonical projection of topoi.
\end{lemma}

\begin{proof}
    This is \cite[III.1.1.4, III.1.4.1 (1)]{Ogus-log-geo}.   
\end{proof}

Moreover, the following result shows that many properties of \emph{Zariski} log-structures are preserved by 
restricting to the Zariski site.

\begin{lemma}[=\protect{\cite[12.2.22]{gabber2018foundationsringtheory}}]\label{preserved-properties-zar-to-et}
    Let $(X, \mcM)$ be a log-scheme such that the log-structure $\mcM$ on $\Et(X)$ is Zariski, 
    and let $\mb P$ be one of the following properties:
    \[ \text{integral, coherent, fine, and fs.} \]
    Then the induced log-structure $\mcM^{\zar}$ on $\zar(X)$ satisfies $\mb P$ 
    if and only if $\mcM$ satisfies $\mb P$.
\end{lemma}


\subsection{Compactifying log-structures}\label{compact-log-struc-defn}
Let $U$ be a nonempty Zariski open subset of a scheme $X$,
and let $j\colon U\to X$ be the open immersion.  The direct image log-structure
\[ \alpha_{U/X}\colon \mc{M}_{U/X}\coloneqq j_*^{\log}(\mc{O}_{\Et (U)}^*) \to \mc{O}_{\Et (X)} \]
is called the \emph{compactifying log-structure} associated to the open immersion $j$.
The sheaf $\mc{M}_{U/X}$ is just the sheaf of local sections of $\mc{O}_{\Et (X)}$ 
whose restrictions to $\Et (U)$
are invertible; see \cite[\S III.1.6]{Ogus-log-geo}.
The log-structure $\alpha_{U/X}$ is the trivial log-structure if $U = X$.
We can also associate 
the open immersion $j\colon U\to X$ the \emph{Zariski compactifying log-structure}, that is,
the sheaf of local sections of $\mc{O}_X$ whose restrictions to $U$ are invertible.

Given an open immersion $U\to X$, we always denote by
$\mc{M}_{U/X}$ the compactifying log-structure on $\Et(X)$, and by
$\mc{M}_{U/X}^{\zar}$ the Zariski compactifying log-structure.
It is evident that $\mc{M}_{U/X}^{\zar}$ is the image 
of $\mc{M}_{U/X}$ via the projection $X_{\et} \to X_{\zar}$; see \S\ref{et-Zar}.


\subsection{Log-structure associated to toroidal embeddings}

Let $(U\subset X)$ be a toroidal embedding.
By saying the \emph{log-structure associated to $(U\subset X)$},
we mean the compactifying log-structure $\mc{M}_{U/X}$ on $\Et (X)$.
On the other hand, let $(U_X\subset X)$ and $(U_Y\subset Y)$ be toroidal embeddings,
and let $\ul{f}\colon X\to Y$ be a morphism of varieties
such that $\ul{f}(U_X)$ is contained in $U_Y$.
By \cite[III.1.6.2]{Ogus-log-geo}, there is a uniquely determined 
log-morphism of log-varieties $f\colon (X, \mc{M}_{U_X/X}) \to (Y, \mc{M}_{U_Y/Y})$
corresponding to $\ul{f}$.  By slightly abusing of terminologies, we also say that the morphism 
of embeddings $\ul{f}\colon (U_X\subset X)\to (U_Y\subset Y)$ (see \S\ref{embedding-morphism-defn})
is a log-morphism of log-varieties.


\subsection{Log-regularity}

Let $X$ be a scheme and $\mc{M}$ a log-structure on $\Et(X)$ such that
$(X, \mc{M})$ is an fs log-scheme.  
For \emph{log-regularity} of $(X, \mc{M})$, we refer readers to
\cite[Definition 2.2]{niziol06}, \cite[12.5.23]{gabber2018foundationsringtheory} and
\cite[\S III.1.11]{Ogus-log-geo}.
Moreover, if $\mc{N}$ is an fs log-structure on $\zar(X)$,
then the \emph{log-regularity} of $(X, \mc{N})$ is defined in \cite[(2.1)]{Kato-toric}.

For an algebraic scheme $X$ (see \S\ref{schemes-varieties-defn}),
we say that $(X, \mcZ)$ is log-regular in $\logzar$ if $\mcZ$ is a log-structure on $\zar(X)$
and if $(X, \mcZ)$ is log-regular in the sense of \cite[(2.1)]{Kato-toric};
we say that $(X, \mcM)$ is log-regular in $\loget$ if $\mcM$ is a log-structure on $\Et(X)$
and if $(X, \mcM)$ is log-regular in the sense of \cite[Definition 2.2]{niziol06}.

Let $\mc{M}$ be an fs log-structure on $\Et(X)$ for a scheme $X$, 
and let $\mc{M}^{\zar}$ be its image via the canonical projection $X_{\et}\to X_{\zar}$; see \S\ref{et-Zar}.
Then by \cite[Lemmas 2.3, 2.4]{niziol06},
if $(X, \mc{M})$ admits a global chart, it is log-regular in the sense of \cite[Definition 2.2]{niziol06}
if and only if $(X, \mc{M}^{\zar})$ is log-regular in the sense of \cite[(2.1)]{Kato-toric}.
More generally, for Zariski log-structures, we have the following result from \cite{gabber2018foundationsringtheory}.

\begin{lemma}[cf. \protect{\cite[12.5.37]{gabber2018foundationsringtheory}}]\label{log-regular-two-tops}
    Let $X$ be a scheme and $\mcM$ an fs Zariski log-structure on $\Et(X)$.
    Then $(X, \mcM^{\zar})$ is log-regular in the sense of \cite[(2.1)]{Kato-toric}
    if and only if $(X, \mcM)$ is log-regular in the sense of \cite[Definition 2.2]{niziol06}.
\end{lemma}

Because of Lemma~\ref{log-regular-two-tops},
if the Grothendieck topology of $X$ for the log-structure 
$\mc{M}$ is clear from the context, we can just say that the log-scheme $(X, \mc{M})$ is log-regular
without specifying that we are using \cite[Definition 2.2]{niziol06} or \cite[(2.1)]{Kato-toric}.

For both the small \'etale sites and Zariski sites,
the following result shows that
log-regularity is an \'etale-local property.

\begin{lemma}\label{log-regular-etale-local}
    Let $(X, \mcM)$ be an fs log-scheme, where $\mcM$ is a log-structure on $\Et(X)$
    (respectively, on $\zar(X)$).  Let $f\colon Y\to X$ be a surjective \'etale morphism of schemes, and
    let $\mcN$ be the inverse image $f^*_{\log}\mcM$ of the log-structure $\mcM$ to $\Et(X)$
    (respectively, to $\zar(X)$).  Then $(X, \mcM)$ is log-regular 
    if and only if $(Y, \mcN)$ is log-regular.
\end{lemma}

\begin{proof}
    Assume that $\mcM$ is a log-structure on $\Et(X)$.  
    Note that $f^*_{\log}\mcM$ is the log-structure associated to the prelog-structure
    (see \cite[III.1.1.5]{Ogus-log-geo})
    \[ f^{-1}\mcM \to f^{-1}(\mcO_{\Et(X)}) \to \mcO_{\Et(Y)}. \]
    As $f$ is \'etale, $f^{-1}\mcM$ is just the restriction
    of $\mcM$ to $\Et(Y)$ by \cite[\S II.2, II.3.1 (a), page 68]{etale-cohomology},
    hence $f^{-1}\mcM \to \mcO_{\Et(Y)}$ is already a log-structure.
    Thus, $\mcN = f^{-1}\mcM = \mcM|_{\Et(Y)}$.
    Then in this case, the result follows immediately from \cite[Definition 2.2]{niziol06}.

    Now assume that $\mcM$ is a log-structure on $\zar(X)$.
    Let $\varepsilon\colon X_{\et}\to X_{\zar}$ be the natural projection of topoi
    (see \S\ref{et-Zar}), and let
    (see \S\ref{zariski-log-structures-defn-sec})
    \[ \wt{\mcM} \coloneqq \varepsilon^*_{\log} \mcM. \]
    Note that by Lemma~\ref{preserved-properties-zar-to-et}, $(X, \wt{\mcM})$
    is also an fs log-scheme.
    Then by Lemma~\ref{log-regular-two-tops}, $(X, \wt{\mcM})$ is log-regular
    if and only if $(X, \mcM)$ is log-regular.  Recall that by
    the construction in \S\ref{zariski-log-structures-defn-sec},
    the restriction of $\wt{\mcM}$ to $\zar(Y)$ is just the log-structure $\mcN$.
    Moreover, as $f$ is \'etale, the restriction of $\wt{\mcM}$ to $\Et(Y)$ is
    the inverse image $\wt{\mcN}$ of $\mcN$ from $\zar(Y)$ to $\Et(Y)$.
    Again by Lemma~\ref{log-regular-two-tops}, $(Y, \wt{\mcN})$ is log-regular if and only if
    $(Y, \mcN)$ is log-regular.  By the result of the \'etale case as above,
    $(X, \wt{\mcM})$ is log-regular if and only if $(Y, \wt{\mcN})$ is log-regular.
    Thus, we can conclude that $(X, \mcM)$ is log-regular if and only if $(Y, \mcN)$ is log-regular.
\end{proof}


\subsection{Log-regularity and compactifying log-structures}\label{log-regular-compactifying-section}

The log-structure of a log-regular log-scheme $(X, \mc{M})$ is compactifying;
indeed, the canonical homomorphism $\mc{M}_{X^*/X}\to \mc{M}$ (see \cite[III.1.6.2]{Ogus-log-geo}) is an isomorphism, where
\[ X^* \coloneqq \Set{ x\in X\mid \mc{M}_{\ol{x}} = \mc{O}_{X, \ol{x}}^* } \]
is the maximal Zariski open subset of $X$ on which the log-structure $\mc{M}$ is trivial.
This is proved for log-structures on Zariski sites in \cite[(11.6)]{Kato-toric}, 
and for general log-schemes in \cite[Proposition 2.6]{niziol06} and \cite[III.1.11.12]{Ogus-log-geo};
see also \cite[12.5.54]{gabber2018foundationsringtheory} for a uniform treatment
of the two topologies.

Log-regularity of a log scheme also implies nice singularities of the underlying scheme as 
the following result shows.

\begin{lemma}[=\protect{\cite[12.5.29]{gabber2018foundationsringtheory}}]\label{log-regular-normal}
    Let $(X, \mcM)$ be a log-regular log-scheme, where $\mcM$ is a log-structure on $\Et(X)$ or $\zar(X)$.
    Then the underlying scheme $X$ is normal and Cohen-Macaulay.
\end{lemma}


\subsection{Log-smooth and log-\'etale log-morphisms}

For the notions of \emph{log-smooth} and \emph{log-\'etale} log-morphisms between log-schemes,
we refer readers to \cite[\S IV.3]{Ogus-log-geo} and \cite[\S 12.3]{gabber2018foundationsringtheory}.
When saying an \emph{\'etale morphism} (or that \emph{a morphism is \'etale}),
we mean an \'etale morphism of schemes in the usual sense, not a log-\'etale log-morphism.

The following result is proved in \cite[(8.3)]{Kato-toric} for log-structures on Zariski sites,
and in \cite[IV.3.5.1]{Ogus-log-geo} for log-structures on small \'etale sites,
which says that a log-variety is log-regular if and only if the canonical log-morphism to $\spec \mbb K$ is log-smooth.

\begin{lemma}\label{log-regular-equi-smooth}
    Let $(X, \mcM)$ be a log-scheme, where $X$ is an algebraic scheme (see \S\ref{schemes-varieties-defn})
    and $\mcM$ is an fs log-structure 
    on $\Et(X)$ (respectively, on $\zar(X)$).  Then $(X, \mcM)$ is log-regular in $\loget$ 
    (respectively, in $\logzar$) if and only if 
    $(X, \mcM)\to \spec \mbb K$ is log-smooth in $\loget$ (respectively, in $\logzar$).
\end{lemma}

By the charts description of log-smooth log-morphisms
(see \cite[Proposition (3.4)]{Kato89}, \cite[IV.3.1.8]{Ogus-log-geo} and \cite[12.3.34]{gabber2018foundationsringtheory}), 
Lemma~\ref{log-regular-equi-smooth} implies that toric varieties are log-regular as follows.
This is very well-know; we include a sketchy proof for clarification.

\begin{lemma}\label{toric-log-regular}
    Let $W_{\sigma}$ be a normal affine toric variety given by $(N,\sigma)$; see \S\ref{toric-couples-defn}.  Let
    $M_{\sigma} \coloneqq \sigma^{\vee}\cap M$, where $M$ is the dual lattice of $N$, 
    and $\sigma^{\vee}$ is the dual cone of $\sigma$ in $M_{\RR} \coloneqq M\otimes_{\ZZ}\RR$.  
    Denote by $\mbb T_{\sigma}$ the torus of $W_{\sigma}$.
    Then 
    \begin{itemize}
        \item [\emph{(1)}] the log-scheme (see \S\ref{log-monoids})
    \[ \mathsf{A}_{M_{\sigma}} \coloneqq \spec (M_{\sigma} \to \mbb{K} [M_{\sigma}]) \]
    is isomorphic to the log-scheme $(W_{\sigma}, \mc{M}_{\mbb{T}_{\sigma}/ W_{\sigma}})$ (see \S\ref{compact-log-struc-defn}), and
        \item [\emph{(2)}] the log-scheme $\mathsf{A}_{M_{\sigma}}$ is a log-regular log-variety in $\loget$.
    \end{itemize}
    Moreover, let $\mc{M}_{\mbb T_{\sigma}/W_{\sigma}}^{\zar}$ be the Zariski compactifying log-structure
    associated to the open embedding $\mbb T_{\sigma}\hookrightarrow W_{\sigma}$ (see \S\ref{compact-log-struc-defn}), then
    $(X, \mc{M}_{\mbb T_{\sigma}/W_{\sigma}}^{\zar})$ is log-regular in $\logzar$.
\end{lemma}

\begin{proof}
    First note that the underlying scheme $\AAA_{M_{\sigma}}$ of $\mathsf{A}_{M_{\sigma}}$ is just $W_{\sigma}$.
    As $W_{\sigma}$ is a normal toric variety, by \cite[Theorem 1.3.5]{CLS:toric},
    $M_{\sigma}$ is an fs monoid with $M_{\sigma}^{\gp}$ torsion free.
    Moreover, by \cite[Proposition 1.1.14]{CLS:toric}, the character lattice of the torus $\mbb{T}_{\sigma}$
    is $M_{\sigma}^{\gp} \cong M$, so the open subscheme
    $\spec (\mbb{K}[M_{\sigma}^{\gp}])$ of $\AAA_{M_{\sigma}}$ 
    is isomorphic to $\mbb{T}_{\sigma}\cong \mbb{G}_{\mbb K}^d$, 
    where $d \coloneqq \rank M = \dim W_{\sigma}$.  
    By \cite[IV.3.1.9]{Ogus-log-geo}, $\mathsf{A}_{M_{\sigma}}\to \spec \mbb K$ is log-smooth
    in $\loget$, hence by Lemma~\ref{log-regular-equi-smooth},
    $\mathsf{A}_{M_{\sigma}}$ is a log-regular log-variety in $\loget$.  
    In addition, by \cite[III.1.2.10]{Ogus-log-geo}, $\spec (\mbb{K}[M_{\sigma}^{\gp}])$
    is the maximal open subscheme of $\AAA_{M_{\sigma}}$ on which the log-structure
    of $\mathsf{A}_{M_{\sigma}}$ is trivial.  Thus, \cite[III.1.11.12]{Ogus-log-geo} 
    shows that the log-structure of $\mathsf{A}_{M_{\sigma}}$
    is isomorphic to the compactifying log-structure $\mc{M}_{\mbb{T}_{\sigma}/W_{\sigma}}$.  
    This proves (1) and (2).  This also shows that $M_{\sigma}\to \mbb K[M_{\sigma}]$
    is a global chart for the log-variety $(W_{\sigma}, \mc{M}_{\mbb{T}_{\sigma}/ W_{\sigma}})$.
    Hence $(X, \mc{M}_{\mbb T_{\sigma}/W_{\sigma}}^{\zar})$ is log-regular 
    in $\logzar$ by Lemmas~\ref{et-to-zar-charts}, \ref{preserved-properties-zar-to-et} 
    and \ref{log-regular-two-tops}.
\end{proof}


The following result shows that for Zariski log-structures, log-smoothness
of log-morphisms on small \'etale sites is equivalent to that on Zariski sites.

\begin{lemma}[cf. \protect{\cite[12.3.27]{gabber2018foundationsringtheory}}]\label{log-sm-et-zar}
    Let $f\colon (X, \mc{M}_X)\to (Y, \mcM_Y)$ be a log-morphism of integral log-schemes
    on small \'etale sites.
    Assume that the log-structures $\mc{M}_X, \mcM_Y$ are Zariski, that is,
    $\mcM_X, \mcM_Y$ are the inverse images of of $\mcM_X^{\zar}, \mcM_Y^{\zar}$ 
    respectively; see Definition~\ref{zar-log-structure}.  Let
    \[ f^{\zar}\colon (X, \mcM_X^{\zar}) \to (Y, \mcM_Y^{\zar}) \]
    be the induced log-morphism of log-schemes on Zariski sites.  Then
    $f$ is log-smooth if and only if $f^{\zar}$ is log-smooth.
\end{lemma}


\section{Log-regularity and toroidal embeddings}\label{log-reg-toroidal-equi-section}

In this section, we show that log-regularity of a log-variety $(X, \mcM)$
is equivalent to giving a toroidal structure on $X$.

\subsection{\'Etale-local charts for toroidal embeddings}
First we show that if $(U\subset X)$ is a toroidal embedding, then
every \'etale-toric chart for $(U\subset X)$ (see Theorem~\ref{general-chart}) gives an \'etale-local chart 
for the log-scheme $(X, \mc{M}_{U/X})$.

\begin{lemma}\label{etale-pullback-chart}
    Let $\ul{f}\colon X\to Y$ be an \'etale morphism of schemes.
    Let $U\subseteq X$ and $V\subseteq Y$ be nonempty Zariski open subsets of $X$ and $Y$ respectively such that 
    \[ U = \ul{f}^{-1}(V). \]
    Then the induced morphism of log-schemes
    \[ f\colon (X, \mc{M}_{U/X})\to (Y, \mc{M}_{V/Y}) \]
    is strict, that is, the induced homomorphism
    \[ \ul{f}^*_{\log}(\mc{M}_{V/Y})\to \mc{M}_{U/X} \]
    of log-structures is an isomorphism.
\end{lemma}

\begin{proof}
    As $U = \ul{f}^{-1}(V)$, by \cite[III.1.6.2]{Ogus-log-geo}, there is a canonical log-morphism of 
    log-schemes $f\colon (X, \mc{M}_{U/X})\to (Y, \mc{M}_{V/Y})$ induced by $\ul{f}$.
    Recall that $\ul{f}^*_{\log}(\mc{M}_{V/Y})$
    is the log-structure associated to the prelog-structure (see \cite[III.1.1.5]{Ogus-log-geo})
    \[ \ul{f}^{-1}(\mc{M}_{V/Y}) \to \ul{f}^{-1}(\mc{O}_{\Et(Y)}) \to \mc{O}_{\Et(X)}. \]
    However, as $X\to Y$ is \'etale,
    $\ul{f}^{-1}(\mc{M}_{V/Y})$ is just the restriction of $\mc{M}_{V/Y}$
    to $(X\to Y)$ by \cite[\S II.2, II.3.1 (a), page 68]{etale-cohomology}, 
    so $\mc{M}_{U/X}\cong \ul{f}^{-1}(\mc{M}_{V/Y})$.  Therefore,
    $\ul{f}^{-1}(\mc{M}_{V/Y})\to \mc{O}_{\Et(X)}$ is isomorphic to $\mc{M}_{U/X}\to \mc{O}_{\Et(X)}$.
    In particular, $\ul{f}^{-1}(\mc{M}_{V/Y})\to \mc{O}_{\Et(X)}$ is already a log-structure, hence $f$ is strict.
\end{proof}


\begin{lemma}\label{toroidal-charts}
    Let $(U\subset X)$ be a toroidal embedding.  Let 
    \[\xymatrix{
     & (U'\subset X')\ar[rd]^q\ar[ld]_p & \\
    (U\subset X) & & (\mbb{T}_{\sigma}\subset W_{\sigma})
    }\]
    be an \'etale-toric chart for $(U\subset X)$ (see Theorem~\ref{general-chart}), that is,
    \begin{itemize}
        \item $(U'\subset X')$ is a strict toroidal embedding,
        \item $W_{\sigma}$ is a normal affine toric variety given by $(N, \sigma)$,
        \item $\mbb{T}_{\sigma}$ is the torus of $W_{\sigma}$,
        \item $p\colon X'\to X$ and $q\colon X'\to W_{\sigma}$ are \'etale morphisms of varieties, and
        \item $p^{-1}(U)$ and $q^{-1}(\mbb T_{\sigma})$ are equal to $U'$.
    \end{itemize}
    Let $M_{\sigma} \coloneqq \sigma^{\vee}\cap M$, where $M$ is the dual lattice of $N$, 
    and $\sigma^{\vee}$ is the dual cone of $\sigma$ in $M_{\RR} \coloneqq M\otimes_{\ZZ}\RR$.  
    Let (see \S\ref{log-monoids})
    \[ \mathsf{A}_{M_{\sigma}} \coloneqq \spec (M_{\sigma} \to \mbb{K} [M_{\sigma}]). \]
    Then the log-morphism of log-schemes induced by $q$,
    \[ (X', \mc{M}_{U'/X'}) \to \mathsf{A}_{M_{\sigma}}, \]
    is an \'etale-local chart for the log-scheme $(X, \mc{M}_{U/X})$.
\end{lemma}

\begin{proof}
    By Lemma~\ref{etale-pullback-chart}, the log-morphisms 
    $(X', \mc{M}_{U'/X'}) \to (X,\mcM_{U/X})$ and $(X', \mc{M}_{U'/X'})\to \mathsf{A}_{M_{\sigma}}$
    between log-schemes are strict, hence the strict log-morphism 
    $(X', \mc{M}_{U'/X'})\to \mathsf{A}_{M_{\sigma}}$ is an \'etale-local chart for $(X, \mc{M}_{U/X})$.
\end{proof}


\subsection{Equivalence of log-regularity and toroidal embeddings}
Now we show the equivalence of log-regularity and toroidal embeddings 
on both the small \'etale sites and Zariski sites.
This is outlined in \cite[(1.7), (1.8)]{Kato-toric}.

\begin{theorem}\label{toroidal-to-log-reg}
    Let $(U\subset X)$ be a toroidal embedding, and let $\mc{M}_{U/X}$ 
    be the associated compactifying log-structure on $\Et(X)$.  
    Then $(X, \mc{M}_{U/X})$ is a log-regular log-variety
    in $\loget$.  Moreover, if $(U\subset X)$ is a strict toroidal embedding,
    then $\mcM_{U/X}$ is Zariski, and $(X, \mc{M}_{U/X}^{\zar})$ is log-regular in $\logzar$, where
    $\mc{M}_{U/X}^{\zar}$ is the Zariski compactifying log-structure on $\zar(X)$; see \S\ref{compact-log-struc-defn}.
\end{theorem}

\begin{proof}
    Let $(U\subset X)$ be a toroidal embedding.  
    By Theorem~\ref{general-chart}, there is an \'etale-toric chart 
    for $(U\subset X)$ near any closed point of $X$ as in Lemma~\ref{toroidal-charts}.
    Keep the notations in Lemma~\ref{toroidal-charts}.
    Since $\mathsf{A}_{M_{\sigma}}$ is log-regular by Lemma~\ref{toric-log-regular},
    and since being log-regular is an \'etale-local property by Lemma~\ref{log-regular-etale-local}, 
    we see that $(X, \mc{M}_{U/X})$ is log-regular in $\loget$.
    
    Now assume that $(U\subset X)$ is a strict toroidal embedding.  
    By Theorem~\ref{strict-chart}, it admits a strict \'etale-toric chart $q\colon U_x\to W_{\sigma}$
    near every closed point $x\in X$, 
    that is, the morphism $p$ in Lemma~\ref{toroidal-charts} is the identity map onto $U_x$.
    By Lemma~\ref{et-to-zar-charts}, we see that $\mcM_{U/X}$ is Zariski.
    This also follows from \cite[III.1.6.5]{Ogus-log-geo} as the components of $X\setminus U$ are normal 
    whence geometrically unibranch; see \cite[Chap. 0, (23.2.1)]{EGA-IV-1}.
    Finally, by Lemma~\ref{log-regular-two-tops}, we can deduce that 
    $(X, \mc{M}_{U/X}^{\zar})$ is log-regular in $\logzar$. 
\end{proof}


Recall that log-regular structures are compactifying; see \S\ref{log-regular-compactifying-section}.
The following result shows that such a compactifying log-structure 
must be induced by a toroidal embedding, that is, 
the converse of Theorem~\ref{toroidal-to-log-reg} also holds.

\begin{theorem}\label{toroidal-regular}
    Let $(X, \mc{M})$ be a log-regular log-variety,
    where $\mcM$ is a log-structure on $\Et(X)$ (respectively, on $\zar(X)$).
    Then $(X^*\subset X)$ is a toroidal embedding (respectively, a strict toroidal embedding), 
    where $X^*$ is the maximal Zariski open subset of $X$ on which $\mc{M}$ is trivial.  
\end{theorem}

\begin{proof}
    \emph{Step 1}.
    Let $(X, \mc{M})$ be a log-regular log-variety, where $\mcM$ is a log-structure on $\Et(X)$.
    For an arbitrary geometric point $\ol{x}$ of $X$, 
    since charts of log-structures exist \'etale-locally by \cite[(2.10)]{Kato89},
    let $Y\to X$ be an \'etale neighbourhood of $\ol{x}$
    such that the restriction $\mc{M}_Y$ of $\mc{M}$ to $Y$ admits a global chart.
    Denote by $\mc{M}_Y^{\zar}$ the restriction of $\mc{M}_Y$ to $\zar(Y)$.
    By Lemma~\ref{et-to-zar-charts},
    the log-structure $\mc{M}_Y$ is Zariski, that is, $\mc{M}_Y = \varepsilon^* \mc{M}_Y^{\zar}$,
    where $\varepsilon\colon Y_{\et}\to Y_{\zar}$ is the natural projection of topoi; see \S\ref{et-Zar}. 

    Let $y\in Y$ be an arbitrary closed point of $Y$.
    Denote by $P \coloneqq \mc{M}_{Y,y}^{\zar}/\mc{O}_{Y,y}^*$ the sharpening of $\mc{M}_{Y,y}^{\zar}$.
    By \cite[(1.6)]{Kato-toric}, we have that
    \begin{itemize}
        \item $P$ is an fs monoid,
        \item $P^{\gp}$ is torsion free (see \cite[I.1.3.5]{Ogus-log-geo}),
        \item there is an open neighbourhood $U\subseteq Y$ of $y$ and a homomorphism
              \[ \varphi\colon P\to \mc{M}_Y^{\zar}|_U  \]
              such that the induced composite morphism $P\xrightarrow{\varphi_y} \mc{M}_{Y,y}^{\zar}\to P$ is the identity, and
        \item there is an open neighbourhood $V\subseteq U$ of $y$ such that $\mc{M}_Y^{\zar}|_V$ is 
        the log-structure associated to the prelog-structure 
              \[ P\xrightarrow{\varphi|_V} \mc{M}_Y^{\zar}|_V \to \mc{O}_V. \]
    \end{itemize}
    Replacing $Y$ by $V$, we can assume that the homomorphism $\varphi$ is defined on the whole $Y$ and 
    that all the properties listed above are satisfied. 
    In particular, $\varphi\colon P\to \mc{M}_Y^{\zar}$ is a chart for the log-variety $(Y, \mc{M}_Y^{\zar})$
    (that is \emph{neat} at $y$; see \cite[\S II.2.3]{Ogus-log-geo}).  Denote by 
    \begin{equation}
        \varphi_Y\colon P\to \mc{O}_Y \label{map-on-OY}
    \end{equation}
    the composite homomorphism
    $P\xrightarrow{\varphi} \mc{M}_Y^{\zar} \to \mc{O}_Y$.  
    As $\varphi$ is a chart, $\varphi_Y$ induces a strict log-morphism of log-schemes
    \begin{equation}
        a\colon (Y, \mcM_Y^{\zar})\to \mathsf{A}_P^{\zar}, \label{chart-on-Y}
    \end{equation}
    where $\mathsf{A}_P^{\zar}$ is the log-variety whose log-structure is 
    the restriction of the log-structure of $\mathsf{A}_P$
    to $\zar(\AAA_P)$; see \S\ref{log-monoids}.  
    In particular, the morphism of underlying varieties
    \begin{equation}
        \ul{a}\colon Y\to \AAA_P, \label{proj-to-AAAP}
    \end{equation}
    is the morphism of varieties induced by \eqref{map-on-OY}. 
    Denote by $Y^*$ the maximal open subset of $Y$ on which $\mcM_Y^{\zar}$ is trivial.
    As $a$ is strict, by \cite[III.1.2.10, III.1.2.11]{Ogus-log-geo},
    $Y^*$ is the inverse image of $\spec \mbb K[P^{\gp}]\subset \AAA_P$ under $\ul{a}$.  

    For the Zariski case, if $\mcM$ is a log-structure on $\zar(X)$,
    by \cite[(1.6)]{Kato-toric}, we can assume that $Y\to X$ is an open immersion.
    Thus, for both the cases of small \'etale sites and Zariski sites,
    it suffices to show that the $(Y^*\subset Y)$ as above is a strict toroidal embedding. 

    \medskip
    
    \emph{Step 2}.
    Let $I(y, \mcM_Y^{\zar})$ be the ideal of $\mc{O}_{Y,y}$ generated 
    by the image of $\mc{M}_{Y,y}^{\zar}\setminus \mc{O}_{Y,y}^*$ in $\mc{O}_{Y,y}$.
    By definition of log-regularity \cite[(2.1)]{Kato-toric}, 
    $\mc{O}_{Y,y}/I(y, \mcM_Y^{\zar})$ is a regular local ring.
    Let $t_1', \dots, t_r'$ be a regular system of parameters of $\mc{O}_{Y,y}/I(y, \mcM_Y^{\zar})$, then
    \[ r = \dim \mc{O}_{Y,y} - \rank_{\ZZ} (P^{\gp}). \]
    Note that here we have applied the fact 
    $(\mc{M}_{Y,y}^{\zar})^{\gp}/\mcO_{Y,y}^*\cong (\mc{M}_{Y,y}^{\zar}/\mcO_{Y,y}^*)^{\gp} = P^{\gp}$,
    which is a result of \cite[I.1.3.5]{Ogus-log-geo}.

    Let $t_1,\dots, t_r$ be a set of preimages of $t_1', \dots, t_r'$ in $\mc{O}_{Y,y}$.
    Shrinking $Y$ near $y$, we can assume that $t_1,\dots, t_r$ are regular functions on the whole $Y$.
    By \cite[(3.2)]{Kato-toric}, the $\mbb K$-morphism of complete local rings induced by $\varphi_Y$ in \eqref{map-on-OY},
    \begin{equation}
        \wh{\psi}\colon \mbb K [[P]][[T_1,\dots, T_r]] \to \wh{\mc{O}}_{Y, y},\,T_i\mapsto t_i, \label{map-on-completion}
    \end{equation}
    is an isomorphism.  Note that the ring $\mbb K[[P]]$ is defined by
    \[ \mbb K [[P]] \coloneqq \varprojlim_{n\in \NN} \frac{\mbb K[P]}{I_P^n}, \]
    where $I_P$ is the maximal ideal of $\mbb K[P]$ 
    generated by $P\setminus P^*$; see \cite[(1.1)]{Kato-toric}.  Thus, $\wh{\mc{O}}_{Y,y}$
    is isomorphic to the completion of 
    \[ \mbb K[P][T_1, \dots, T_r] = \mbb K [P]\otimes_{\mbb K} \mbb K [T_1, \dots, T_r]
    \cong \mbb K [P\oplus \NN^r] \]
    with respect to the maximal ideal 
    \begin{equation}
        \mf m \coloneqq I_P\cdot \mbb K [P\oplus \NN^r] + (T_1, \dots, T_r)\cdot \mbb K [P\oplus \NN^r], \label{max-ideal}
    \end{equation}
    where $P\oplus \NN^r$ is direct sum of the monoids $P$ and $\NN^r$; cf. \cite[I.1.1.5]{Ogus-log-geo}. 
    As $P$ and $\NN^r$ are fs monoids, $P\oplus \NN^r$ is also an fs monoid.  Then
    \[ (P\oplus \NN^r)^{\gp} = P^{\gp}\oplus \ZZ^r \]
    is a torsion free finitely generated Abelian group, hence a lattice.  
    By \cite[Theorem 1.3.5]{CLS:toric},
    \[ W \coloneqq \spec \mbb K [P\oplus \NN^r] \]
    is a normal affine toric variety (of dimension $\dim Y$) with torus $\mbb T_W = \spec \mbb K [P^{\gp}\oplus \ZZ^r]$. 

    \medskip

    \emph{Step 3}.
    We can assume $Y$ is affine up to shrinking near $y$.  Denote by 
    \[ \psi\colon \mbb K[P\oplus \NN^r]\to \mbb K[Y] \]
    the homomorphism induced by $\varphi_Y$ and $\wh{\psi}$ 
    (see \eqref{map-on-OY} and \eqref{map-on-completion}), that is,
    $\psi$ and $\varphi_Y$ coincide on $P$, and $\psi$ maps every $T_i$ to $t_i$ as $\wh{\psi}$ does.
    Recall that $\mf{m}$ is the maximal ideal of $\mbb K [P\oplus \NN^r]$ in \eqref{max-ideal}, 
    which defines a closed point $w_0\in W$.  By \eqref{map-on-completion}, we have that
    \begin{itemize}
        \item the extended ideal $\psi(\mf m)\cdot \mbb K[Y]$ is contained in the maximal ideal $\mf{n}_y\subset \mbb K[Y]$, which defines the closed point $y\in Y$, and
        \item $\psi(\mf m)\cdot \wh{\mcO}_{Y,y}$ and $\mf{n}_y\wh{\mcO}_{Y,y}$ coincide in $\wh{\mcO}_{Y,y}$, which is the maximal ideal of $\wh{\mcO}_{Y,y}$.
    \end{itemize}
    Then by \cite[Theorem 7.5]{CA-Mat}, we can conclude that
    \begin{equation}
        \psi(\mf m)\cdot \mbb K[Y] = \mf{n}_y \label{maximal-ideals-equal}
    \end{equation}
    up to shrinking $Y$ near $y$.  Denote by
    \[ q_0\colon Y\to W \]
    the morphism of affine varieties corresponding to $\psi$.  Then \eqref{maximal-ideals-equal} shows that
    $q_0(y)=w_0$, and that $q_0$ is unramified at $y$.  Moreover, by \eqref{map-on-completion},
    and by \cite[(3), page 46]{CA-Mat}, $q_0$ is flat at $y$.
    Thus, up to shrinking $Y$ near $y$, we see that $q_0\colon Y\to W$ is \'etale.  

    \medskip

    \emph{Step 4}.
    Let $\rho\colon W\to W$ be the isomorphism of varieties induced by the homomorphism
    \[ \mbb K[P]\otimes_{\mbb K} \mbb K[T_1, \dots, T_r] \to \mbb K[P]\otimes_{\mbb K} \mbb K[T_1, \dots, T_r],\,\,T_i\mapsto T_i+1\text{ for }1\le i\le r,  \]
    which is identity on $\mbb K[P]$.  Denote by
    \[ q\colon Y\to W \]
    the composite morphism $Y\xrightarrow{q_0}W\xrightarrow{\rho} W$, which is also \'etale.  Then
    \[ q(y) = \rho(w_0) = w \]
    is a closed point contained in the open subset 
    $W^{\circ} \coloneqq \spec \mbb K[P]\times_{\spec \mbb K} \spec \mbb K[\ZZ^r]$ of $W$.
    Up to shrinking $Y$ near $y$, we can assume that the image $q(Y)$ is contained in $W^{\circ}$.
    
    Moreover, recall that $\psi$ and $\varphi_Y$ agree on $P$, 
    hence the morphism of varieties $\ul{a}$ in \eqref{proj-to-AAAP}
    is equal to the composite morphism 
    \[ Y\xrightarrow{q} W \to \spec \mbb K[P], \]
    where the second morphism is the canonical projection 
    \[ \spec \mbb K [P]\times_{\spec \mbb K} \spec \mbb K[\NN^r] \to \spec \mbb K[P]. \]
    Since $Y^*$ is the inverse image of $\spec \mbb K[P^{\gp}]$ (see Step 1),
    $Y^*$ is the inverse image of 
    \[ \spec \mbb K[P^{\gp}\oplus \ZZ^r] = W^{\circ} \cap \big(\spec \mbb K [P^{\gp}]\times_{\spec \mbb K}\spec \mbb K[\NN^r]\big) \]
    under the morphism $q\colon Y\to W$.  In other words, we have
    \[ Y^* = q^{-1}(\mbb T_W). \]
    Since $W$ is a normal toric variety, every irreducible component of $W\setminus \mbb T_W$
    is normal; see \cite[Proposition 3.2.7]{CLS:toric}.  As $q\colon Y\to W$
    is \'etale, we see that every irreducible component of $Y\setminus Y^*$ is also normal.
    Moreover, it is evident that $q\colon Y\to W$ is a strict \'etale-toric chart for $(Y^*\subset Y)$
    near $y$ in the sense of Definition~\ref{strict-chart-defn}.
    Thus, we can conclude that $(Y^*\subset Y)$ is a strict toroidal embedding.
\end{proof}


\section{Log-smoothness and toroidal morphisms}\label{etale-charts-toroidal-morphisms-section}

In this section, we show the equivalence between log-smooth log-morphisms 
and dominant toroidal morphisms.
As a result, we show that dominant toroidal morphisms 
admit \'etale-toric charts as in \cite[\S 5.3]{ACP-tropical-curves}.
The results are proved in the sequence:
\begin{itemize}
    \item dominant toroidal morphisms are log-smooth (see Theorems~\ref{toroidal-to-sm-zar}, \ref{toroidal-to-sm-et}),
    \item log-smooth log-morphisms admit \'etale-local charts (see Theorem~\ref{etale-chart-log-sm-thm}), which implies the equivalence between dominant toroidal morphisms and log-smoothness, and
    \item dominant toroidal morphisms admit \'etale-toric charts as in \cite[\S 5.3]{ACP-tropical-curves} (see Theorems~\ref{etale-chart-toroidal-morphism-thm}, \ref{etale-chart-toroidal-morphism-thm-zar}). 
\end{itemize}


\subsection{Ring of dual numbers}

Let $R$ be a ring and $M$ an $R$-module.
The \emph{ring of dual numbers on $M$} is the set $R\oplus M$ with componentwise addition 
and multiplication given by 
\[ (r, m)\cdot (r', m') \coloneqq (rr', r.m' + r'.m), \]
where $r,r'\in R$ and $m,m'\in M$.  
We denote this ring by $R[M]$ or $R\ltimes M$; see \cite{tri-abelian-book}.
The multiplicative identity of $R[M]$ is $(1, 0)$.
It is evident that $R[M]$ is an $R$-algebra with structure homomorphism $R\to R[M],\,r\mapsto (r, 0)$. 
The \emph{augmentation} $\pi\colon R[M]\to R$ is the ring homomorphism given by $(r, m)\mapsto r$.
The $R$-module $M$ can identified with the kernel of $\pi$, which is also denoted by $M$.
Then $M$ is an ideal of $R[M]$ satisfying $M^2 = 0$.

\begin{lemma}\label{dual-num-complete}
    Let $R$ be a complete local ring and $M$ a finite $R$-module.
    Then the ring of dual numbers $R[M]$ is also a complete local ring, and
    the canonical homomorphism $R\to R[M]$ of $R$-algebras is local.
\end{lemma}

\begin{proof}
    Denote by $\mf n$ the maximal ideal of $R$.
    It is evident that $\mf{n}[M] \coloneqq \mf{n}\oplus M$ is the unique maximal ideal of $R[M]$. 
    Then $R[M]$ is a local ring, and $R\to R[M]$ is a local homomorphism.
    We claim that $\mf{n}[M]^k \cong \mf{n}^k\oplus \mf{n}^{k-1}M$ as $R[M]$-modules 
    for every $k\in \NN$.  Indeed, for an arbitrary $(a, b)\in \mf{n}[M]$,
    and for $(a', b')\in \mf{n}^k\oplus \mf{n}^{k-1}M$, we see that
    \[ (a, b)\cdot (a',b') = (aa', a.b'+a'.b) \in \mf{n}^{k+1}\oplus \mf{n}^kM. \]
    Hence $\mf{n}[M]^{k+1}\subseteq \mf{n}^{k+1}\oplus \mf{n}^kM$.
    Conversely, it is evident that $\mf{n}^{k+1}\oplus 0\subseteq \mf{n}[M]^{k+1}$ and that
    $0\oplus \mf{n}^kM = (\mf{n}^k\oplus 0)\cdot (0\oplus M) \subseteq \mf{n}[M]^{k+1}$,
    whence $\mf{n}^{k+1}\oplus \mf{n}^kM \subseteq \mf{n}[M]^{k+1}$.  
    Thus, the claim is verified by induction on $k$.  Then
    \begin{equation}
    \begin{aligned}\label{iso-R-modules}
        \varprojlim_{k\in \NN} \dfrac{R[M]}{\mf{n}[M]^k} 
    &= \varprojlim_{k\in \NN} \dfrac{R\oplus M}{\mf{n}^k\oplus \mf{n}^{k-1}M}
    \cong \varprojlim_{k\in \NN} (R/\mf{n}^k)\oplus (M/\mf{n}^{k-1}M)  \\
    &= \big(\varprojlim_{k\in \NN} R/\mf{n}^k\big)\oplus \big(\varprojlim_{k\in \NN} M/\mf{n}^{k-1}M \big)
    = R \oplus M = R[M],
    \end{aligned}
    \end{equation}
    which establishes an isomorphism of $R$-modules.  A routine calculation shows
    that the isomorphism in \eqref{iso-R-modules} above is compatible with the ring structures
    of dual numbers, hence we can conclude that $R[M]$ is a complete local ring.
\end{proof}


\subsection{Log-smoothness for toroidal morphisms}
In this subsection, we show that dominant toroidal morphisms on either small \'etale sites 
or Zariski sites are log-smooth.  
For \emph{logarithmic differentials} (or \emph{log-differentials} for short),
we refer readers to \cite[\S IV.1]{Ogus-log-geo} and \cite[\S 12.3]{gabber2018foundationsringtheory}.
We only consider log-differentials in the usual Zariski topology.

We first treat the log-smoothness for Zariski sites, which is Theorem~\ref{toroidal-to-sm-zar}.
The proof of Theorem~\ref{toroidal-to-sm-zar} is due to Martin Olsson.

\begin{lemma}\label{diff-criterion-log-sm-zar}
    Let $(U_X\subset X)$ and $(U_Y\subset Y)$ be strict toroidal embeddings, and let
    \[ f\colon (U_X\subset X)\to (U_Y\subset Y) \] 
    be a morphism of embeddings, 
    where $f\colon X\to Y$ is dominant.  Denote by 
    \[ \mcM \coloneqq \mc{M}_{U_X/X}^{\zar}\, \text{ and }\, \mcN \coloneqq \mc{M}_{U_Y/Y}^{\zar} \] 
    the Zariski compactifying log-structures associated to $(U_X\subset X)$ and $(U_Y\subset Y)$.  
    Then there is an exact sequence of sheaves of log-differentials
    \[ 0\to f^*\big(\Omega_Y^1(\log \mcN) \big)\to \Omega_X^1(\log \mc{M}) 
    \to \Omega_{X/Y}^1\big(\log (\mc{M}/\mcN)\big) \to 0. \]
\end{lemma}

\begin{proof}
    By \cite[(12.3.12)]{gabber2018foundationsringtheory}, there is an exact sequence of $\mc{O}_X$-modules
    \[ f^*\big(\Omega_Y^1(\log \mcN) \big)\xrightarrow{\delta} \Omega_X^1(\log \mc{M}) 
    \to \Omega_{X/Y}^1\big(\log (\mc{M}/\mcN)\big) \to 0. \]
    Hence it suffices to show that $\delta$ injective.
    Denote by $K$ the kernel of $\delta$.
    By Lemma~\ref{log-regular-equi-smooth}, Theorem~\ref{toroidal-to-log-reg} 
    and \cite[12.3.26]{gabber2018foundationsringtheory}, 
    $f^*\big(\Omega_Y^1(\log \mcN) \big)$ is locally free of finite rank,
    so $K$ is a torsion free $\mc{O}_X$-module.  
    By \cite[(12.3.6)]{gabber2018foundationsringtheory}, we have
    \[ \Omega_X^1|_{U_X} \cong \big(\Omega_X^1(\log \mc{M})\big)|_{U_X}\,
    \text{ and }\, \Omega_Y^1|_{U_Y}\cong \big(\Omega_Y^1(\log \mcN)\big)|_{U_Y}. \]
    As $U_X$ is mapped into $U_Y$, the morphism $\delta$ on $U_X$ is isomorphic to 
    the morphism $\epsilon$ in the canonical exact sequence of usual K\"ahler differentials
    (see \cite[Corollaire (16.4.19)]{EGA-IV-4})
    \[ f^*(\Omega_Y^1|_{U_Y})\xrightarrow{\epsilon} \Omega_X^1|_{U_X}\to \Omega_{U_X/U_Y}^1\to 0. \]
    Denote by $K_{\epsilon}$ the kernel of $\epsilon$.
    Since $U_Y$ is nonsingular, $\Omega_Y^1|_{U_Y}$ is locally free, so 
    $f^*(\Omega_Y^1|_{U_Y})$ is locally free too.
    Then $K_{\epsilon}$ is a torsion free $\mcO_{U_X}$-module.
    As $\mbb K$ has characteristic zero, by generic smoothness and 
    \cite[Proposition (17.2.3)]{EGA-IV-4}, $K_{\epsilon}$
    is zero at the generic point of $X$.  (Here we have applied the assumption that $X\to Y$ is dominant.)
    Thus, by \cite[Chap. I, Proposition (7.4.6)]{EGA-I}, we can deduce that $K_{\epsilon} = 0$.
    Then $K$ vanishes on $U_X$.  However, recall that $K$ is torsion free, then we can conclude that $K=0$;
    in other words, $\delta$ is injective. 
\end{proof}


\begin{theorem}\label{toroidal-to-sm-zar}
    Let $f\colon (U_X\subset X)\to (U_Y\subset Y)$ be a \emph{dominant} toroidal morphism of 
    \emph{strict} toroidal embeddings.  Then the associated log-morphism of log-varieties in $\logzar$
    \[ f^{\zar}\colon (X, \mcM_{U_X/X}^{\zar}) \to (Y, \mcM_{U_Y/Y}^{\zar}) \]
    is log-smooth in $\logzar$.
\end{theorem}

\begin{proof}
    \emph{Step 1}.
    In this step, we set up necessary notation and 
    show that dominant toric morphisms are log-smooth (in Zariski topology).
    Let $D_X \coloneqq X\setminus U_X$ and $D_Y \coloneqq Y\setminus U_Y$.
    Let $x\in X$ be a closed point and $y \coloneqq f(x)$.  As $f$ is toroidal, 
    there exist a commutative diagram
    \begin{equation}
        \begin{aligned}\label{toric-models-diagram}
            \xymatrix{
            X\ar[d]^f & \spec \whtOO_{X,x}\ar[l]_-{\rho_x}\ar[r]^-{\phi}\ar[d]^{\wh{f}} & \spec \whtOO_{W,w}\ar[r]^-{\rho_w}\ar[d]^{\wh{g}} & W\ar[d]^g \\
            Y & \spec \whtOO_{Y,y}\ar[r]^-{\psi}\ar[l]_-{\rho_y} & \spec \whtOO_{V,v}\ar[r]^-{\rho_v} & V
            }\end{aligned}
    \end{equation}
    where 
    \begin{itemize}
        \item $\Set{(W, C), w}$, $\Set{(V, B), v}$ are local toric models of $\Set{(X, D_X), x}$, $\Set{(Y, D_Y),y}$,
        \item $g\colon W\to V$ is a toric morphism of normal affine toric varieties,
        \item $\phi, \psi$ are isomorphisms of schemes induced by local toric models,
        \item $\rho_x, \rho_y, \rho_w, \rho_v$ are canonical morphisms, and
        \item $\wh{f}, \wh{g}$ are the natural morphisms induced by $f,g$ respectively.
    \end{itemize}
    Write
    \[ \wh{X} \coloneqq \spec \whtOO_{X,x}, \,\wh{Y} \coloneqq \spec \whtOO_{Y,y},
    \, \wh{W} \coloneqq \spec \whtOO_{W,w} \,\text{ and }\,
    \wh{V} \coloneqq \spec \whtOO_{V,v}. \]
    As $X, Y$ are normal varieties, by Lemma~\ref{local-irreducibility},
    $\wh{X}, \wh{Y}$ are normal irreducible schemes.
    By assumption, $f$ is dominant, 
    then $\wh{f}$ maps the generic point of $\wh{X}$ to the generic point of $\wh{Y}$.
    As $\phi, \psi$ are isomorphisms, $\wh{g}$ also maps the generic point of $\wh{W}$
    to the generic point of $\wh{V}$.  Then the generic point of $W$ is mapped to 
    the generic point of $V$ by $g$.  Thus, we see that $g$ is also dominant.

    Let $\mbb T_W, \mbb T_V$ be the tori of $W, V$ respectively.  Let
    \[ \mcM_W \coloneqq \mcM_{\mbb T_W/W}^{\zar}\,\text{ and }\, 
    \mcM_V \coloneqq \mcM_{\mbb T_V/V}^{\zar}. \]
    We claim that the natural log-morphism 
    \[ g^{\zar}\colon (W, \mcM_W) \to (V, \mcM_V) \]
    is log-smooth in $\logzar$.  To this end,
    we can assume that $W$ is given by $(N_W, \sigma)$ 
    and that $V$ is given by $(N_V, \tau)$; see \S\ref{toric-couples-defn}. 
    Then by \cite[Theorem 3.3.4]{CLS:toric}, $g$ is induced by a map of lattices $\ol{g}\colon N_W\to N_V$
    such that $\ol{g}_{\RR}(\sigma)\subseteq \tau$, where $\ol{g}_{\RR}\colon (N_W)_{\RR}\to (N_V)_{\RR}$
    is the $\RR$-linear map induced by $\ol{g}$.  Taking the dual lattices, we get the dual map of lattices
    $\ol{g}^{\vee}\colon M_V\to M_W$.  
    Then $g^{\zar}$ is given by the homomorphism of fs monoids induced by $\ol{g}^{\vee}$,
    \[ \phi_g\colon \tau^{\vee}\cap M_V\to \sigma^{\vee}\cap M_W. \]
    Note that by the proof of \cite[Theorem 1.2.18]{CLS:toric}, we have
    \[ (\sigma^{\vee}\cap M_W)^{\gp} = M_W\,\text{ and }
    \, (\tau^{\vee}\cap M_V)^{\gp} = M_V. \]
    Thus, $\phi_g^{\gp}$ is equal to the dual map of lattices $\ol{g}^{\vee}$.
    As $g$ is dominant, $\mbb T_W\to \mbb T_V$ is also dominant.  
    Hence $\phi_g^{\gp}$ is an injective homomorphism of Abelian groups.
    Recall that $\chara \mbb K = 0$, then by \cite[12.3.34]{gabber2018foundationsringtheory},
    we can conclude that $g^{\zar}$ is log-smooth in $\logzar$.

    \medskip
    
    \emph{Step 2}.
    In this step, we reduce the verification of log-smoothness to 
    a lifting property of homomorphisms between modules.
    For simplicity of notations, let 
    \[ \mcM_X \coloneqq \mcM_{U_X/X}^{\zar}\,\text{ and }\,\mcM_Y \coloneqq \mcM_{U_Y/Y}^{\zar}. \]
    As we are working on Zariski sites, in the rest of the proof,
    we will denote a log-morphism by the same symbol of the morphism
    between the underlying schemes.  
    Hence we will also denote by $f$ the log-morphism $(X, \mcM_X)\to (Y, \mcM_Y)$.
    This convention should not lead to confusion from the contexts.

    By \cite[12.3.20]{gabber2018foundationsringtheory}, the sheaf of log-differentials
    $\Omega_{X/Y}^1\big(\log (\mc{M}_X/\mcM_Y)\big)$ of $f$ is a coherent $\mcO_X$-module;
    cf. the notations in \cite[\S 12.3]{gabber2018foundationsringtheory}.
    By Lemma~\ref{diff-criterion-log-sm-zar} and \cite[12.3.32]{gabber2018foundationsringtheory}
    (or \cite[IV.3.2.3]{Ogus-log-geo}), it suffices to show that 
    $\Omega_{X/Y}^1\big(\log (\mc{M}_X/\mcM_Y)\big)$ is locally free.
    Then it suffices to show that for every closed point $x\in X$, 
    the stalk of $\Omega_{X/Y}^1\big(\log (\mc{M}_X/\mcM_Y)\big)$ at $x$ is a free $\mcO_{X,x}$-module.
    To this end, denote by 
    \[ \Omega \coloneqq \Omega_{X/Y}^1\big(\log (\mc{M}_X/\mcM_Y)\big), \]
    and by $\wh{\Omega}$ the pullback of $\Omega$ to $\wh{X}$.
    By \cite[Chapitre III, \S 3.5, Corollaire 2, page 248]{Bourbaki:ca:1-4}, 
    it suffices to show that $\wh{\Omega}$ is a free $\whtOO_{X,x}$-module.  
    Then by \cite[Theorem 2.5]{CA-Mat}, it suffices to show that $\wh{\Omega}$
    is a projective $\whtOO_{X,x}$-module.
    To this end, it suffices to verify that 
    for some surjective homomorphism of $\whtOO_{X,x}$-modules $F\to \wh{\Omega}$ with
    $F$ finitely generated and free, the induced map
    \begin{equation}
        \homo_{\whtOO_{X,x}}(\wh{\Omega}, F)\to \homo_{\whtOO_{X,x}}(\wh{\Omega}, \wh{\Omega}) 
        \label{bijection-homo-map}
    \end{equation}
    is surjective; see \cite[Proposition A3.1, page 615]{Eisenbund-CA-book}.

    \medskip

    \emph{Step 3}.
    Let $E$ be an $\whtOO_{X,x}$-module, and let $\whtOO_{X,x}[E]$ be the ring of dual numbers on $E$.  Let
    \[ \wh{X}[E] \coloneqq \spec \big(\whtOO_{X,x}[E]\big). \]
    Let
    \[ r\colon \wh{X}[E]\to \wh{X} \] 
    be the \emph{retraction} induced by 
    the structure homomorphism $\whtOO_{X,x}\to \whtOO_{X,x}[E]$, which 
    sends every $a\in \whtOO_{X,x}$ to $(a, 0)$.  Let
    \[ \mcM_{\wh{X}} \coloneqq (\rho_x)^*_{\log} \mcM_X \,\text{ and }\, 
    \mcM_{\wh{X}[E]} \coloneqq r^*_{\log} (\mcM_{\wh{X}}). \]
    We view $(\wh{X}[E], \mcM_{\wh{X}[E]})$ as a log-scheme over $(Y, \mcM_Y)$ via the composite log-morphism
    \begin{equation}
        (\wh{X}[E], \mcM_{\wh{X}[E]})\xrightarrow{r} (\wh{X}, \mcM_{\wh{X}}) \xrightarrow{\rho_x}
    (X, \mcM_X)\xrightarrow{f} (Y, \mcM_Y). \label{structure-log-map}
    \end{equation}
    Denote by $\logsch_{\zar}/(Y, \mcM_Y)$ the category of log-schemes over $(Y, \mcM_Y)$, and by
    \[ \Modd\logsch_{\zar}/(Y, \mcM_Y) \]
    the category whose objects are all the pairs $\Set{(Q, \mcM_Q), \mcF}$,
    where $(Q, \mcM_Q)$ is an object in $\logsch_{\zar}/(Y,\mcM_Y)$, 
    and $\mcF$ is an $\mcO_Q$-module.  The morphisms
    \[ \Set{(Q, \mcM_Q), \mcF} \to \Set{(Z, \mcM_Z), \mcG} \]
    in $\Modd\logsch_{\zar}/(Y, \mcM_Y)$ are the pairs $(\theta, \varphi)$ consisting of a log-morphism
    \[ \theta\colon (Q, \mcM_Q)\to (Z, \mcM_Z) \] 
    in $\logsch_{\zar}/(Y,\mcM_Y)$, and a morphism 
    \[ \varphi\colon \theta^*\mcG\to \mcF \] 
    of $\mcO_Q$-modules; cf. \cite[\S 12.3.5]{gabber2018foundationsringtheory}.
    By \cite[\S 12.3.14]{gabber2018foundationsringtheory}, there is a natural bijection 
    \[ \mathfrak{g}_E\colon \mbb H_{\modd}(E) \to \mbb H_{\Log}[E], \]
    where
    \begin{equation}
        \mbb H_{\modd}(E) \coloneqq \homo_{\Modd\logsch_{\zar}/(Y,\mcM_Y)}\bigg( \Set{(\wh{X}, \mcM_{\wh{X}}), E}, 
            \Set{(X, \mcM_X), \Omega}\bigg) \label{big-bijection-domain}
    \end{equation}
    and
    \begin{equation}
        \mbb H_{\Log}[E] \coloneqq \homo_{\logsch_{\zar}/(Y, \mcM_Y)} \bigg( (\wh{X}[E], \mcM_{\wh{X}[E]}), (X, \mcM_X) \bigg).
        \label{big-bijection-target}
    \end{equation}
    For evert $(\theta, \varphi)\in \mbb H_{\modd}(E)$,
    we denote its image in $\mbb H_{\Log}[E]$ by $\mathfrak{g}_{E,(\theta, \varphi)}$; 
    cf. \cite[\S 12.3.14]{gabber2018foundationsringtheory}.  Let
    \[ \mbb H_{\modd}(E)_{\rho_x} \coloneqq \Set{ (\theta, \varphi)\in \mbb H_{\modd}(E)\mid \theta = \rho_x }. \]
    It is evident that there is a natural identification
    \begin{equation}
        \mbb H_{\modd}(E)_{\rho_x} = \homo_{\whtOO_{X,x}}(\wh{\Omega}, E). \label{bijection-modules}
    \end{equation}
    The augmentation $\whtOO_{X,x}[E]\to \whtOO_{X,x}$ induces a log-morphism
    \[ \gamma\colon (\wh{X}, \mcM_{\wh{X}}) \to (\wh{X}[E], \mcM_{\wh{X}[E]}), \]
    which we call the \emph{reduction of $(\wh{X}[E], \mcM_{\wh{X}[E]})$}.
    By the construction in \cite[\S 12.3.14]{gabber2018foundationsringtheory}, 
    for every $(\theta, \varphi)\in \mbb H_{\modd}(E)_{\rho_x}$, 
    the image $\mathfrak{g}_{E,(\theta, \varphi)}$ satisfies that
    \[ \mathfrak{g}_{E, (\theta, \varphi)}\circ \gamma = \rho_x. \]
    Let
    \[ \mbb H_{\Log}[E]_{\rho_x} \coloneqq \Set{\alpha\in \mbb H_{\Log}[E] 
    \mid \alpha\circ \gamma = \rho_x }. \]
    Then $\mathfrak{g}_E$ restricts to a bijection 
    \[ \mathfrak{g}_E|_{\rho_x}\colon \mbb H_{\modd}(E)_{\rho_x} \to \mbb H_{\Log}[E]_{\rho_x}. \]
    By \eqref{bijection-modules}, we get a natural bijection
    \begin{equation}
        \homo_{\whtOO_{X,x}}(\wh{\Omega}, E) 
        \to \mbb H_{\Log}[E]_{\rho_x}, \label{bijection-transfer-mods}
    \end{equation}
    which holds for any $\whtOO_{X,x}$-module $E$.

    \medskip

    \emph{Step 4}.
    Recall that $F$ is a finite free $\whtOO_{X,x}$-module admitting 
    a surjection $F\to \wh{\Omega}$.  Then by \eqref{bijection-transfer-mods}, 
    for surjectivity of \eqref{bijection-homo-map}, it suffices to show the natural map
    \begin{equation}
        \mbb H_{\Log}[F]_{\rho_x} \to \mbb H_{\Log}[\wh{\Omega}]_{\rho_x}
        \label{surj-to-pf}
    \end{equation}
    is surjective.  It is easy to see that for every $h\in \mbb H_{\Log}[\wh{\Omega}]_{\rho_x}$, 
    inverse images of $h$ under \eqref{surj-to-pf} are identified with 
    dotted arrows filling in the commutative diagram
    \begin{equation}
        \begin{aligned}\label{lifting-diagram}
            \xymatrix{
            (\wh{X}[\wh{\Omega}], \mcM_{\wh{X}[\wh{\Omega}]})\ar[r]^-{\wh{h}}\ar[d]^i\ar@/^2pc/[rr]^-h & (\wh{X}, \mcM_{\wh{X}})\ar[r]^-{\rho_x}\ar[d] & (X, \mcM_X)\ar[d]^f \\
            (\wh{X}[F], \mcM_{\wh{X}[F]})\ar[r]\ar@{-->}[rru] & (\wh{Y}, \mcM_{\wh{Y}})\ar[r]^-{\rho_y} & (Y, \mcM_Y)
            }\end{aligned}
    \end{equation}
    where the composition of bottom horizontal arrows is the structure map of $(\wh{X}[F], \mcM_{\wh{X}[F]})$
    as a log-scheme over $(Y, \mcM_Y)$; see \eqref{structure-log-map}.
    
    We claim that the natural log-morphism
    \[ i\colon (\wh{X}[\wh{\Omega}], \mcM_{\wh{X}[\wh{\Omega}]})
    \to (\wh{X}[F], \mcM_{\wh{X}[F]}) \]
    is a log-thickening of order one; cf. \cite[IV.2.1.1]{Ogus-log-geo}.
    Indeed, let $K$ be the kernel of $F\to \wh{\Omega}$,
    then $i\colon \wh{X}[\wh{\Omega}] \to \wh{X}[F]$ is the closed immersion defined by
    the ideal $I_K \coloneqq 0\oplus K$ of $\whtOO_{X,x}[F]$, which clearly satisfies that $I_K^2=0$.
    It is evident that $i^*_{\log} \mcM_{\wh{X}[F]} = \mcM_{\wh{X}[\wh{\Omega}]}$
    as $\mcM_{\wh{X}[F]}, \mcM_{\wh{X}[\wh{\Omega}]}$ are inverse images of $\mcM_{\wh{X}}$ via retractions.
    Hence $i$ is a strict closed immersion of log-schemes.
    By Lemma~\ref{formal-log-regular-schemes}, $\mcM_{\wh{X}}$ admits an fs chart.
    As $\mcM_{\wh{X}[F]}$ is the inverse image of $\mcM_{\wh{X}}$,
    we see that $\mcM_{\wh{X}[F]}$ also admits an fs chart.
    Then by \cite[II.2.3.6]{Ogus-log-geo}, the sheaf of monoids $\mcM_{\wh{X}[F]}$ is fs, hence integral.
    Thus, the subgroup $(1, 0) + I_K$ of $\mcO_{\wh{X}[F]}^*\cong \mcM_{\wh{X}[F]}^*$ 
    operates freely on $\mcM_{\wh{X}[F]}$.  By \cite[IV.2.1.1]{Ogus-log-geo},
    $i$ is a log-thickening of order one.
    
    \medskip

    \emph{Step 5}.
    It is evident that \eqref{toric-models-diagram} induces 
    a commutative diagram of log-schemes
    \begin{equation}
        \begin{aligned}\label{log-toric-models}
            \xymatrix{
            (X, \mcM_X)\ar[d]^f & (\wh{X}, \mcM_{\wh{X}})\ar[l]_-{\rho_x}\ar[r]^-{\phi}\ar[d]^{\wh{f}} & (\wh{W}, \mcM_{\wh{W}})\ar[r]^-{\rho_w}\ar[d]^{\wh{g}} & (W, \mcM_W)\ar[d]^g \\
            (Y, \mcM_Y) & (\wh{Y}, \mcM_{\wh{Y}})\ar[r]^-{\psi}\ar[l]_-{\rho_y} & (\wh{V}, \mcM_{\wh{V}})\ar[r]^-{\rho_v} & (V, \mcM_V)
            }
        \end{aligned}
    \end{equation}
    where $\phi, \psi$ are isomorphisms of log-schemes.  Then the left square of \eqref{lifting-diagram}
    fits into the commutative diagram of log-schemes
    \begin{equation}
        \begin{aligned}\label{lifting-diagram-toric-side}
            \xymatrix{
            (\wh{X}[\wh{\Omega}], \mcM_{\wh{X}[\wh{\Omega}]})\ar[r]^-{\wt{h}}\ar[d]^i & (\wh{W}, \mcM_{\wh{W}})\ar[r]^-{\rho_w}\ar[d]^{\wh{g}} & (W, \mcM_W)\ar[d]^g \\
            (\wh{X}[F], \mcM_{\wh{X}[F]})\ar[r] & (\wh{V}, \mcM_{\wh{V}})\ar[r]^-{\rho_v} & (V, \mcM_V)
            }\end{aligned}
    \end{equation}
    where $\wt{h} \coloneqq \wh{h}\circ \phi$; see \eqref{lifting-diagram} for $\wh{h}$.

    By Step 1, $g$ is log-smooth in $\logzar$.  Then by \cite[IV.3.1.1]{Ogus-log-geo} 
    (see also \cite[12.3.22]{gabber2018foundationsringtheory}), there is a log-morphism
    \[ u \colon (\wh{X}[F], \mcM_{\wh{X}[F]}) \to (W, \mcM_W), \]
    which is a deformation of $\rho_w\circ\wt{h}$ to $(\wh{X}[F], \mcM_{\wh{X}[F]})$;
    see \cite[IV.2.2.1]{Ogus-log-geo}.  By Lemma~\ref{dual-num-complete},
    $\whtOO_{X,x}[F]$ is a complete local ring, hence $u$ factors through $(\wh{W}, \mcM_{\wh{W}})$.
    In other words, there is a log-morphism 
    \[ \wt{u} \colon (\wh{X}[F], \mcM_{\wh{X}[F]}) \to (\wh{W}, \mcM_{\wh{W}}), \]
    which fits into \eqref{lifting-diagram-toric-side} and makes the diagram commutative.
    Recall that $\phi, \psi$ in \eqref{log-toric-models} are isomorphisms of log-schemes.
    Then $\wt{u}$ corresponds to a log-morphism
    \[ \wh{u}\colon (\wh{X}[F], \mcM_{\wh{X}[F]}) \to (\wh{X}, \mcM_{\wh{X}}), \]
    which is a deformation of $\wh{h}$ in \eqref{lifting-diagram} to $(\wh{X}[F], \mcM_{\wh{X}[F]})$.
    Then $\rho_x\circ \wh{u}$ gives a dotted arrow in \eqref{lifting-diagram} making the diagram commutative,
    which shows that \eqref{surj-to-pf} is surjective.
    Therefore, by \eqref{bijection-transfer-mods}, we can conclude that \eqref{bijection-homo-map}
    is surjective, hence $\Omega$ is locally free.
\end{proof}


\begin{remark}
    Keep the notations in Theorem~\ref{toroidal-to-sm-zar}.
    If $f$ is not dominant, then $f^{\zar}$ can not be log-smooth.
    To this end, assume that $f$ is not dominant, but $f^{\zar}$ is log-smooth.
    As $f$ is a morphism of embeddings, it induces a morphism of nonsingular varieties 
    $f\colon U_X\to U_Y$.  As $f^{\zar}$ is log-smooth,
    the restriction of $f^{\zar}$ to $U_X$, $(U_X, \mcM_{U_X})\to (U_Y, \mcM_{U_Y})$,
    is also log-smooth, where $\mcM_{U_X}, \mcM_{U_Y}$ are the trivial log-structures on $U_X, U_Y$.
    Then by \cite[12.3.27]{gabber2018foundationsringtheory}
    (see also \cite[IV.3.1.6]{Ogus-log-geo}), $f\colon U_X\to U_Y$ is smooth in the usual sense,
    contradicting the assumption that $f$ is not dominant.
\end{remark}


Now we show that dominant toroidal morphisms are log-smooth in \'etale topology.

\begin{theorem}\label{toroidal-to-sm-et}
    Let $f\colon (U_X\subset X)\to (U_Y\subset Y)$ be a \emph{dominant} 
    toroidal morphism of toroidal embeddings.
    Then the associated log-morphism of log-varieties in $\loget$
    \[ f^{\et}\colon (X, \mcM_{U_X/X})\to (Y, \mcM_{U_Y/Y}) \]
    is log-smooth in $\loget$.
\end{theorem}

\begin{proof}
    Pick an arbitrary closed point $x\in X$, and let $y \coloneqq f(x)$.
    By Lemma~\ref{pass-to-strict}, there is an \'etale neighbourhood $\pi\colon Y'\to Y$ of $y$
    such that $(U_{Y'}\subset Y')$ is a strict toroidal embedding, where $U_{Y'} \coloneqq \pi^{-1}(U_Y)$.
    Consider the following Cartesian diagram
    \[\xymatrix{
    X\ar[d]^{f} & X'\ar[d]^{f'}\ar[l]_{\mu} \\
    Y & Y'\ar[l]_{\pi}
    }\]
    where $X' \coloneqq X\times_Y Y'$.  Set $U_{X'} \coloneqq \mu^{-1}(U_X)$.
    By \cite[Theorem (3.5)]{Kato89}, log-smoothness in $\loget$ is an \'etale-local property.
    Since $X$ is normal (see \cite[Lemma 3.8]{birkar2023singularities}), 
    and since $\mu$ is \'etale, $X'$ has normal irreducible components.
    Then for log-smoothness of $f^{\et}$, we can assume that $X'$ is irreducible.
    
    As $\mu\colon X'\to X$ is \'etale, 
    it is evident that $(U_{X'}\subset X')$ is a toroidal embedding and that 
    \[ f'\colon (U_{X'}\subset X')\to (U_{Y'}\subset Y') \]
    is also a dominant toroidal morphism.  
    Let $x'\in X'$ be a closed point mapping to $x\in X$.
    Again by Lemma~\ref{pass-to-strict}, there is an \'etale neighbourhood $\lambda\colon X''\to X'$
    of $x'$ such that $(U_{X''}\subset X'')$ is a strict toroidal embedding, 
    where $U_{X''} \coloneqq \lambda^{-1}(U_{X'})$.  Then the induced morphism
    \[ f''\colon (U_{X''}\subset X'')\to (U_{Y'}\subset Y') \]
    is a dominant toroidal morphism of strict toroidal embeddings.
    Moreover, we can assume that the log-varieties $(X'', \mcM_{U_{X''}/X''})$ and $(Y', \mcM_{U_{Y'}/Y'})$
    admit global charts, hence the log-structures are Zariski; see Lemma~\ref{et-to-zar-charts}.
    In addition, by Lemma~\ref{log-sm-et-zar}, the log-morphism
    \[ (f'')^{\et}\colon (X'', \mcM_{U_{X''}/X''}) \to (Y', \mcM_{U_{Y'}/Y'}) \]
    is log-smooth in $\loget$ if and only if 
    \[ (f'')^{\zar}\colon (X'', \mcM_{U_{X''}/X''}^{\zar}) \to (Y', \mcM_{U_{Y'}/Y'}^{\zar}) \]
    is log-smooth in $\logzar$.  By Theorem~\ref{toroidal-to-sm-zar}, $(f'')^{\zar}$
    is indeed log-smooth in $\logzar$, then we see that $(f'')^{\et}$ is log-smooth
    in $\loget$.  Thus, $f^{\et}$ is log-smooth in $\loget$.
\end{proof}


\subsection{\'Etale-toric charts for log-smooth log-morphisms}
The existence of \'etale-local charts for log-smooth log-morphisms
is a classical result in logarithmic geometry; see \cite[Theorem (3.5)]{Kato89},
\cite[\S IV.3.3]{Ogus-log-geo} and \cite[12.3.37]{gabber2018foundationsringtheory}, etc.
For an approach that does not depend on arguments in logarithmic geometry,
see \cite[Proposition 3.3]{denef2013remarkstoroidalmorphisms}.

\begin{theorem}\label{etale-chart-log-sm-thm}
    Let $f\colon (X, \mcM)\to (Y, \mcN)$ be a log-smooth log-morphism in $\loget$,
    where $(X,\mcM)$ and $(Y, \mcN)$ are log-regular log-varieties in $\loget$.  
    Then for every closed point $x\in X$ and $y\coloneqq \ul{f}(x)$, there is a commutative diagram in $\loget$
    \begin{equation}
        \begin{aligned}\label{et-chart-log-sm-diagram}
            \xymatrix{
            (X, \mcM)\ar[d]^f & (X', \mcM')\ar[l]_-{\mu}\ar[r]^-a\ar[d]^{f'} & \mathsf{A}_Q\ar[d]^{\mathsf{A}_{\theta}} \\
            (Y, \mcN) & (Y', \mcN')\ar[l]_-{\pi}\ar[r]^-b & \mathsf{A}_P
            }
        \end{aligned}
    \end{equation}
    where
    \begin{itemize}
        \item the morphisms $\ul{\mu}, \ul{\pi}$ of underlying varieties are \'etale neighbourhoods of $x, y$,
        \item $\mcM', \mcN'$ are the inverse images of $\mcM, \mcN$ to $\Et(X'), \Et(Y')$,
        \item $a,b$ are charts for $(X', \mcM'), (Y',\mcN')$ subordinate to fs monoids $Q, P$,
        \item the log-morphisms $a,b$ are log-\'etale in $\loget$,
        \item $\theta\colon P\to Q$ is an injective homomorphism of monoids,
        \item $f'$ is the induced log-morphism of log-regular log-varieties in $\loget$,
        \item $(b, \theta, a)$ is a chart for the log-morphism $f'$ subordinate to $\theta\colon P\to Q$.
        \item $(X', (\mcM')^{\zar})$ and $(Y', (\mcN')^{\zar})$ are log-regular log-varieties in $\logzar$,
        \item the induced log-morphism on Zariski sites 
        \[ (f')^{\zar}\colon (X', (\mcM')^{\zar}) \to (Y', (\mcN')^{\zar}) \] 
        is log-smooth in $\logzar$,
        \item the schemes $\AAA_Q, \AAA_P$ are toric varieties, and
        \item the underlying morphism $\AAA_{\theta}\colon \AAA_Q\to \AAA_P$ of schemes for the log-morphism $\mathsf{A}_{\theta}$ is a dominant toric morphism of toric varieties.
    \end{itemize}
\end{theorem}

\begin{proof}
    \emph{Step 1}.
    Let $\ul{\pi}\colon Y'\to Y$ be an \'etale neighbourhood of $y$ such that the inverse image $\mcN'$
    of $\mcN$ to $\Et(Y')$ admits a global chart.  
    We can assume that $Y'$ is irreducible, whence a variety.  
    Denote by $\pi\colon (Y', \mcN')\to (Y, \mcN)$ the log-morphism of log-varieties in $\loget$.
    
    By Lemma~\ref{et-to-zar-charts},
    the log-structure $\mcN'$ is Zariski.  Moreover, by Lemma~\ref{log-regular-etale-local}, $(Y', \mcN')$
    is also log-regular in $\loget$.  Thus, Lemma~\ref{log-regular-two-tops} shows that
    $(Y', (\mcN')^{\zar})$ is log-regular in $\logzar$.  
    Denote by $Y'^*$ the maximal open subset of $Y'$ on which $(\mcN')^{\zar}$ is trivial.
    By Theorem~\ref{toroidal-regular}, $(Y'^*\subset Y')$ is a strict toroidal embedding.
    Also by \S\ref{log-regular-compactifying-section}, we see that $(\mcN')^{\zar}$ is equal to the Zariski
    compactifying log-structure $\mcM_{Y'^*/Y'}^{\zar}$ on $\zar(Y')$.

    Let $y'\in Y'$ be a closed point mapping to $y\in Y$.
    By Theorem~\ref{strict-chart}, up to shrinking $Y'$ near $y'$ (in the Zariski topology),
    there is a strict \'etale-toric chart $\pi'\colon Y'\to V$ 
    to a normal affine toric variety $V$ such that 
    $(\pi')^{-1}(\mbb T_V) = Y'^*$.  We can assume that $V$ is given by $(N_V, \tau)$; 
    see \S\ref{toric-couples-defn}.  Let $M_V$ be the dual lattice of $N_V$, and let
    \[ P \coloneqq \tau^{\vee}\cap M_V, \]
    which is an fs monoid with $P^{\gp} = M_V$, so $P^{\gp}$ is a free Abelian group.  
    Moreover, by Lemma~\ref{toroidal-charts} and \cite[IV.3.1.6]{Ogus-log-geo}, 
    the log-morphism in $\loget$ induced by $\pi'$,
    \[ b\colon (Y', \mcN')\to \mathsf{A}_P, \]
    is strict and log-\'etale.  

    \medskip
    
    \emph{Step 2}.
    Let $X' \coloneqq X\times_Y Y'$, and
    let $x'\in X'$ be a closed point mapping to $x, y'$.
    By \cite[IV.3.3.1]{Ogus-log-geo}, up to replacing $X',Y'$ 
    by another \'etale neighbourhoods of $\ol{x},\ol{y}$,
    the chart $b$ fits into an \'etale-local chart $(b, \theta, a)$ for $f$:
    \begin{equation}
        \begin{aligned}\label{et-chart-toric-diagram}
            \xymatrix{
            X'\ar[rd]_{f'}\ar[r]^-{a_{\theta}}\ar@/^2pc/[rr]^a & Y_{\theta}'\ar[r]\ar[d]^{f_{\theta}'} & \mathsf{A}_Q\ar[d]^{\mathsf{A}_{\theta}} \\
               & Y'\ar[r]^-b & \mathsf{A}_P
            }\end{aligned}
    \end{equation}
    where
    \begin{itemize}
        \item $\theta\colon P\to Q$ is an injective homomorphism of monoids,
        \item $Y_{\theta}'\to \mathsf{A}_Q$ is the base change of $b$ in the category $\loget$,
        \item $a_{\theta}$ is log-\'etale and strict, and
        \item the chart $a$ is exact at $\ol{x}$.
    \end{itemize}
    Note that in \eqref{et-chart-toric-diagram}, the sheaves of monoids
    $\mcM',\mcN'$, etc., are not written out explicitly.  
    Moreover, $Q$ is a fine monoid by the proof of \cite[IV.3.3.1]{Ogus-log-geo};
    see \cite[II.2.4.4, III.1.2.7]{Ogus-log-geo}.

    As log-\'etale log-morphisms are stable under 
    compositions and base changes (see \cite[IV.3.1.2]{Ogus-log-geo}), the log-morphism 
    \[ a\colon (X', \mc{M}')\to \mathsf{A}_Q \] 
    of log-schemes is strict and log-\'etale in $\loget$.  Hence by \cite[IV.3.1.6]{Ogus-log-geo},
    the morphism of the underlying schemes $\ul{a}\colon X'\to \AAA_Q$ is \'etale.
    Then by \cite[Proposition (2.1.13)]{EGA-IV-II}, 
    the image of $\ul{a}$ is an irreducible and normal open subscheme of $\AAA_Q$.
    Up to taking the normalisation of $\AAA_Q$ and taking the irreducible 
    component containing the image of $\ul{a}$,
    we can assume that $Q$ is fs and that $Q^{\gp}$ is torsion free; cf. the proof of \cite[I.3.4.1 (3)]{Ogus-log-geo}. 

    \medskip
    
    \emph{Step 3}.
    Recall that $P = \tau^{\vee}\cap M_V$ with $P^{\gp} = M_V$.
    Since $Q$ is an fs monoid with $Q^{\gp}$ free, and since $\AAA_{Q^{\gp}}$
    is an open subscheme of $\AAA_Q$, we see that $Q^{\gp}$ is a lattice 
    of rank $d \coloneqq \dim \AAA_Q = \dim X$.
    Let $\mc{A}_P$ and $\mc{A}_Q$ be finite sets of generators of $P$ and $Q$ respectively
    such that $\mc{A}_P$ is contained in $\mc{A}_Q$ via the injection $P\to Q$.
    Denote by $\cone (\mc{A}_P)$ (respectively, by $\cone (\mc{A}_Q)$) the cone generated by 
    $\mc{A}_P$ in $P^{\gp}\otimes_{\ZZ} \RR$ (respectively, by $\mc{A}_Q$ in $Q^{\gp}\otimes_{\ZZ}\RR$).
    The proof of \cite[Theorem 1.3.5]{CLS:toric} shows that
    \[ \big( \cone(\mc{A}_P)\big)^{\vee} \]
    is a strongly convex rational polyhedral cone in $(P^{\gp})^{\vee}\otimes_{\ZZ}\RR$ such that 
    \[ \tau^{\vee}\cap M_V = \cone(\mc{A}_P)\cap M_V. \]
    By the uniqueness statement in \cite[Proposition 1.1, page 3]{oda-convex-bodies-toric}, we see that
    \[ \tau = \big( \cone(\mc{A}_P)\big)^{\vee}. \]
    Again by the proof of \cite[Theorem 1.3.5]{CLS:toric},
    \[ \sigma \coloneqq \big( \cone (\mc{A}_Q) \big)^{\vee} \]
    is a strongly convex rational polyhedral cone in $(Q^{\gp})^{\vee}\otimes_{\ZZ}\RR$ such that 
    \[ Q = \sigma^{\vee} \cap Q^{\gp}. \]
    Thus, the underlying scheme $\AAA_Q$ of the log-scheme $\mathsf{A}_Q$
    is the normal affine toric variety defined by $((Q^{\gp})^{\vee}, \sigma)$.
    Moreover, as $\cone (\mc{A}_P)$ is mapped into $\cone (\mc{A}_Q)$, the image of
    $\sigma$ via the surjection $(Q^{\gp})^{\vee}\otimes_{\ZZ}\RR \to (P^{\gp})^{\vee}\otimes_{\ZZ}\RR$
    is contained $\tau$, hence the morphism of schemes 
    $\AAA_{\theta}\coloneqq \ul{\mathsf{A}}_{\theta}\colon \AAA_Q\to \AAA_P$ 
    is a toric morphism of toric varieties.
    Note that as $P\to Q$ is injective, the homomorphism of monoid algebras $\mbb{K}[P]\to \mbb{K}[Q]$
    is also injective, so the generic point of $\AAA_Q$ is mapped to the generic point of $\AAA_P$ 
    via $\AAA_{\theta}$.  Therefore, we can conclude that
    $\AAA_{\theta}$ is a dominant toric morphism of toric varieties. 

    \medskip

    \emph{Step 4}.
    By construction, $(X', \mcM')$ admits an global chart, then
    by Lemma~\ref{et-to-zar-charts}, $\mcM'$ is a Zariski log-structures on $\Et(X')$.
    By Lemma~\ref{log-regular-etale-local}, $(X', \mcM')$ is log-regular in $\loget$.
    By Lemma~\ref{log-regular-two-tops}, $(X', (\mcM')^{\zar})$ is log-regular in $\logzar$.
    Moreover, by Theorem~\ref{toroidal-to-sm-et}, $\mathsf{A}_{\theta}$ is log-smooth in $\loget$.
    As log-smoothness is stable under base changes and compositions, we see that 
    $f'\colon (X',\mcM')\to (Y',\mcN')$ is log-smooth in $\loget$.  
    Then by Lemma~\ref{log-sm-et-zar}, the log-morphism 
    $(f')^{\zar}\colon (X', (\mcM')^{\zar}) \to (Y', (\mcN')^{\zar})$ is log-smooth in $\logzar$.
\end{proof}


\begin{corollary}\label{log-sm-to-toroidal}
    Let $f\colon (X, \mcM)\to (Y,\mcN)$ be a log-smooth log-morphism of log-regular log-varieties 
    in $\loget$.  Let $X^*$ (respectively, $Y^*$) be the maximal Zariski open subset of 
    $X$ (respectively, of $Y$) on which $\mcM$ (respectively, $\mcN$) is trivial.
    Then the morphism $\ul{f}\colon X\to Y$ of the underlying varieties induces
    a morphism of embeddings
    \[ \ul{f}\colon (X^*\subset X)\to (Y^*\subset Y), \]
    which is a dominant toroidal morphism of toroidal embeddings.
\end{corollary}

\begin{proof}
    This follows immediately from Theorem~\ref{toroidal-regular} 
    and Theorem~\ref{etale-chart-log-sm-thm}.
\end{proof}


\begin{corollary}\label{log-sm-to-toroidal-zar}
    Let $f\colon (X, \mcM)\to (Y,\mcN)$ be a log-smooth log-morphism of log-regular log-varieties 
    in $\logzar$.  Let $X^*$ (respectively, $Y^*$) be the maximal Zariski open subset of 
    $X$ (respectively, of $Y$) on which $\mcM$ (respectively, $\mcN$) is trivial.
    Then the morphism $\ul{f}\colon X\to Y$ of the underlying varieties induces
    a morphism of embeddings
    \[ \ul{f}\colon (X^*\subset X)\to (Y^*\subset Y), \]
    which is a dominant toroidal morphism of \emph{strict} toroidal embeddings.
\end{corollary}

\begin{proof}
    This follows from Lemma~\ref{log-sm-et-zar}, Theorem~\ref{toroidal-regular}
    and Corollary~\ref{log-sm-to-toroidal}.
\end{proof}


\subsection{\'Etale-toric charts for toroidal morphisms}

An immediate consequence of Theorem~\ref{etale-chart-log-sm-thm} is the existence
of \'etale-toric charts for toroidal morphisms between toroidal embeddings, which is not obvious from
the definition of toroidal morphisms at the first glance; cf. \S\ref{toroidal-morphisms}.
This is taken as definition of toroidal morphisms in \cite[\S 5.3]{ACP-tropical-curves}.

\begin{theorem}\label{etale-chart-toroidal-morphism-thm}
    Let $f\colon (U_X\subset X)\to (U_Y\subset Y)$ be a dominant toroidal morphism of toroidal embeddings.
    Let $x\in X$ be a closed point and $y=f(x)$.
    Then there is a commutative diagram 
    \begin{equation}
        \begin{aligned}\label{etale-chart-toroidal-morphism}
            \xymatrix{
            (U_X\subset X)\ar[d]^f & (U_{X'}\subset X')\ar[d]^{f'}\ar[r]^-{\mu'}\ar[l]_-{\mu} & W\ar[d]^g \\
            (U_Y\subset Y) & (U_{Y'}\subset Y')\ar[r]^-{\pi'}\ar[l]_-{\pi} & V
            }\end{aligned}
    \end{equation}
    where 
    \begin{itemize}
        \item $f'$ is a dominant toroidal morphism of strict toroidal embeddings,
        \item $g$ is a dominant toric morphism of toric varieties (with tori $\mbb T_W, \mbb T_V$), 
        \item $\pi, \mu, \pi', \mu'$ are \'etale morphisms of varieties, 
        \item $\mu\colon X'\to X$, $\pi\colon Y'\to Y$ are \'etale neighbourhoods of $x,y$ respectively, and
        \item $\pi^{-1}(U_Y) = U_{Y'} = (\pi')^{-1}(\mbb T_V)$ and $\mu^{-1}(U_X) = U_{X'} = (\mu')^{-1}(\mbb T_W)$.
    \end{itemize}
\end{theorem}

\begin{proof}
    Denote by $f^{\et}\colon (X, \mcM_{U_X/X})\to (Y, \mcM_{U_Y/Y})$ the associated log-morphism
    of log-varieties in $\loget$.  By Theorem~\ref{toroidal-to-sm-et}, $f^{\et}$ is log-smooth in $\loget$.
    Then the desired \eqref{etale-chart-toroidal-morphism} is just a translation 
    of \eqref{et-chart-log-sm-diagram} to the language of toroidal embeddings and toroidal morphisms
    by applying Theorem~\ref{toroidal-to-log-reg}, Theorem~\ref{toroidal-regular} 
    and Corollary~\ref{log-sm-to-toroidal}.
\end{proof}

\begin{theorem}\label{etale-chart-toroidal-morphism-thm-zar}
    Let $f\colon (U_X\subset X)\to (U_Y\subset Y)$ be 
    a dominant toroidal morphism of \emph{strict} toroidal embeddings.
    Let $x\in X$ be a closed point and $y=f(x)$.
    Then there is a commutative diagram 
    \begin{equation*}
        \begin{aligned}
            \xymatrix{
            (U_X\subset X)\ar[d]^f & (U_{X'}\subset X')\ar[d]^{f'}\ar[r]^-{\mu'}\ar[l]_-{\mu} & W\ar[d]^g \\
            (U_Y\subset Y) & (U_{Y'}\subset Y')\ar[r]^-{\pi'}\ar[l]_-{\pi} & V
            }\end{aligned}
    \end{equation*}
    where 
    \begin{itemize}
        \item $f'$ is a dominant toroidal morphism of strict toroidal embeddings,
        \item $g$ is a dominant toric morphism of toric varieties (with tori $\mbb T_W, \mbb T_V$), 
        \item $\pi\colon Y'\to Y$ is an open immersion of varieties with $x\in \pi(Y')$,
        \item $\mu, \pi', \mu'$ are \'etale morphisms of varieties, 
        \item $\mu\colon X'\to X$ is an \'etale neighbourhood of $x$, and
        \item $\pi^{-1}(U_Y) = U_{Y'} = (\pi')^{-1}(\mbb T_V)$ and $\mu^{-1}(U_X) = U_{X'} = (\mu')^{-1}(\mbb T_W)$.
    \end{itemize}
\end{theorem}

\begin{proof}
    By Theorem~\ref{strict-chart}, we can take $\pi'$ as a strict \'etale-toric chart for $(U_Y\subset Y)$ near $y$,
    that is, $\pi\colon Y'\to Y$ is an open immersion.
    Then the rest of the proof follows from the same argument of Theorem~\ref{etale-chart-log-sm-thm}
    and Theorem~\ref{etale-chart-toroidal-morphism-thm}.
\end{proof}

\begin{remark}
    Keep the notations in Theorem~\ref{etale-chart-toroidal-morphism-thm-zar}.
    Although $\pi\colon Y'\to Y$ is an open immersion near $y\in Y$,
    the \'etale morphism $\mu\colon X'\to X$ in general can not be taken to be an open immersion;
    see the proof of \cite[Proposition 3.3]{denef2013remarkstoroidalmorphisms}.
\end{remark}


\section{Saturated base change of toroidal morphisms}\label{sat-base-change-section}

\subsection{Fibres of toroidal morphisms}

Fibres of a toroidal morphism behave like fibres of toric morphisms.

\begin{lemma}\label{toroidal-fibres}
    Let $f\colon (U_X\subset X) \to (U_Y\subset Y)$ be 
    a dominant toroidal morphism of toroidal embeddings.
    Then
    \begin{itemize}
        \item [\emph{(1)}] every irreducible component of a nonempty closed fibre of the induced morphism $f^{-1}(U_Y)\to U_Y$ is normal and intersects $U_X$,
        \item [\emph{(2)}] if $f$ is surjective, then $f(U_X) = U_Y$,
        \item [\emph{(3)}] $f^{-1}(U_Y)\to U_Y$ is flat with closed fibres of pure dimension $\dim X - \dim Y$, and
        \item [\emph{(4)}] every nonempty closed fibre of $U_X\to U_Y$ is nonsingular of pure dimension $\dim X - \dim Y$ (but it may have several irreducible components).
    \end{itemize}
\end{lemma}

\begin{proof}
    \emph{Step 1}.
    Pick arbitrary closed points $x\in X$ and $y\in Y$ such that $f(x) = y$.
    Since $f$ is dominant and toroidal, by Theorem~\ref{etale-chart-toroidal-morphism-thm}, 
    there is a commutative diagram as \eqref{etale-chart-toroidal-morphism}
    \begin{equation}
        \begin{aligned}\label{etale-neighbd-diagram}
            \xymatrix{
            X\ar[d]^f & X'\ar[l]_-{p}\ar[r]^-{q}\ar[d]^{f'} & W_{\sigma}\ar[d]^g \\
            Y & Y'\ar[l]_-{\phi}\ar[r]^-{\psi} & W_{\tau}
            }\end{aligned}
    \end{equation}
    where $p,q,\phi, \psi$ are \'etale morphisms, $p,\phi$ are \'etale neighbourhoods of $x,y$ respectively,
    and $g\colon W_{\sigma}\to W_{\tau}$ is a dominant toric morphism of toric varieties.  
    Denote by $\mbb{T}_{\sigma}$ (respectively, by $\mbb{T}_{\tau}$) the torus of $W_{\sigma}$
    (respectively, of $W_{\tau}$), then the induced homomorphism of tori $\mbb{T}_{\sigma}\to \mbb{T}_{\tau}$
    is also dominant.  By \cite[Proposition 1.1.1 (a)]{CLS:toric}, 
    $g(\mbb{T}_{\sigma})$ is a closed subtorus of $\mbb{T}_{\tau}$, 
    hence we must have $g(\mbb{T}_{\sigma}) = \mbb{T}_{\tau}$. 

    Let $x'\in X',y'\in Y'$ be closed points mapping to $x,y$ such that $f'(x') = y'$.
    Denote by $u,v$ the images of $x', y'$ in  $W_{\sigma}, W_{\tau}$ respectively.  
    Then clearly the fibre $X'_{y'}$ is a common \'etale neighbourhood of $(X_y, x)$ and $((W_{\sigma})_v, u)$. 
    To simplify notations in the rest of the proof, for every pair of closed points $x\in X, y\in Y$,
    we always use the same notations as above, such as $W_{\sigma}, W_{\tau}$, $u,v$, etc.,
    for the commutative diagram \eqref{etale-chart-toroidal-morphism} 
    in Theorem~\ref{etale-chart-toroidal-morphism-thm}. 
    This will not lead to confusion from the context. 

    \medskip

    \emph{Step 2}.
    By considering the actions of tori, 
    it is evident that all the fibres of $g$ over $\mbb{T}_{\tau}$ 
    have isomorphic open neighbourhoods in $W_{\sigma}$;
    in particular, all the fibres of $g$ over $\mbb{T}_{\tau}$ are isomorphic;
    cf. \cite[Proposition 2.1.4]{HLY02_toric_morphisms}.  Then generic flatness shows that the 
    morphism $g^{-1}(\mbb{T}_{\tau})\to \mbb{T}_{\tau}$ is flat.
    Thus, every fibre of $g$ over $\mbb{T}_{\tau}$ is of pure dimension 
    $\dim X - \dim Y$.  Moreover, as $W_{\sigma}$ is normal, we see that 
    every fibre of $g$ over $\mbb T_{\tau}$ is a normal scheme of pure dimension $\dim X - \dim Y$,
    but it may have several irreducible components.
    
    On the other hand, recall that $\mbb T_{\sigma}\to \mbb T_{\tau}$ is surjective by Step 1.  
    For the same reason as above, 
    all the fibres of $\mbb T_{\sigma}\to \mbb T_{\tau}$ are isomorphic.  Then generic smoothness 
    implies that every fibre of $\mbb{T}_{\sigma}\to \mbb{T}_\tau$ is nonsingular
    of pure dimension $\dim X - \dim Y$.
    Let $v'\in \mbb{T}_{\tau}$ be a closed point, 
    and let $G'$ be an irreducible component of $(W_{\sigma})_{v'}$ 
    with generic point $\eta'$.
    As $v'\in \mbb{T}_{\tau}$, $(W_{\sigma})_{v'}$ is isomorphic to a general fibre $(W_{\sigma})_{v''}$ of $g$,
    where $v''\in \mbb{T}_{\tau}$ is a general closed point.  
    Let $t\in \mbb T_{\tau}$ such that $t.v' = v''$.  As $\mbb T_{\sigma}\to \mbb T_{\tau}$
    is surjective, there is an $s\in \mbb T_{\sigma}$ such that $g(s) = t$.  As $g$ is equivariant, 
    there is an irreducible component $G''$ of $(W_{\sigma})_{v''}$ such that $s(G') = G''$.
    Since $\dim (W_{\sigma}\setminus \mbb{T}_{\sigma}) < \dim W_{\sigma}$,
    generic flatness shows that the generic point $\eta''$ of $G''$ is contained in $\mbb{T}_{\sigma}$.
    Since $s^{-1}. \eta'' = \eta'$, and since $\mbb T_{\sigma}$ is an orbit
    of the torus action, we see that $\eta'\in G'$ is also contained in $\mbb{T}_{\sigma}$,
    i.e., $\mbb{T}_{\sigma}$ intersects $G'$. 

    \medskip

    \emph{Step 3}.
    Let $y\in U_Y$ be a closed point such that
    the fibre $X_y$ of $f$ over $y$ is nonempty.
    Let $x\in X$ be a closed point such that $y=f(x)$.
    Then we can form the diagram \eqref{etale-neighbd-diagram} in Step 1;
    in particular, $(X'_{y'},x')$ is a common \'etale neighbourhood of $(X_y,x)$ and $((W_{\sigma})_v, u)$,
    where $x'\in X',y'\in Y', u\in W_{\sigma}, v\in W_{\tau} $ are closed points as in Step 1.

    By Step 2, $(W_{\sigma})_v$ is a normal scheme of pure dimension $\dim X - \dim Y$.
    Thus, we see that $X_y$ is also a normal scheme of pure dimension $\dim X - \dim Y$ near $x$.
    Varying $x$ in the fibre $X_y$, we can conclude that $X_y$ is normal and every irreducible component of
    $X_y$ has dimension $\dim X - \dim Y$.
    
    Let $F$ be an irreducible component of $X_y$, which is a normal subvariety of $X$.
    Let $x\in X$, which is a general closed point of $F$,
    then $x$ is a nonsingular point of $F$.
    For the points $x,y$, form the diagram \eqref{etale-neighbd-diagram}.
    Then $u\in W_{\sigma}$ is also a nonsingular point of the fibre $(W_{\sigma})_v$.
    Let $G$ be an irreducible component of $(W_{\sigma})_v$ containing $u$.
    Then there is a normal irreducible component $F'$ of $X'_{y'}$ such that
    $(F', x')$ is a common \'etale neighbourhood of $(F, x)$ and $(G, u)$.
    Denote by 
    \[ \ol{p}\colon F'\to F \,\text{ and }\, \ol{q}\colon F'\to G \]
    the corresponding \'etale morphisms induced by $p,q$.
    By Step 2, $G\cap \mbb{T}_{\sigma}$ is a nonempty open dense subset of $G$.
    Since $G$ is irreducible, $G\cap \mbb{T}_{\sigma}\cap \ol{q}(F')$ is also a nonempty open dense subset of $G$, hence
    $V\coloneqq (\ol{q})^{-1}(\mbb{T}_{\sigma})$ is a nonempty open subset of $F'$.
    Thus, $U_X\cap F = \ol{p}(V)$ is a nonempty open dense subset of $F$, so (1) holds.
    The result (2) follows immediately from (1).
    Moreover, by \cite[(3), page 46]{CA-Mat}, 
    the morphism $f^{-1}(U_Y)\to U_Y$ is flat.  This shows (3).
    As regularity of schemes is an \'etale-local property, 
    we can conclude that (4) holds.
\end{proof}


\subsection{Saturated base change in $\loget$}

The following result is a general version of Theorem~\ref{sat-base-change-intro} 
for toroidal embeddings whose toroidal boundaries may have self-intersections.

\begin{theorem}\label{sat-base-change}
	Let $(U_X\subset X)$, $(U_Y\subset Y)$ and $(U_Z\subset Z)$ be toroidal embeddings
    with the corresponding compactifying log-structures $\mc{M}_{U_X/X}$, 
    $\mc{M}_{U_Y/Y}$ and $\mc{M}_{U_Z/Z}$ respectively.  Let 
    \[ f\colon (U_Y\subset Y)\to (U_Z\subset Z)\,\text{ and }\, g\colon (U_X\subset X)\to (U_Z\subset Z) \]
	be morphisms of embeddings.  Assume that 
	$g$ is a dominant toroidal morphism; equivalently,
	$(X, \mc{M}_{U_X/X}) \to (Z, \mc{M}_{U_Z/Z})$ is a log-smooth log-morphism in $\loget$.
    
	Denote by $W$ the normalisation of an irreducible
	component of $Y\times_Z X$ that dominates $Y$, and by $p\colon W\to Y$ and $q\colon W\to X$
	the induced projection morphisms respectively.
    \[\xymatrix{
	(U_W\subset W)\ar[r]^-q\ar[d]^p & (U_X\subset X)\ar[d]^g \\
	(U_Y\subset Y)\ar[r]^-f & (U_Z\subset Z)
	}\]
	Let 
    \[ U_W \coloneqq p^{-1}(U_Y)\cap q^{-1}(U_X). \]
    Then $U_W$ is a nonempty open subset of $W$, and $(U_W\subset W)$ is also a toroidal embedding.
    Moreover, the induced morphism of embeddings
    \[ p\colon (U_W\subset W)\to (U_Y\subset Y) \]
    is a dominant toroidal morphism of toroidal embeddings.
\end{theorem}

\begin{proof}
    To avoid confusion about notation of morphisms and log-morphisms 
    in the categories of schemes and log-schemes, we denote by 
    \[ f^{\et}\colon (Y, \mcM_{U_Y/Y})\to (Z, \mcM_{U_Z/Z})\,\text{ and }\,
    g^{\et}\colon (X, \mcM_{U_X/X}) \to (Z, \mcM_{U_Z/Z}) \]
    the log-morphism of the log-regular log-varieties in $\loget$, 
    then the morphisms of underlying varieties
    associated to $f^{\et}, g^{\et}$ are just denoted by $f,g$.

    \medskip
    
    \emph{Step 1}. The log-varieties $(X, \mc{M}_{U_X/X})$, $(Y, \mc{M}_{U_Y/Y})$ and $(Z, \mc{M}_{U_Z/Z})$
    are fs as \'etale-locally they admit fs charts by Theorem~\ref{general-chart}.  Denote by 
    \[ \mathbf{\Omega} \coloneqq (\Omega, \mc{M}_{\Omega}), \,\,
    \mathbf{\Omega}^{\inte} \coloneqq (\Omega^{\inte}, \mc{M}_{\Omega^{\inte}})\, \text{ and }\,
    \mathbf{\Omega}^{\fs} \coloneqq (\Omega^{\fs}, \mc{M}_{\Omega^{\fs}}) \] 
    the fibre product of $(Y, \mc{M}_{U_Y/Y})$ and $(X, \mc{M}_{U_X/X})$
    along $(Z, \mc{M}_{U_Z/Z})$ in the category of log-schemes, fine log-schemes and fs log-schemes respectively.
    In particular, the underlying scheme $\Omega$ of $\mathbf{\Omega}$ is isomorphic to $Y\times_Z X$;
    see \cite[III.2.1.2]{Ogus-log-geo}.
    
    By \cite[III.2.1.5]{Ogus-log-geo}, the morphism of underlying schemes
    $\Omega^{\inte} \to \Omega$ (respectively, $\Omega^{\fs} \to \Omega^{\inte}$)
    is a closed immersion (respectively, a finite surjective morphism).
    Moreover, by \cite[IV.3.1.11]{Ogus-log-geo}, $\mathbf{\Omega}^{\fs}\to \mathbf{\Omega}^{\inte}$
    and $\mathbf{\Omega}^{\inte}\to \mathbf{\Omega}$ are log-\'etale log-morphisms of log-schemes in $\loget$.
    As log-smooth log-morphisms are stable under compositions and base changes
    in the category of log-schemes (see \cite[IV.3.1.2]{Ogus-log-geo}), the induced log-morphism
    \[ \mathbf{\Omega}^{\fs}\to (Y, \mc{M}_{U_Y/Y}) \] 
    is log-smooth in $\loget$.  
    Thus, $\mathbf{\Omega}^{\fs}$ is also log-regular in $\loget$ 
    by \cite[IV.3.5.3]{Ogus-log-geo}.
    In particular, the log-structure $\mc{M}_{\Omega^{\fs}}$ is compactifying, that is,
    the canonical homomorphism $\mc{M}_{\Omega^{\fs, *}/\Omega^{\fs}}\to \mc{M}_{\Omega^{\fs}}$
    is an isomorphism (see \cite[III.1.11.12]{Ogus-log-geo}), 
    where $\Omega^{\fs,*}$ is the maximal Zariski open subset of $\Omega^{\fs}$
    on which the log-structure $\mc{M}_{\Omega^{\fs}}$ is trivial.
    By Lemma~\ref{log-regular-normal}, $\Omega^{\fs}$
    has normal irreducible components, hence $\Omega^{\fs}$ is a disjoint union 
    of normal varieties.  Moreover, by Theorem~\ref{toroidal-regular}, every irreducible component
    of $\Omega^{\fs}$ is a toroidal variety. 

    Let $U$ be the fibre product $U_Y\times_{U_Z} U_X$, and let
    $(U, \mc{M}_U)$ be the log-scheme endowed with the trivial log-structure.
    We claim that $U$ is nonempty.
    Denote by $\eta_Y$ the generic point of $Y$.
    As $f$ is a morphism of embeddings, the image $\eta_f \coloneqq f(\eta_Y)$ is contained in $U_Z$.
    By assumption, $Y\times_Z X$ is nonempty, then $\eta_f$ is contained in the image
    of $g\colon X\to Z$.  By Lemma~\ref{toroidal-fibres}, we see that $U$ is nonempty.
    
    By Lemma~\ref{trivial-fibre-product}, 
    $(U, \mc{M}_U)$ is the restriction of $\mathbf{\Omega}$ to $\Et (U)$.
    By the constructions in \cite[III.2.1.4, III.2.1.5]{Ogus-log-geo}, 
    there is a log-morphism $(U, \mc{M}_U)\to \mathbf{\Omega}^{\fs}$ such that
    \begin{itemize}
        \item underlying morphism of schemes of $(U, \mc{M}_U)\to \mathbf{\Omega}^{\fs}$ is an open immersion,
        \item $(U, \mc{M}_U)$ is also the restriction of $\mathbf{\Omega}^{\fs}$ to $\Et (U)$, and
        \item the composite log-morphism $(U, \mcM_U)\to \mathbf{\Omega}^{\fs}\to \mathbf{\Omega}$ is equal to the canonical log-morphism $(U, \mcM_U)\to \mathbf{\Omega}$.
    \end{itemize} 
    Thus, we can view $U$ as an open subscheme of $\Omega^{\fs,*}$.
    Moreover, by \cite[III.1.11.6]{Ogus-log-geo}, the scheme $U$ is nonsingular in the usual sense.

    \medskip

    \emph{Step 2}.
    Let $W'$ be an irreducible component of $Y\times_Z X$ that dominates $Y$.
    Recall that $\eta_Y$ is the generic point of $Y$, and $\eta_f \coloneqq f(\eta_Y) \in U_Z$.
    Then the generic fibre of $W'\to Y$ is an irreducible component of 
    $(\spec k(\eta_Y))\times_Z X$.  By Lemma~\ref{toroidal-fibres}, 
    every irreducible component of $(\spec k(\eta_f))\times_Z X$ intersects the open subset $U_X$ of $X$,
    hence the pullback of $U_X$ to $W'$ is a dense open subset of $W'$.
    Let $W$ be the normalisation of $W'$ with projection morphisms $p\colon W\to Y$
    and $q\colon W\to X$, then $p^{-1}(U_Y)\cap q^{-1}(U_X)$ is a nonempty dense open subset of $W$. 
    
    Since $U$ is a nonsingular scheme, there is a unique irreducible component $U_W$ of $U$
    admitting an open immersion $U_W\hookrightarrow W'$.
    Then the normalisation $W\to W'$ is an isomorphism over $U_W$.
    Thus, viewed as an open subset of $W$, we get
    \[ U_W = p^{-1}(U_Y)\cap q^{-1}(U_X). \]
    Let $W''$ be the irreducible component of $\Omega^{\fs}$ that contains $U_W$.
    By \cite[III.2.1.5]{Ogus-log-geo}, there is an induced 
    finite surjective morphism $W''\to W'$.
    As $U_W$ embeds into both $W'$ and $W''$,
    the morphism $W''\to W'$ is also birational.  
    Since $\Omega^{\fs}$ is normal, the finite, birational morphism
    $W''\to W'$ is the normalisation of $W'$, that is, $W''\cong W$.

    Denote by $W^*\subseteq W$ the maximal open subset on which $\mcM_{\Omega^{\fs}}|_{\Et(W)}$ 
    is trivial.
    By \cite[12.2.8]{gabber2018foundationsringtheory}, $W^*$ is mapped into $U_Y, U_X$ 
    via $p\colon W\to Y, q\colon W\to X$.  Then 
    \[ U_W\subseteq W^* \subseteq p^{-1}(U_Y)\cap q^{-1}(U_X) = U_W. \]
    Thus, we see that $W^* = U_W$.  Hence by Theorem~\ref{toroidal-regular}, 
    $(U_W\subset W)$ is a toroidal embedding.
    Moreover, since $\mathbf{\Omega}^{\fs} \to (Y, \mc{M}_{U_Y/Y})$ is log-smooth in $\loget$, 
    by Corollary~\ref{log-sm-to-toroidal},
    we can conclude that $p\colon (U_W\subset W)\to (U_Y\subset Y)$ is a dominant toroidal morphism.
\end{proof}

\subsection{Saturated base change in $\logzar$}

Now we prove Theorem~\ref{sat-base-change-intro}.

\begin{theorem}[=\protect{Theorem~\ref{sat-base-change-intro}}]\label{sat-base-change-final}
	Let $(U_X\subset X)$, $(U_Y\subset Y)$ and $(U_Z\subset Z)$ be \emph{strict} toroidal embeddings
    with the corresponding Zariski compactifying log-structures $\mc{M}_{U_X/X}^{\zar}$, 
    $\mc{M}_{U_Y/Y}^{\zar}$ and $\mc{M}_{U_Z/Z}^{\zar}$ respectively.  Let 
    \[ f\colon (U_Y\subset Y)\to (U_Z\subset Z)\,\text{ and }\, g\colon (U_X\subset X)\to (U_Z\subset Z) \]
	be morphisms of embeddings.  Assume that 
	$g$ is a dominant toroidal morphism; equivalently,
	$(X, \mc{M}_{U_X/X}^{\zar}) \to (Z, \mc{M}_{U_Z/Z}^{\zar})$ is a log-smooth log-morphism in $\logzar$.
    
	Denote by $W$ the normalisation of an irreducible
	component of $Y\times_Z X$ that dominates $Y$, and by $p\colon W\to Y$ and $q\colon W\to X$
	the induced projection morphisms respectively.
    \[\xymatrix{
	(U_W\subset W)\ar[r]^-q\ar[d]^p & (U_X\subset X)\ar[d]^g \\
	(U_Y\subset Y)\ar[r]^-f & (U_Z\subset Z)
	}\]
	Let 
    \[ U_W \coloneqq p^{-1}(U_Y)\cap q^{-1}(U_X). \]
    Then $U_W$ is a nonempty open subset of $W$, 
    and $(U_W\subset W)$ is also a \emph{strict} toroidal embedding.
    Moreover, the induced morphism of embeddings
    \[ p\colon (U_W\subset W)\to (U_Y\subset Y) \]
    is a dominant toroidal morphism of toroidal embeddings.
\end{theorem}

\begin{proof}
    The proof works exactly the same as Theorem~\ref{sat-base-change},
    so we only give a sketch to indicate the citations of related results 
    for Zariski sites.  The log-varieties $(X, \mc{M}_{U_X/X}^{\zar})$, 
    $(Y, \mc{M}_{U_Y/Y}^{\zar})$ and $(Z, \mc{M}_{U_Z/Z}^{\zar})$ are 
    fs as Zariski-locally they admit fs charts by Theorem~\ref{strict-chart}.
    Denote by 
    \[ f^{\zar}\colon (Y, \mcM_{U_Y/Y}^{\zar})\to (Z, \mcM_{U_Z/Z}^{\zar})\,\text{ and }\,
    g^{\zar}\colon (X, \mcM_{U_X/X}^{\zar}) \to (Z, \mcM_{U_Z/Z}^{\zar}) \]
    the corresponding log-morphism of log-regular log-varieties in $\logzar$.
    Denote by 
    \[ \mathbf{\Theta} \coloneqq (\Theta, \mc{M}_{\Theta}), \,\,
    \mathbf{\Theta}^{\inte} \coloneqq (\Theta^{\inte}, \mc{M}_{\Theta^{\inte}}) \text{ and }\,
    \mathbf{\Theta}^{\fs} \coloneqq (\Theta^{\fs}, \mc{M}_{\Theta^{\fs}}) \] 
    the fibre product of $(Y, \mc{M}_{U_Y/Y}^{\zar})$ and $(X, \mc{M}_{U_X/X}^{\zar})$
    along $(Z, \mc{M}_{U_Z/Z}^{\zar})$ in the category of log-schemes, 
    fine log-schemes, and fs log-schemes in $\logzar$ respectively.
    Note that the construction of \cite[III.2.1.2]{Ogus-log-geo} also works for the Zariski site.
    Then $\Theta$ is isomorphic to $Y\times_Z X$, and $\mcM_{\Theta}$ is a coherent log-structure
    on $\zar(\Theta)$.
    
    By \cite[12.2.36 (ii)]{gabber2018foundationsringtheory}, the morphism of schemes
    $\Theta^{\inte} \to \Theta$ (respectively, $\Theta^{\fs} \to \Theta^{\inte}$)
    is a closed immersion (respectively, a finite morphism).
    By \cite[12.3.34]{gabber2018foundationsringtheory}, and by
    the proof of \cite[12.2.35]{gabber2018foundationsringtheory},
    \[ \mathbf{\Theta}^{\fs}\to \mathbf{\Theta}^{\inte}\, \text{ and } \,\mathbf{\Theta}^{\inte}\to \mathbf{\Theta}\]
    are log-\'etale in $\logzar$.
    As log-smoothness in $\logzar$ is stable under compositions and base changes
    (see \cite[12.3.24]{gabber2018foundationsringtheory}),
    the induced log-morphism of log-schemes 
    \[ \mathbf{\Theta}^{\fs}\to (Y, \mc{M}_{U_Y/Y}^{\zar}) \] 
    is log-smooth in $\logzar$.  
    Then by \cite[(8.2)]{Kato-toric} or \cite[12.5.44]{gabber2018foundationsringtheory},
    the log-scheme $\mathbf{\Theta}^{\fs}$ is log-regular in $\logzar$.
    By Lemma~\ref{log-regular-normal}, $\Theta^{\fs}$
    has normal irreducible components, hence $\Theta^{\fs}$ is a disjoint union 
    of normal varieties.  
    Then by Theorem~\ref{toroidal-regular}, every irreducible component
    of $\Theta^{\fs}$ is a strict toroidal variety. 

    Let $U$ be the fibre product $U_Y\times_{U_Z} U_X$, and let
    $(U, \mc{M}_U^{\zar})$ be the log-scheme endowed with the trivial log-structure.
    By the same argument in Theorem~\ref{sat-base-change}, $U$ is nonempty.
    By Lemma~\ref{trivial-fibre-product}, $(U, \mc{M}_U^{\zar})$ is the restriction of $\mathbf{\Theta}$ to $\zar (U)$.
    Denote by $\Theta^{\fs,*}$ the maximal open subset of $\Theta^{\fs}$ on which $\mcM_{\Theta^{\fs}}$
    is trivial.
    By the construction in \cite[12.2.35]{gabber2018foundationsringtheory},
    $U$ is an open subscheme of $\Theta^{\fs}$ such that $\mc{M}_U^{\zar}$ is the restriction 
    of $\mc{M}_{\Theta^{\fs}}$ to $\zar(U)$, whence $U\subseteq \Theta^{\fs,*}$.  Moreover, 
    by \cite[12.3.27]{gabber2018foundationsringtheory}, 
    the scheme $U$ is nonsingular in the usual sense.

    Let $W'$ be an irreducible component of $Y\times_Z X$ that dominates $Y$.
    Let $W$ be the normalisation of $W'$ with projection morphisms $p\colon W\to Y, q\colon W\to X$.
    By the same argument in Theorem~\ref{sat-base-change}, 
    $W$ is an irreducible component of $\Theta^{\fs}$, and
    there is a unique irreducible component $U_W$ of $U$ contained in $W$, which is equal to $p^{-1}(U_Y)\cap q^{-1}(U_X)$.
    
    Denote by $\mc{M}_W^{\zar}$ the restriction of the log-structure $\mc{M}_{\Theta^{\fs}}$ to $\zar(W)$.
    The same argument in Theorem~\ref{sat-base-change} shows that 
    $U_W$ is the maximal open subset of $W$ on which $\mc{M}_W^{\zar}$ is trivial.
    By Theorem~\ref{toroidal-regular},
    $(U_W\subset W)$ is a strict toroidal embedding.
    By Corollary~\ref{log-sm-to-toroidal-zar},
    we see that $p\colon (U_W\subset W)\to (U_Y\subset Y)$ is a dominant toroidal morphism.
\end{proof}


\medskip

\printbibliography

\vspace{1em}
 
\noindent\small{Santai Qu} 

\noindent\small{\textsc{Institute of Geometry and Physics, University of Science and Technology of China, Hefei, Anhui Province, China} }

\noindent\small{Email: \texttt{santaiqu@ustc.edu.cn}}

\end{document}